\documentclass[12pt]{amsart}
\usepackage{a4wide}

\usepackage{times}
\usepackage{mathtools, amssymb}
\usepackage{graphicx,xspace}
\usepackage{epsfig}
\usepackage{dsfont}
\usepackage[usenames,dvipsnames]{xcolor}
\usepackage{tikz}
\usepackage[T1]{fontenc}
\usepackage[utf8]{inputenc}
\usepackage{array,multirow} 
\usepackage{bbold}
\usepackage{bm} 
\usepackage{csquotes}
\usepackage{cancel}

\usepackage{hyperref,pifont}
\usepackage{todonotes}
\usepackage{dsfont}

\usepackage[capitalize]{cleveref}

\theoremstyle{plain}   

\newtheorem{lemma}{Lemma}[section]
\newtheorem{theorem}[lemma]{Theorem}
\newtheorem{corollary}[lemma]{Corollary}
\newtheorem{proposition}[lemma]{Proposition}
\newtheorem{definition}[lemma]{Definition}

\theoremstyle{remark}
\newtheorem{remark}[lemma]{Remark}
\newtheorem{example}[lemma]{Example}

\def\eps{\varepsilon}
\def\SSX{\mathcal{S}_{\mathrm{X}}}
\def\SSC{\mathcal{S}_{\mathrm{C}}}
\def\SX{S_{\mathrm{X}}}
\def\SC{S_{\mathrm{C}}}
\def\KK{\mathcal{K}}
\def\CRT{\mathcal{T}_\infty}
\def\Gn{\bm{G}^{(n)}}
\def\Tn{\bm{T}^{(n)}}
\def\Gm{\bm{G}^{(m)}}

\def\mD{\mathcal{D}}
\def\mDSX{\mathcal{D}_{\SSX}}
\def\mDSC{\mathcal{D}_{\SSC}}
\def\mDK{\mathcal{D}_{\KK}}
\def\mDdc{\overline{\mathcal{D}}}
\def\mDSXdc{\overline{\mathcal{D}}_{\SSX}}
\def\mDSCdc{\overline{\mathcal{D}}_{\SSC}}
\def\mDKdc{\overline{\mathcal{D}}_{\KK}}
\def\mDab{\mathcal{D}_{a}^{b}}

\def\mDXC{\mathcal{D}_{\SSX}^{\SSC}}

\def\mDaK{\mathcal{D}_{a}^{\KK}}
\def\mDaX{\mathcal{D}_{a}^{\SSX}}
\def\mDaC{\mathcal{D}_{a}^{\SSC}}

\def\mDKb{\mathcal{D}_{\KK}^{b}}
\def\mDXb{\mathcal{D}_{\SSX}^{b}}
\def\mDCb{\mathcal{D}_{\SSC}^{b}}
\def\mDap{\mathcal{D}_{a}^{\bullet}}
\def\mDpb{\mathcal{D}_{\bullet}^{b}}
\def\J{\mathcal{J}}
\def\I{\mathcal{I}}

\def\mDdcaK{\overline{\mathcal{D}}_{a}^{\KK}}
\def\mDdcaX{\overline{\mathcal{D}}_{a}^{\SSX}}

\def\mDdcKb{\overline{\mathcal{D}}_{\KK}^{b}}
\def\mDdcXb{\overline{\mathcal{D}}_{\SSX}^{b}}

\def\mDdcap{\overline{\mathcal{D}}_{a}^{\bullet}}
\def\mDdcpb{\overline{\mathcal{D}}_{\bullet}^{b}}
\def\mDdcab{\overline{\mathcal{D}}_{a}^{b}}

\def\Ddc{\overline{D}}
\def\DSXdc{\overline{D}_{\SSX}}
\def\DSCdc{\overline{D}_{\SSC}}
\def\DKdc{\overline{D}_{\KK}}
\def\DdcSX{\overline{D}_{\SSX}}
\def\DdcSC{\overline{D}_{\SSC}}
\def\DdcK{\overline{D}_{\KK}}

\def\DSX{D_{\SSX}}
\def\DSC{D_{\SSC}}
\def\DK{D_{\KK}}
\def\Dab{{D}_{a}^{b}}
\def\DKK{{D}_{\KK}^{\KK}}
\def\DKX{{D}_{\KK}^{\SSX}}
\def\DKC{{D}_{\KK}^{\SSC}}
\def\DXK{{D}_{\SSX}^{\KK}}
\def\DXX{{D}_{\SSX}^{\SSX}}
\def\DXC{{D}_{\SSX}^{\SSC}}
\def\DCK{{D}_{\SSC}^{\KK}}
\def\DCX{{D}_{\SSC}^{\SSX}}
\def\DCC{{D}_{\SSC}^{\SSC}}

\def\DdcKK{\overline{D}_{\KK}^{\KK}}
\def\DdcKX{\overline{D}_{\KK}^{\SSX}}
\def\DdcKC{\overline{D}_{\KK}^{\SSC}}
\def\DdcXK{\overline{D}_{\SSX}^{\KK}}
\def\DdcXX{\overline{D}_{\SSX}^{\SSX}}
\def\DdcXC{\overline{D}_{\SSX}^{\SSC}}
\def\DdcCK{\overline{D}_{\SSC}^{\KK}}
\def\DdcCX{\overline{D}_{\SSC}^{\SSX}}
\def\DdcCC{\overline{D}_{\SSC}^{\SSC}}

\def\Ddcab{\overline{D}_{a}^{b}}

\def\mDt{\mathcal{D}_{t_0}}
\def\Dt{D_{t_0}}

\def\mESX{\mathcal{E}_{\SSX}}
\def\mESC{\mathcal{E}_{\SSC}}
\def\mEK{\mathcal{E}_{\KK}}
\def\mL{\mathcal{L}}
\def\mZ{\mathcal{Z}}

\def\ESX{E_{\SSX}}
\def\ESC{E_{\SSC}}
\def\EK{{E}_{\KK}}
\def\EXX{{E}_{\SSX}^{\SSX}}

\def\EXp{{E}_{\SSX}^{\bullet}}
\def\EpX{{E}_{\bullet}^{\SSX}}

\def\Lp{L^{\bullet}}

\def\mESX{\mathcal E_{\SSX}}
\def\mESC{\mathcal E_{\SSC}}
\def\mEK{{\mathcal E}_{\KK}}
\def\mEXX{{\mathcal E}_{\SSX}^{\SSX}}

\def\mEXp{{\mathcal E}_{\SSX}^{\bullet}}
\def\mEpX{{\mathcal E}_{\bullet}^{\SSX}}

\def\mLp{{\mathcal L}^{\bullet}}

\def\mEt{\mathcal{E}_{t_0}}
\def\Et{E_{t_0}}

\def\rhotl{\rho_{3\ell}}

\usepackage{navigator}
\embeddedfile{Companion Maple worksheet -- version pdf}{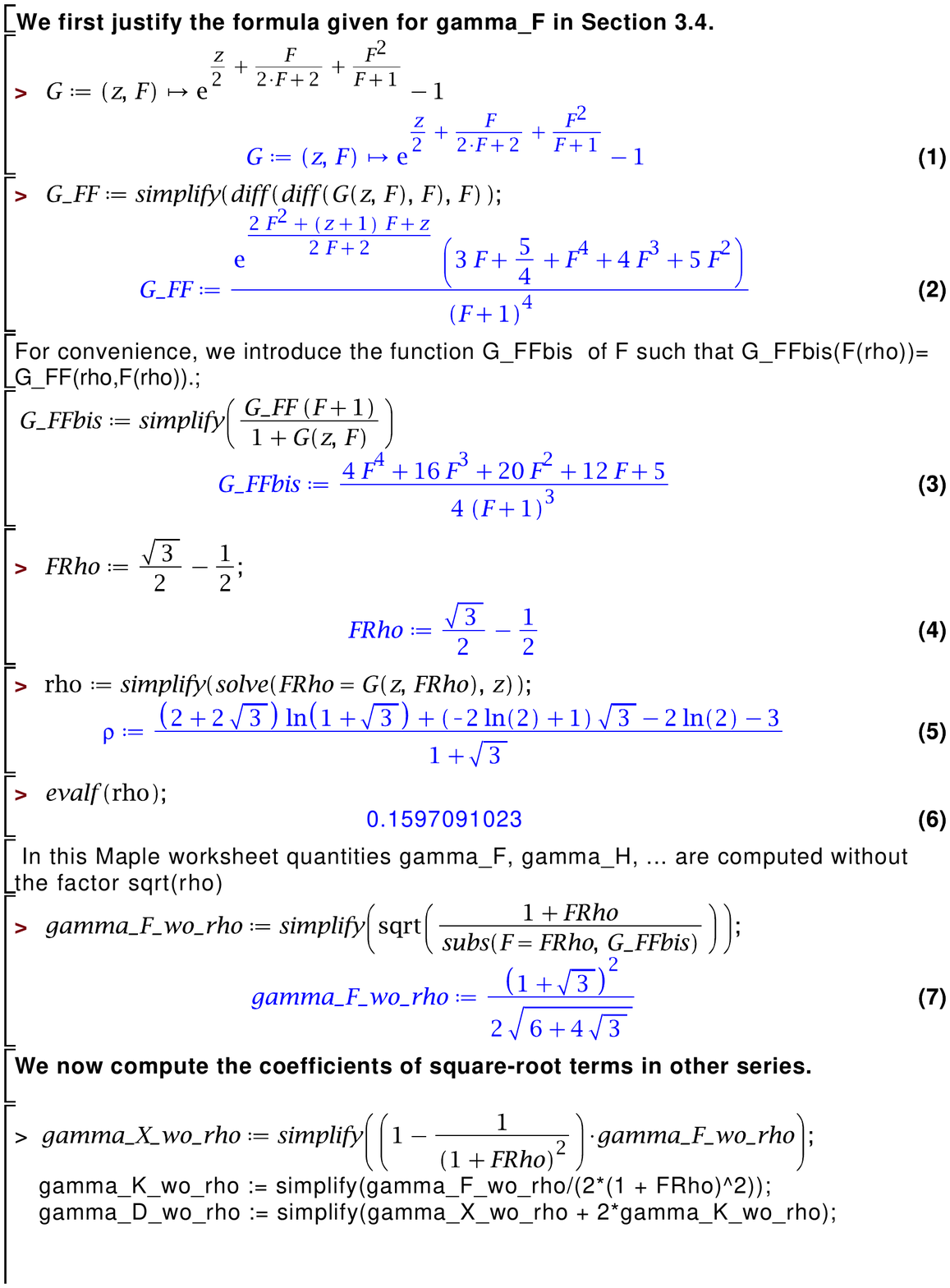}
\embeddedfile{Companion Maple worksheet -- version mw}{SplitDecomposition_worksheet.mw}

\def\MGP{\mathbb M}

\def\DGP{d_{\text{GP}}}

\newcommand\Set{\mathrm{Set}}
\newcommand\Supp{\mathrm{Supp}}
\newcommand\X{\mathcal{X}}
\renewcommand\O{\mathcal{O}}

\newcommand\bX{\bm{\mathcal{X}}}

\newcommand\jump{\mathrm{jp}}

\def\E{\mathfrak E}
\def\Ind{\mathbb{1}}

\def\bbR{\mathbb{R}}

\DeclareMathOperator{\Gr}{Gr}

\DeclareMathOperator{\Arg}{Arg}

\newcommand{\Exc}{ \scalebox{1.1}{$\mathfrak{e}$}}

\title[Scaling limits and split decomposition]{Scaling limit of graph classes through split decomposition}

\author[F. Bassino]{Frédérique Bassino}
       \address[FB]{Université Sorbonne Paris Nord, LIPN, CNRS UMR 7030, F-93430 Villetaneuse, France}
       \email{bassino@lipn.univ-paris13.fr}

 \author[M. Bouvel]{Mathilde Bouvel}
   \address[MB]{Université de Lorraine, CNRS, Inria, LORIA, F-54000 Nancy, France}
   \email{mathilde.bouvel@loria.fr}
   
 \author[V. Féray]{Valentin Féray}
  \address[VF]{Université de Lorraine, CNRS, IECL, F-54000 Nancy, France}
  \email{valentin.feray@univ-lorraine.fr}

 \author[L. Gerin]{Lucas Gerin}
       \address[LG]{CMAP, \'Ecole Polytechnique, CNRS, Route de Saclay, 91128 Palaiseau Cedex, France}
       \email{gerin@cmap.polytechnique.fr}

 \author[A. Pierrot]{Adeline Pierrot}
 \address[AP]{LISN, Université Paris-Saclay, Bat. 650 Ada Lovelace, 91405 Orsay Cedex, France}
       \email{adeline.pierrot@lri.fr}

\keywords{Brownian Continuum Random Tree, Distance hereditary graphs, Split decomposition, Analytic combinatorics, Graph scaling limits}

\makeatletter
\@namedef{subjclassname@2020}{%
  \textup{2020} Mathematics Subject Classification}
\makeatother

\subjclass[2020]{60C05,05C80,05A16}

\begin{document}

\begin{abstract}
We prove that Aldous' Brownian CRT is the scaling limit, with respect to the Gromov--Prokhorov topology, of uniform random graphs in each of the three following families of graphs: distance-hereditary graphs, $2$-connected distance-hereditary graphs and $3$-leaf power graphs. 
Our approach is based on the split decomposition and on analytic combinatorics.\end{abstract}

\maketitle

\section{Introduction}

In the present article we obtain scaling limit results for large graphs taken uniformly at random in the class of distance-hereditary graphs (DH graphs for short) and in two interesting subclasses: $2$-connected distance-hereditary graphs and $3$-leaf power graphs.
In all cases, the limit is the celebrated Brownian continuum random tree (Brownian CRT for short).
We start by giving some background on these graph classes.

\subsection{Distance-hereditary graphs and interesting subclasses}
DH graphs are the connected graphs for which the distances in any connected induced subgraph are the same as in the original graph. 
They enjoy many other characterizations, for instance by avoidance of induced subgraphs.
Among other properties, they form a subclass of perfect graphs and have clique-width at most three. They have been widely studied 
in the algorithmic literature: in particular,
it has been proved that many NP-hard problems can be solved in polynomial time for DH graphs (see \emph{e.g.} \cite{CogisThierry}); additionally,
DH graphs can be recognized efficiently, both in the static and dynamic framework 
(see \cite{SplitTrees}, and references therein).

To establish such algorithmic properties, a key feature of distance-hereditary graphs
is that they are nicely decomposable for the so-called \emph{split decomposition}.
More recently, this split decomposition has also been used
to give precise enumerative results and sampling algorithms on the class of distance-hereditary graphs
and some of its subclasses \cite{ChauveFusyLumbroso,BL18}. 
The analysis of distance-hereditary graphs (and subclasses) via the symbolic method, as done by Chauve,
Fusy and Lumbruso \cite{ChauveFusyLumbroso} (and reviewed in \cref{Sec:CombiDH} below) is a starting point
for the present paper. 
More precisely, in our work, we aim at illustrating
the usefulness of the split decomposition (combined with symbolic and analytic combinatorics) to study large random DH graphs.

Let us comment on the choice of graph classes considered in this article, in addition to the DH graphs already discussed.
The class of $3$-leaf power graphs has been studied in \cite{SplitTrees} (resp.~
\cite{ChauveFusyLumbroso}) to illustrate the versatility of algorithmic (resp.~enumerative)
results obtained through the split decomposition.
It is therefore natural for us to use it to illustrate as well the versatility of
the probabilistic approach through the split decomposition.
Since $3$-leaf power graphs are defined via trees (see \cref{def:3Leaf}),
their convergence to an infinite tree might seem expected.
On the contrary, conditioning random DH graphs 
to be $2$-connected makes them further from being trees.
Our result indicates that, nevertheless, at the level of scaling limits,
$2$-connected DH graphs are tree-like and converge to the Brownian CRT.

Another motivation for considering $3$-leaf power graphs and $2$-connected DH graphs
is that, unlike unconstrained DH graphs, they do not form
what is called a {\em block-stable} class of graphs.
Indeed, such block-stable graph classes
have already been studied in the discrete probability literature 
\cite{DrmotaEtAlSubcritical,DrmotaNoy}.
In particular, a scaling limit result for random graphs in such classes 
(under an additional subcriticality hypothesis)
is provided in \cite{SubcriticalClasses}, covering the case of unconstrained DH graphs.
It is therefore important to show that our approach through split decomposition
works also for classes which are not block-stable; 
and an obvious way to obtain a class of graphs which is not block-stable is to impose the constraint of being $2$-connected.

\subsection{The results}
A standard question in the theories of random trees, random maps and more recently
random graphs is to look for limits of random graph sequences, for various topologies. 
To this end, we consider graphs as discrete metric measure spaces.
A metric measure space (mm-space for short) is a triple $(X,d,\mu)$,
where $(X,d)$ is a complete and separable metric space 
and $\mu$ a probability measure on $X$. 
A finite connected graph can be seen as a mm-space, 
where $X$ is the vertex set of the graph,
$d$ the graph distance, and $\mu$ the uniform distribution on $X$.
In this setting scaling limits of random graphs correspond to the convergence of random mm-spaces, 
after renormalization of the distances.

For metric measure spaces there are two classical topologies used in the literature, the Gromov--Prohorov (GP) topology 
and the stronger Gromov--Hausdorff--Prohorov (GHP)  topology. 
Our result holds with respect to the GP topology (see \cref{Sec:CritereGP} for the definition). 
We believe that it could be extended to the GHP topology, using a criterion provided by Athreya-L\"{o}hr-Winter \cite{athreya2016gap}. 
However, this would likely require tools and methods very different from those of the present paper, and is therefore beyond its scope.

For $n\geq 1$ denote by $\mathcal{G}_{d}^{(n)}$ (resp. $\mathcal{G}_{2c}^{(n)}$, resp. $\mathcal{G}_{3\ell}^{(n)}$) the set of DH graphs (resp. $2$-connected DH graphs, resp.  $3$-leaf power graphs) with vertex set $[n]:=\{1,\dots,n\}$ (we say that such graphs have \emph{size} $n$).

We also denote by $ (\CRT,d_\infty,\mu_\infty)$ the Brownian CRT
 equipped with the mass measure $\mu_\infty$. 
The Brownian CRT has been introduced by Aldous in \cite{AldousCRTIII}
and is a now standard object in the discrete probability literature 
(for details and references, see \cref{Sec:CritereGP}). 

\begin{theorem}\label{th:GrosTheoreme}
For every family $f \in \{d,2c,3\ell\}$ and $n\geq 1$ let $\Gn_f$ be a uniform random graph in $\mathcal{G}_f^{(n)}$. Let $\mu_n$ be the uniform measure on the set of vertices $[n]$ and $d_{\Gn_f}$ be the graph distance in $\Gn_f$. Then there exists a constant $c_f>0$ such that the following convergence holds in distribution for the Gromov--Prohorov topology:
\begin{equation}\label{eq:ConvergenceGrosTheoreme}
\left([n],\frac{c_f}{\sqrt{n}}d_{\Gn_f}, \mu_n\right) \stackrel{n\to +\infty}{\to} (\CRT,d_\infty,\mu_\infty).
\end{equation}
\end{theorem}
Constants in \cref{eq:ConvergenceGrosTheoreme} are explicit: namely,
\begin{align*}
c_d&= \frac{\sqrt{2}}{\gamma_H} \approx 0.3602 \text{ \quad where } \gamma_H \text{ is defined in \cref{eq:gamma_H} p.~\pageref{eq:gamma_H},}\\
c_{2c}&= \frac{\sqrt{2}}{\gamma_{H,2c}} \approx 0.1885 \text{ \quad where } \gamma_{H,2c} \text{ is defined in \cref{eq:gamma_H_2c} p.~\pageref{eq:gamma_H_2c},}\\
c_{3\ell}&=\frac{\sqrt{2}}{\gamma_{E,3\ell}} \approx 0.9266 \text{ \quad where } \gamma_{E,3\ell} \text{ is defined in \cref{eq:gamma_E3l} p.~\pageref{eq:gamma_E3l}}.
\end{align*}

\cref{fig:simulation,fig:simulation2c,fig:simulation3l} show two realizations of uniform distance-hereditary graphs
with  a few hundred vertices, respectively in the unconstrained, $2$-connected and 3-leaf power graph cases.
\begin{figure}[htbp]
\begin{center}
\includegraphics[width=7cm]{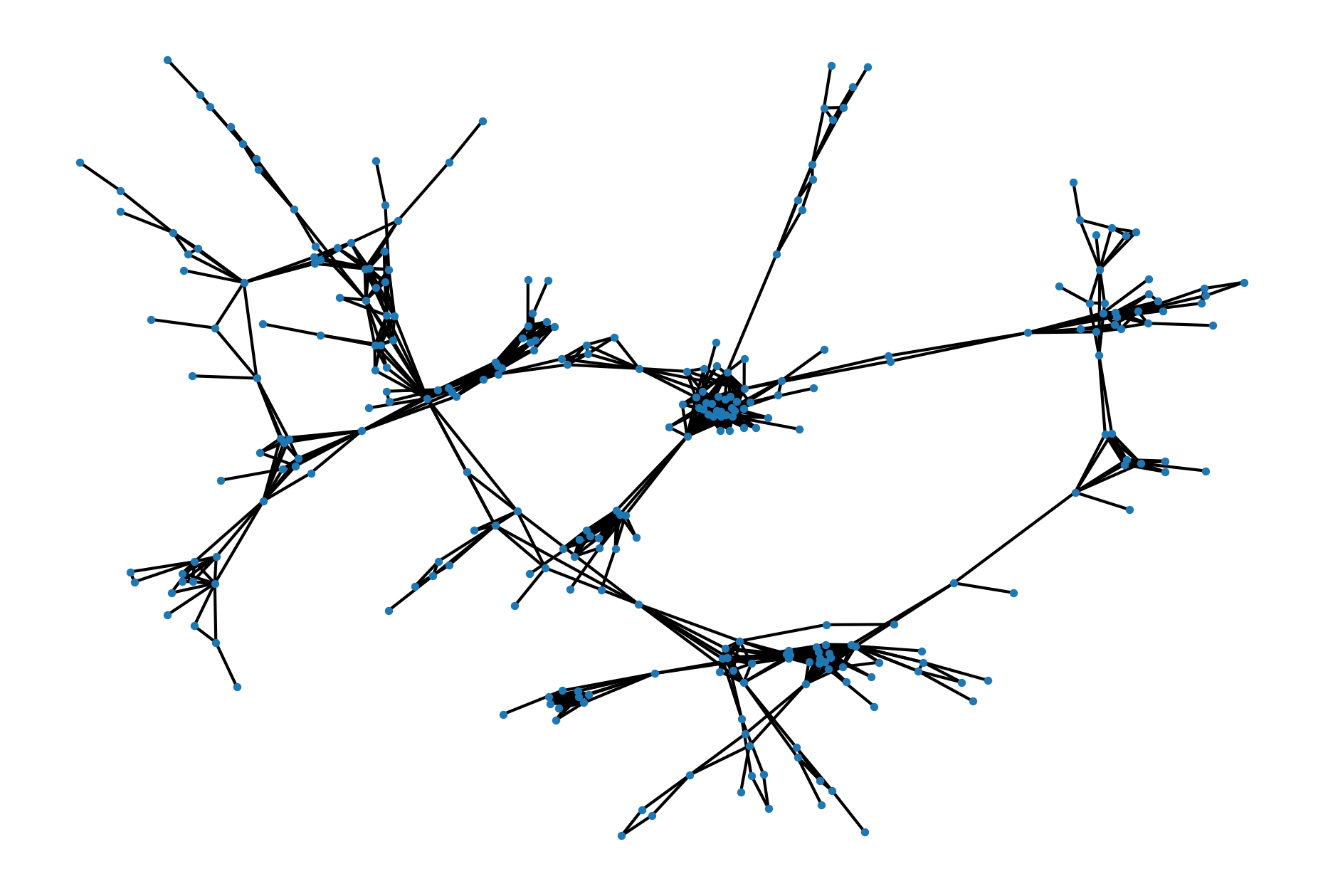}
\hspace{1cm}
\includegraphics[width=7cm]{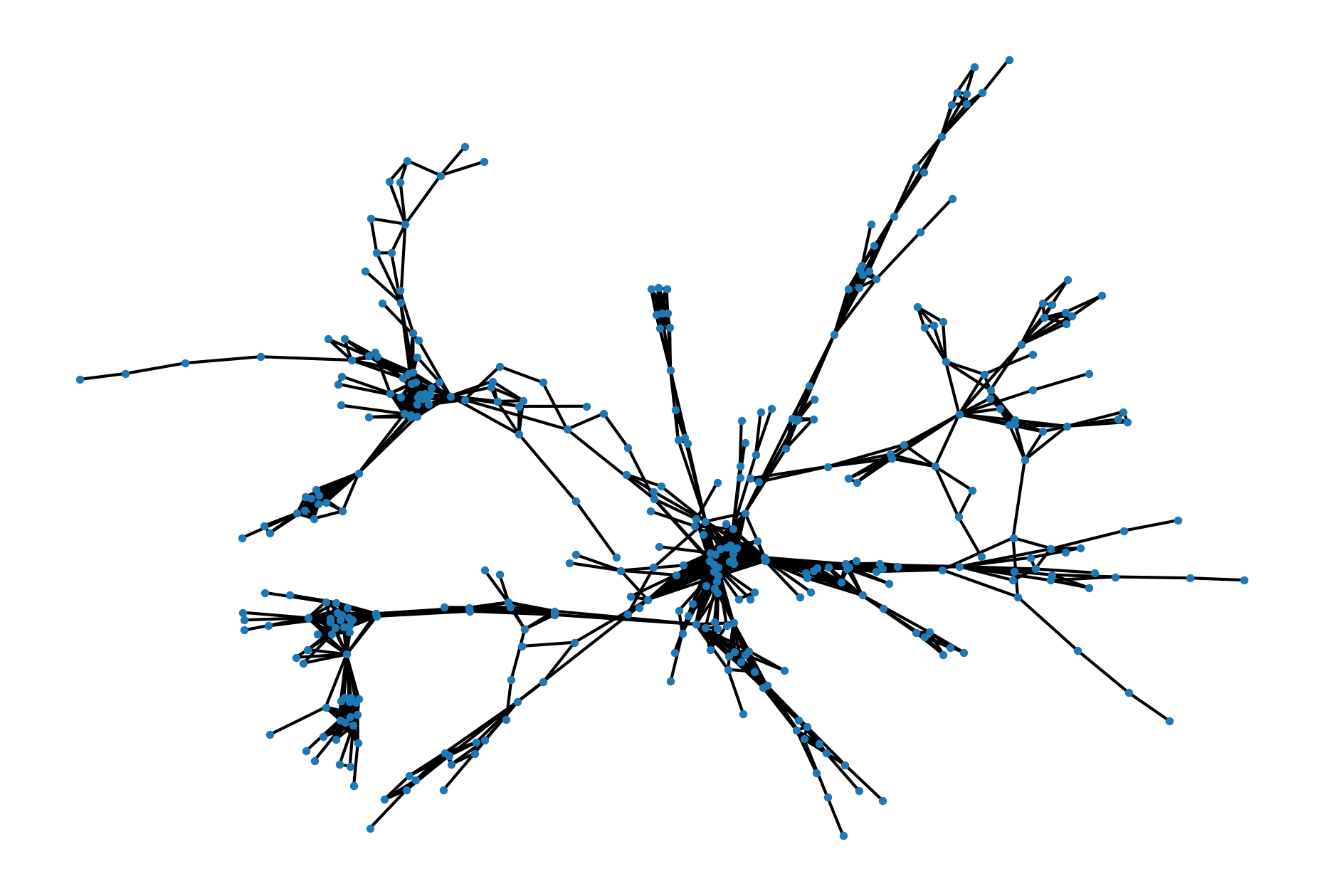}
\end{center}
\caption{Two samples of uniform random DH graphs of respective sizes $n=290$ and $n=388$. Both graphs were generated with a Boltzmann sampler (see \cite{Boltzmann}) using the combinatorial specification given in  \cref{eq:SpecifDH} p.~\pageref{eq:SpecifDH} and plotted with \texttt{python} library \texttt{networkx}.}
\label{fig:simulation}
\end{figure}

\begin{figure}[htbp]
\begin{center}
\includegraphics[width=7cm]{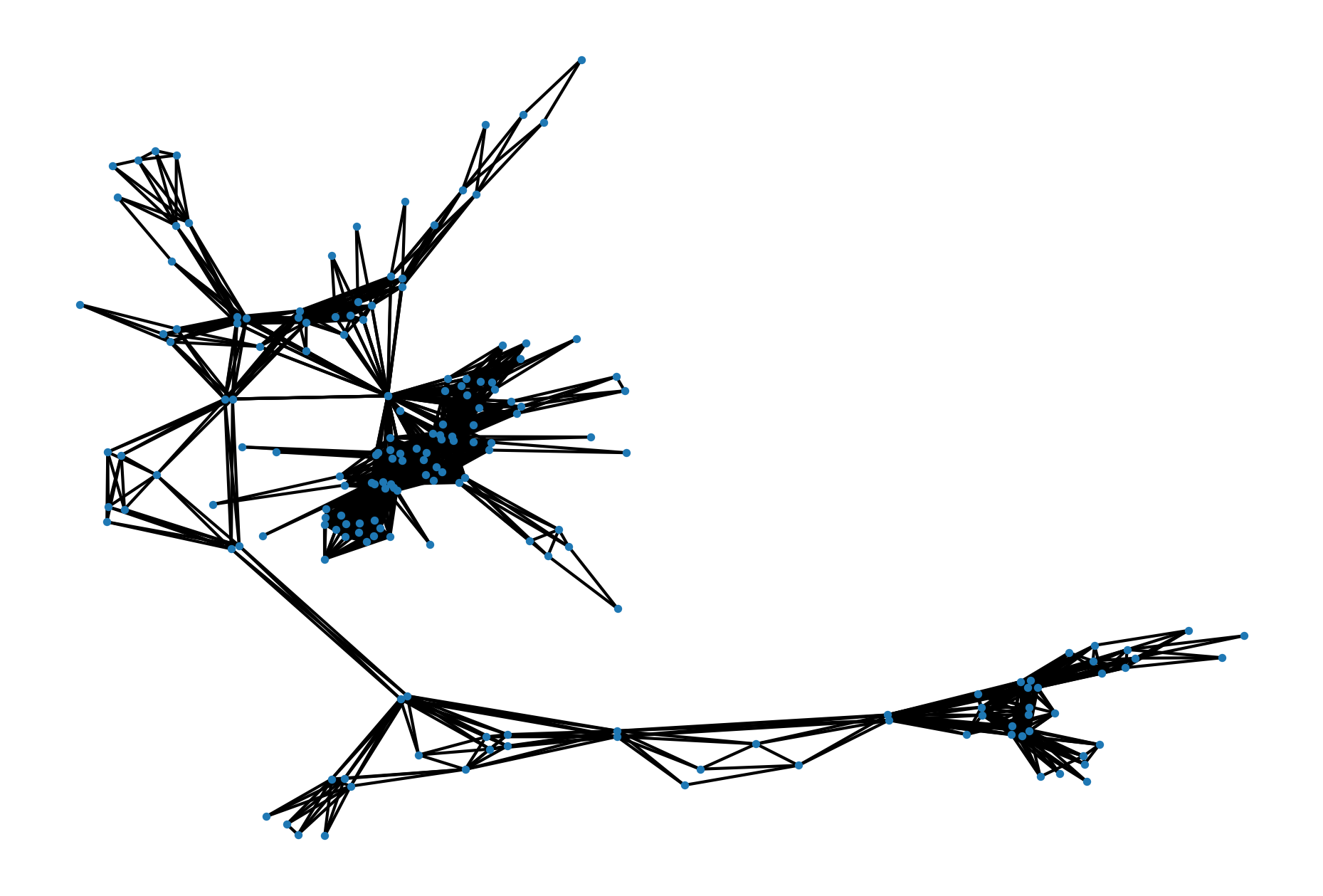}
\hspace{1cm}
\includegraphics[width=7cm]{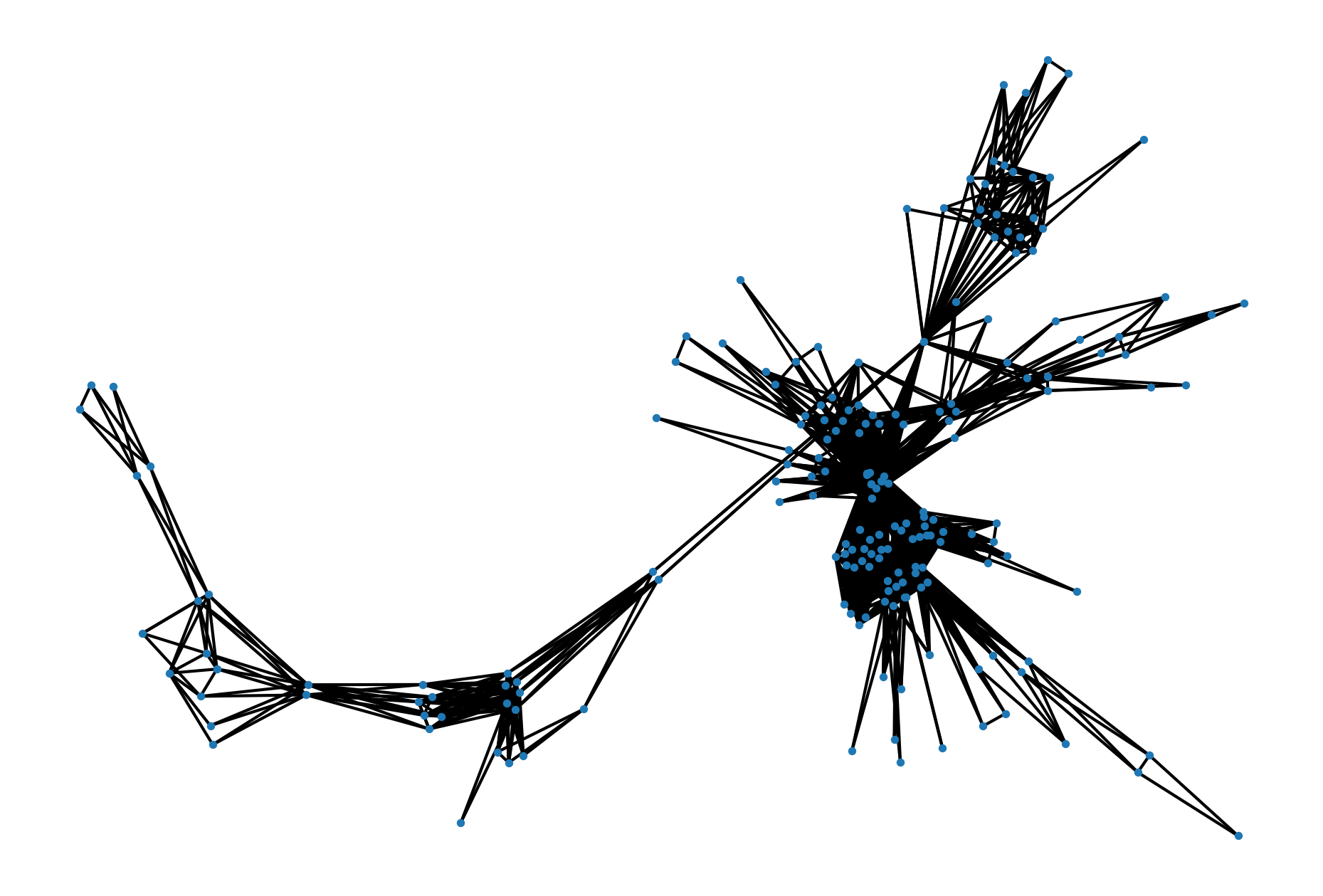}
\end{center}
\caption{Two samples of uniform random $2$-connected DH graphs of respective sizes $n=186$ and $n=197$. Graphs were generated with the combinatorial specification given in  \cref{eq:SpecifDH2c} p.~\pageref{eq:SpecifDH2c}.}
\label{fig:simulation2c}
\end{figure}

\begin{figure} [htbp]
\begin{center}
\includegraphics[width=7cm]{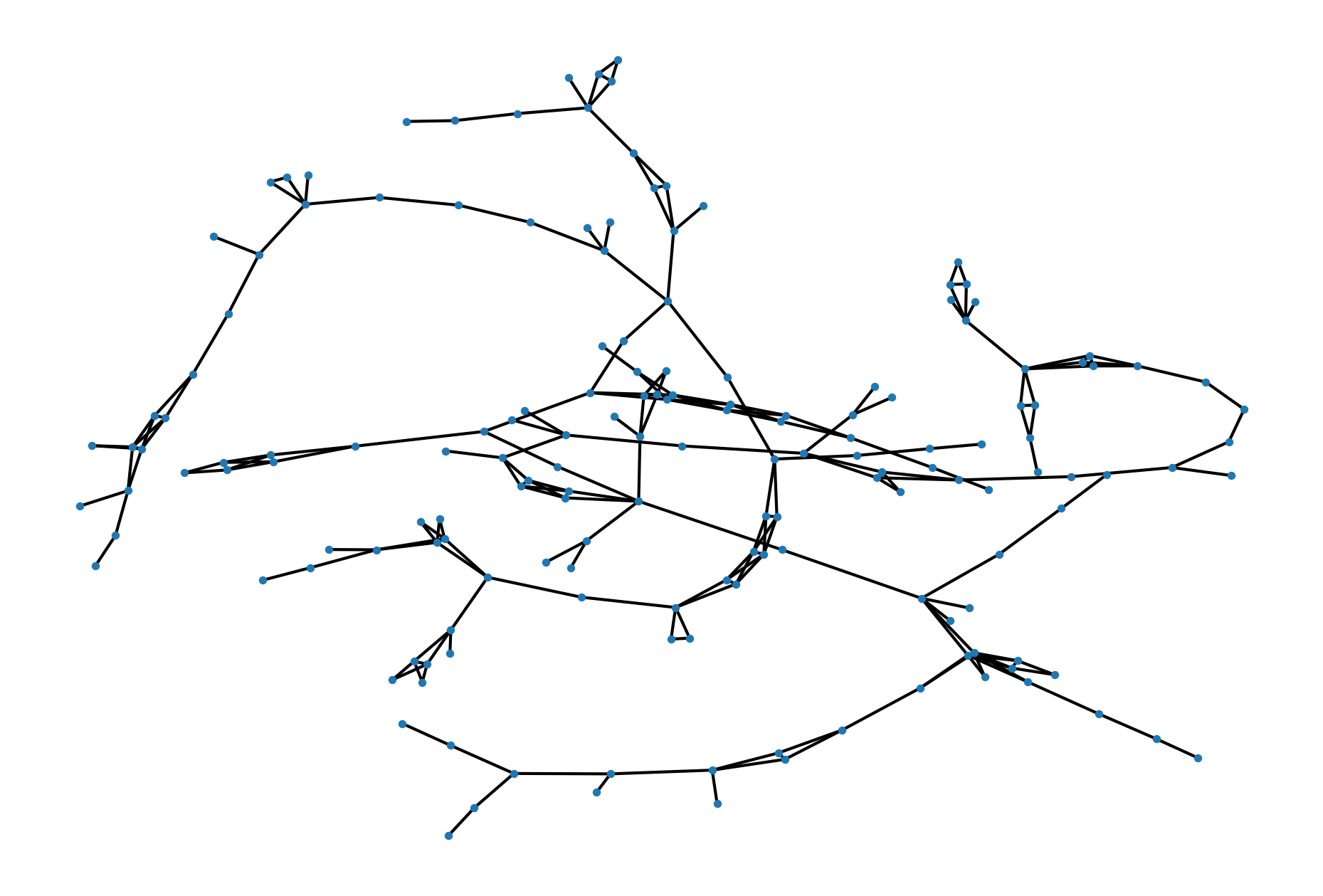}
\hspace{1cm}
\includegraphics[width=7cm]{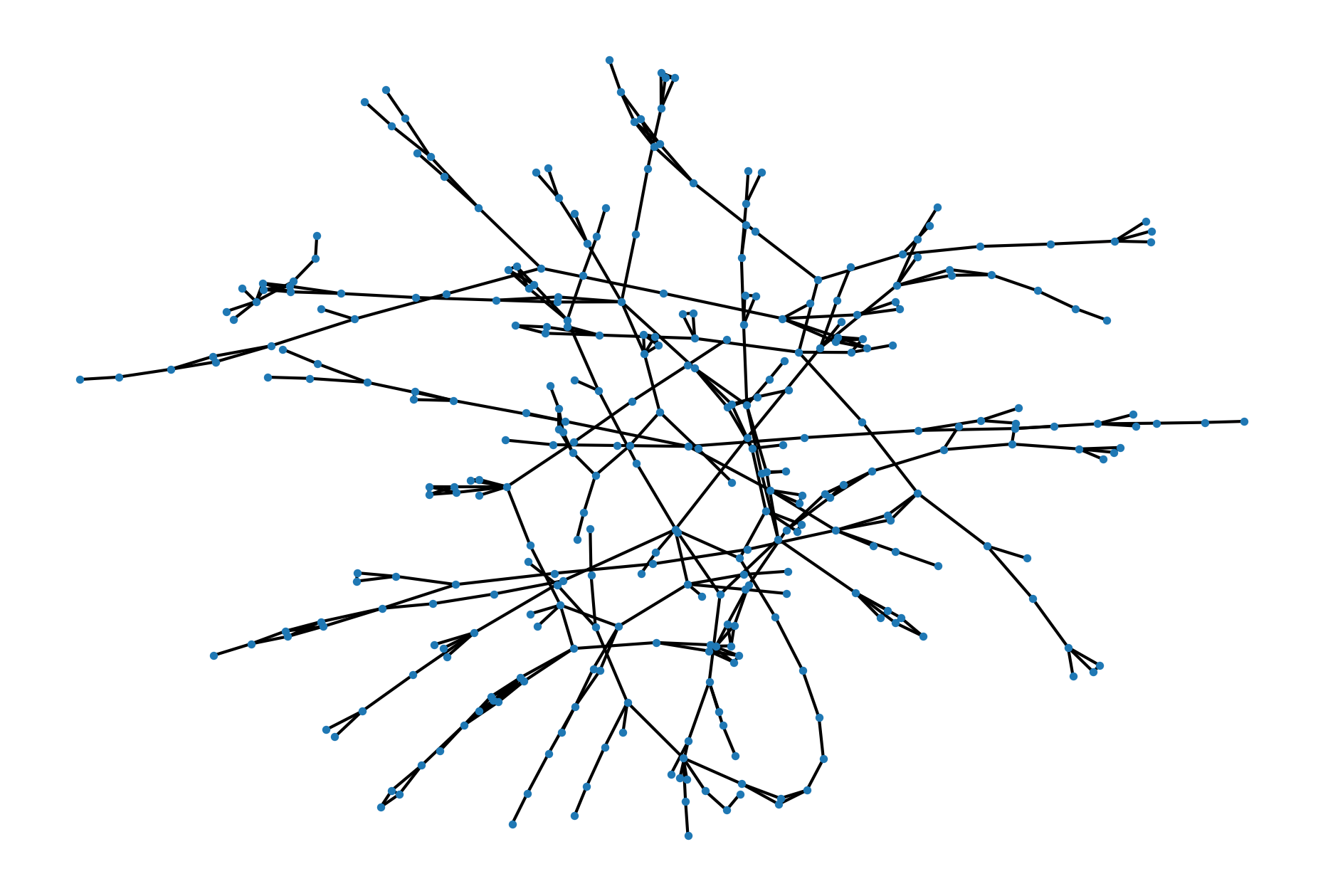}
\end{center}
\caption{Two samples of uniform random $3$-leaf power graphs of respective sizes $n=170$ and $n=231$. Graphs were generated with the combinatorial specification given in  \cref{eq:Specif3Leaf} p.~\pageref{eq:Specif3Leaf}.}
\label{fig:simulation3l}
\end{figure}

As mentioned above, in the case $f=d$ ({\it i.e.}~random unconstrained DH graphs),
\cref{th:GrosTheoreme} is not new. Indeed, DH graphs
form a subcritical block-stable class of graphs,
and it is proved in \cite{SubcriticalClasses} that uniform random graphs 
in such classes converge to the Brownian CRT\footnote{In \cite{SubcriticalClasses},
the convergence is proven only for the Gromov-Hausdorff (GH) topology, which is incomparable with
the GP topology we use here. We believe however that without much further effort,
their argument in fact proves convergence in the stronger GHP topology, see \cref{Sec:AppendixGH-GHP} for details.}.
On the contrary, $3$-leaf power graphs and $2$-connected DH graphs
are not block-stable graph classes, and \cref{th:GrosTheoreme} is new in these cases.
The stronger connectivity of $2$-connected DH graphs is reflected 
in the value of the renormalizing constant, which is smaller in the $2$-connected case
than in the unconstrained and $3$-leaf power cases.

We can restate  \cref{th:GrosTheoreme} in more concrete terms, 
which actually describe how we intend to prove \cref{th:GrosTheoreme}. 
It is known (see \cite{GromovProhorov} or \cref{Sec:CritereGP} below)
that  convergence in distribution in the Gromov--Prohorov sense is equivalent to the convergence in distribution of the relative distances between $k$ uniform vertices in the graph, for every $k$. For $k=2$ \cref{th:GrosTheoreme} says that if ${\mathbf v}_0,{\mathbf v}_1$ are uniform i.i.d.~vertices in $\Gn_f$ then
\begin{equation}\label{eq:Fonction2points}
\frac{c_f}{\sqrt{n}}d_{\Gn_f}({\mathbf v}_0,{\mathbf v}_1) 
\xrightarrow[n\to+\infty]{(d)} d_\infty(v_0,v_1)
\end{equation}
where $v_0,v_1$ are independent and $\mu_\infty$-distributed in $\CRT$. It turns out that the random variable  $d_\infty(v_0,v_1)$ is known to follow the Rayleigh distribution, \emph{i.e.} has density
$xe^{-x^2/2}$ on $\bbR_+$. More generally \cref{th:GrosTheoreme} amounts to saying that \eqref{eq:Fonction2points} holds jointly for $k$ uniform i.i.d. vertices in $\Gn_f$. 
The joint limiting distribution, \emph{i.e.}~the distribution of the distances between $k$ random
points in the CRT, is given below in \cref{lem:distances_CRT} (see also \cite{AldousCRTIII}).

We finish by discussing how our result fits in the literature on  convergence of discrete graph models to the CRT.
It is now well established that the Brownian CRT is the universal limit of many important families of random trees, see, \emph{e.g.}, \cite{LeGall}.
In addition, a few families of graphs which are \emph{not} trees are also known to converge towards the CRT, although such results are less common in the literature. We can cite
some models of random planar maps \cite{AlbenqueMarckert,Bettinelli,Caraceni,JansonStefansson},  some models of  random dissections \cite{RandomDissections},  and random graphs in subcritical \emph{block-stable} graph classes \cite{SubcriticalClasses}, as mentioned above.
Our paper exhibits two new families of nontree graphs classes converging to the CRT (and an alternative proof for a third class).

\subsection{Proof strategy}
As indicated above, one goal of this paper is to investigate 
the possibility of establishing scaling limit results for graphs
using the split decomposition.
In this regard, it is a natural continuation of a series of papers,
using other tree decompositions to obtain limiting results for combinatorial objects:
the substitution decomposition for permutations \cite{Nous1,Nous2,Nous3,SubsClosedRandomTrees},
and the modular decomposition for graphs \cite{NousCographes,BCographes}.
In most of these papers, the general proof strategy is the following.
First, we use some criteria to characterize the convergence of our combinatorial objects $G_n$
(either in the permuton or graphon sense) as the
convergence of the density of every substructure $H$ (pattern or induced subgraph) 
induced by $k$ random elements in our random permutation or random graph.
Then, fixing a substructure $H_0$ of size $k$, we study the combinatorial class of objects
with $k$ marked elements inducing that substructure $H_0$.
From this, we use singularity analysis to compute the asymptotic density of $H_0$
in a random object $G_n$.

The general strategy used in this paper is similar, albeit with important novelties.
As said above, the Gromov--Prohorov convergence is equivalent to the convergence, 
for all $k$, of the matrix of the distances between $k$ uniform random vertices in the graph.
This criterion resembles those for permuton and graphon convergence,
except that, for a fixed $k$, distance matrices live in a continuous space 
(real-valued $k \times k$ matrices), while patterns or induced subgraphs belong to a finite set.
Hence, to prove Gromov--Prohorov convergence, it is not possible
to simply consider the probability that the distance matrices are equal to a given matrix,
and study its asymptotics through analytic combinatorics.
To overcome this difficulty, we need to consider multivariate generating series,
where the additional parameters encode various distances in the graph induced by $k$ random points.
Then, instead of the classical transfer theorem, 
we use a slightly generalized version of the Semi-large powers Theorem
(see \cref{Sec:AppendixSemiLarge});
this theorem is known to explain from an analytic point of view the appearance
of the Rayleigh distribution, it is therefore not surprising that we use it here.

As a final note, let us mention that we are not aware of other works,
where convergence to the Brownian CRT is proved through the same set of tools.
We hope that this method will prove useful in other contexts in the future.

\begin{remark}
A natural alternative strategy to prove our main result would be the following:
first prove that the split decomposition tree $\Tn_f$ associated with $\Gn_f$ tends to the CRT,
and then prove that $\Gn_f$ and $\Tn_f$ are close, up to some scaling factor, for the GP topology.
This is in essence the strategy used in \cite{SubcriticalClasses} for subcritical block-stable
classes of graphs, except that the block-decomposition tree is used instead of the split decomposition tree.
There are however important (though not necessarily impossible to overcome) difficulties
to follow this route in our case.

First, the split decomposition trees $\Tn_f$ associated to our three models
can be represented as multitype Galton-Watson trees conditioned to having a given
number of leaves (as witnessed by the systems of equations \eqref{eq:SpecifDH},
\eqref{eq:SpecifDH2c} and \eqref{eq:Specif3Leaf}).
Convergence results to the CRT for conditioned multitype Galton-Watson trees
are available in the literature (see, \emph{e.g.}, \cite{miermont2008multitype}).
However, such results are usually obtained for trees conditioned to having
a given number of \emph{vertices},
and in the irreducible case.
Here we want to condition on the number of leaves,
and, in one of our models, namely for 3-leaf power graphs,
the system of equations defining the class
is not irreducible, see \cref{eq:Specif3Leaf}.
Therefore proving the convergence of the split decomposition trees
to the CRT would need some work on models of random trees.

A second difficulty is that the convergence of the split decomposition trees
does not imply directly the convergence of the associated graphs.
For this, we would need to prove that distances in the graph are close, up to a constant factor,
to that in the tree. But distances in the graph are determined by the decoration
of vertices in the split decomposition (see \cref{ssec:jumps}).
One would therefore need to understand the distribution of such decorations
(\emph{i.e.}~of types in our multi-type model) on paths between marked leaves and branching
points in split decomposition trees. Again, this might be feasible but certainly requires work.

We have preferred to develop an approach via analytic combinatorics, as explained above,
which is in some sense more direct and more original.
\end{remark}

\subsection{Outline of the paper}
In order to simplify the presentation of the proofs we chose to focus first on the class of 
unconstrained DH graphs. 
We explain later (in \cref{Sec:2connectedDH,Sec:3leaf}) how to adapt the result
 to $2$-connected DH graphs and to $3$-leaf power graphs.
\begin{itemize}
  \item In \cref{Sec:CritereGP} we state a criterion for the convergence towards the Brownian CRT w.r.t.~the Gromov--Prohorov topology. This criterion essentially follows from \cite{GromovProhorov,loehr2013equivalence} and from Aldous' construction of the Brownian CRT \cite{AldousCRTIII}.
\item In \cref{Sec:CombiDH} we give the necessary background of graph theory. We will see that there is a correspondence between DH graphs and certain \emph{clique-star trees}. \cref{Sec:CombiDH} ends with exact and asymptotic enumerative formulas for DH graphs. 
  The material of this section is mainly taken from papers of Gioan--Paul \cite{SplitTrees}
  and Chauve--Fusy--Lumbroso \cite{ChauveFusyLumbroso}.
\item \cref{Sec:TreesWithMarkedLeaves} is devoted to the combinatorial and analytic study of clique-star trees with  a marked leaf. These are building blocks for the combinatorial decomposition
  of trees with several marked leaves done in \cref{Sec:TreesWithInducesSubtrees},
  keeping track of distances in the graph between the corresponding vertices.
The convergence of a uniform random unconstrained DH graph to the Brownian CRT,
\emph{i.e.}~the case $f=d$ in \cref{th:GrosTheoreme}, is proved at the end 
of \cref{Sec:TreesWithInducesSubtrees}.
\item In \cref{Sec:2connectedDH,Sec:3leaf} we extend the main result to $2$-connected DH graphs and to $3$-leaf power graphs, respectively.
\item In \cref{Sec:AppendixSemiLarge} we give a complete proof of a (minor) generalization of the \emph{Semi-large powers Theorem} (\cite[Theorem IX.16]{Violet}),
  which is central in our proofs. 
\item  \cref{Sec:AppendixDH_are_stable} and  \cref{Sec:AppendixGH-GHP} clarify the relation between the present work and the paper  \cite{SubcriticalClasses}.
\end{itemize}
\medskip

{\em Note:} 
Some computations in the proofs of our main results require the use
of a computer algebra system. To help the reader, we provide a companion
Maple worksheet, both in mw and pdf formats.
These files are embedded into this pdf
(alternatively you can download the source of the arXiv version to get the files).

\section{Toolbox: the Gromov--Prohorov topology and the Brownian CRT}\label{Sec:CritereGP}

\subsection{A criterion for Gromov--Prohorov convergence}

\begin{definition}
A metric measure  space (called mm-space for short) is a triple
$(X,d,\mu)$, where $(X,d)$ is a complete and separable metric space 
and $\mu$ a probability measure on $X$.
\end{definition}

{\em Gromov--Prohorov distance.}
We let $\MGP$ be the set of all mm-spaces\footnote{To avoid Russell's paradox, throughout the section,
we actually take the set of mm-spaces {\em whose elements are not themselves metric spaces}.}, modulo the following relation:
$(X,d,\mu) \sim (X',d',\mu')$ if there is an isometric embedding  $\Phi:X \to X'$
such that $\Phi_*(\mu)=\mu'$.
Note that $\Phi$ does not need to be surjective, so that we need to consider
the transitivity and reflexivity closure of that relation.
In particular one always has $(X,d,\mu) \sim (\Supp(\mu),d,\mu)$, 
where $\Supp(\mu)$ is the support of $\mu$.

On the set $\MGP$, one can define a distance as follows.
First we recall the notion of Prohorov distance:
for Borel probability measures $\mu$ and $\nu$ on the same metric space $Y$, we set
\[d_P(\mu,\nu)=\inf\big\{\eps>0: \, \mu(A)\le \nu(A^\eps)+\eps \text{ and }\nu(A)\le \mu(A^\eps)+\eps
\text{ for all measurable sets }A\subseteq Y\big\},\]
where $A^\eps$ is the $\eps$-halo of $A$, {\it i.e.} the set of all points at distance at most $\eps$ of $A$.
This distance metrizes the weak convergence of probability measures.
Then, given two mm-spaces $(X,d,\mu)$ and  $(X',d',\mu')$, we set
\[\DGP\big( (X,d,\mu), (X',d',\mu') \big)
= \inf_{(Y,d_Y),\Phi,\Phi'}  d_P\big( \Phi_*(\mu), \Phi'_*(\mu') \big),\]
where the infimum is taken over isometric embeddings $\Phi:X \to Y$ and $\Phi':X' \to Y$
into a common metric space $(Y,d_Y)$.
One can prove \cite[Section 5]{GromovProhorov} that $\DGP$ is a distance on $\MGP$ 
and that the resulting metric space $(\MGP,\DGP)$ is complete and separable.

\medskip

{\em Criterion of convergence.}
Let $\mathcal X=(X,d,\mu)$ be an mm-space and fix an integer $k \ge 0$.
We let $x_1, \dots, x_k$ be i.i.d. random elements of $X$, with law $\mu$.
We record their pairwise distances in a matrix, namely we set
\[A_k^{\X} = \big( d(x_i,x_j) \big)_{1\le i,j \le k}.\]
This is a random $k \times k$ square matrix, whose law depends on the mm-space $\X$ 
we start with. 

We will also consider random mm-spaces, which we denote with boldface. 
In this case, {\em conditionally on $\bX=\bm{(X,d,\mu)}$},
 we let $x_1, \dots, x_k$ be i.i.d. random elements of $\bm X$,
with law $\bm \mu$ and we define as above $A_k^{\bX}$ to be their distance matrix.

We have the following characterization of convergence in distribution in $(\MGP,\DGP)$,
essentially given in \cite{GromovProhorov,loehr2013equivalence}.
\begin{theorem}\label{Th:CritereGP}
Let ${\bX_n}=\bm{(X_n,d_n,\mu_n)}$ for any $n\geq 1$ and ${\bX}=\bm{(X,d,\mu)}$ be random mm-spaces.
Then the following properties are equivalent:
\begin{enumerate}
\item $\bX_n$ converges in distribution to $\bX$ for the Gromov--Prohorov distance $\DGP$ as $n \to +\infty$. 
\item For any fixed $k\ge 1$, the random distance matrix $A_k^{\bX_n}$ converges in distribution to $A_k^{{\bX}}$ as $n$ tends to $+\infty$.
\end{enumerate}
\end{theorem}
\begin{proof}
In \cite[Theorem 5]{GromovProhorov}, it is proved in the deterministic setting that convergence for Gromov--Prohorov distance
is equivalent to the convergence of the so-called {\em polynomial functions},
 {\it i.e.}~of bounded continuous functions of (the entries of) distance matrices.
 It is then observed in \cite[Corollary 2.8]{loehr2013equivalence} that polynomial functions
 are {\em convergence-determining}, {\it i.e.} one has convergence in distribution of random mm-spaces
 if the expectations of all polynomial functions converge.
 On the other hand since polynomial functions are the continuous bounded 
 functions of distance matrices,
 the convergence of expectations of polynomial functions is equivalent
  to the convergence in distribution of the distance matrices. This completes the proof.
\end{proof}

\subsection{Distance matrix of the CRT}\label{ssec:CRT}

The Brownian CRT $(\CRT,d_\infty,\mu_\infty)$ is a random variable taking values in the set of compact metric measure spaces (see \cite{AldousCRTIII}).

Informally, the mutual distances of $k$ points in $(\CRT,d_\infty,\mu_\infty)$ have the same distribution as the distances between the $k$ leaves of a uniform random \emph{$k$-proper tree} (defined below) in which edges have random length distributed according a multivariate Rayleigh distribution.
This actually characterizes the distribution of $\CRT$, as stated below. 

\begin{definition}
A $k$-proper tree $t_0$ is an (unrooted) nonplane tree with $k+1$ leaves
where each internal node has degree $3$.
One of the leaves is considered as the root-leaf ($\ell_0$) and the other leaves are identified with $\{\ell_1,\dots,\ell_k\}$.
\end{definition}

\begin{lemma}
\label{lem:distances_CRT}
For every $k\geq 2$ and every $k$-proper tree $t_0$, we fix a labeling of its edges $e_0,\dots ,e_{2k-2}$.

The distribution of $(\CRT,d_\infty,\mu_\infty)$ is characterized by the property that for every $k\geq 2$, 
if we take $v_0,\dots , v_k$ uniform and independent in $\CRT$ with distribution $\mu_\infty$, then 
\begin{equation}\label{eq:MarginCRT}
\bigg(d_\infty(v_i,v_j)\bigg)_{0\leq i,j\leq k}\stackrel{(d)}{=}\bigg(\sum_{r:\ e_r\in\mathcal{P}_{i,j}^{\bm{t}_0}} X_r\bigg)_{0\leq i,j\leq k},
\end{equation}
where the RHS is a random matrix whose distribution is defined as follows (where $(2k-3)!!$ is the product of all odd positive integers less than or equal to $2k-3$):
\begin{itemize}
\item  $\bm{t}_0$ is a uniform  $k$-proper tree;
\item $(X_0,\dots,X_{2k-2})$ have joint distribution
$$
(2k-3)!! \, \cdot \, \Sigma_i{x_i}\exp\left(-(\Sigma_i x_i)^2/2\right)\mathbf{1}_{x_0 ,\dots,x_{2k-2}>0}\ \mathrm{d}x_0\dots \mathrm{d}x_{2k-2}
$$
and are independent from $\bm{t}_0$;
\item The sum in \eqref{eq:MarginCRT} runs over the set of edges $e_r$ of the path $\mathcal{P}_{i,j}^{\bm{t}_0}$ joining leaves $\ell_i$ and $\ell_j$ in $\bm{t}_0$.
\end{itemize}
\end{lemma} 

Aldous proved that this object exists and indeed properly defines a random metric space 
\cite[Lemma 21]{AldousCRTIII}. The reader may be more familiar with an alternative and more constructive definition of the CRT which we briefly recall. 
Starting from a normalized Brownian excursion $\Exc$, 
$\CRT$ is defined as the quotient $[0,1]/\sim_{\Exc}$ where $\sim_{\Exc}$ is the ``gluing'' procedure which identifies any two points of $\Exc$ at the same height having only higher points of $\Exc$ between them (see~\cite[Section 2]{LeGall}). 
Aldous~\cite[Cor. 22]{AldousCRTIII} proved that both constructions coincide. 
Through the latter construction, the mass measure $\mu_\infty$ is defined as the push-forward of the Lebesgue measure by the quotient map associated with $\sim_{\Exc}$. 

\section{Combinatorial analysis of distance-hereditary trees}\label{Sec:CombiDH}

In this section, we first recall the encoding of distance-hereditary graphs by clique-star trees
(which is a special case of the encoding of general graphs by split decomposition trees).
This is done in \cref{sec: split} and largely follows \cite[Sections 2.1-2.2]{SplitTrees} (itself inspired by \cite{CunninghamSplit}).
We then explain how distances in a DH graph can be recovered from the associated clique-star tree (\cref{ssec:jumps}).
We could not find this result in the literature, though this might be known to experts.
The last two sections provide a combinatorial and analytic study
of the generating series of DH graphs (or rather of the associated trees);
this mainly follows the work of Chauve--Fusy--Lumbroso \cite{ChauveFusyLumbroso}.
This whole section can be seen as combinatorial preliminaries
for the proof of the convergence of unconstrained DH graphs to the Brownian CRT
(case $f=d$ in \cref{th:GrosTheoreme}).

\subsection{Clique-star trees} \label{sec: split}

\begin{definition}
A graph-decorated\footnote{In \cite{SplitTrees}, the term {\em graph-labeled tree} is used;
we prefer here to speak of {\em graph-decorated tree} to avoid
confusion with labeling in the sense of labeled combinatorial classes \cite{Violet},
a notion that we will use throughout the article.} tree is a (nonplane unrooted) tree $\tau$ in which every 
internal node $v$ of degree $k$ is decorated with
a graph $\Gamma_v$ with $k$ vertices;
moreover, for each $v$, we fix a bijection $\rho_v$ from the tree-edges incident to $v$ to the vertices of $\Gamma_v$. 
\end{definition}
We fix some terminology and conventions. 
To avoid confusion between decoration graphs $\Gamma_v$
and other graphs, we use the term \emph{decoration} for $\Gamma_v$ and {\em marker vertices} for its vertices.
An edge $e$ of $\tau$ between two nodes $v$ and $v'$
is sometimes seen as connecting the marker vertices $q=\rho_v(e)$ to $q'=\rho_{v'}(e)$.
In particular in graphical representations, we draw an edge $e$ of the tree between nodes $v$ and $v'$
from $q=\rho_v(e)$ to $q'=\rho_{v'}(e)$. 
When we refer to the bijection $\rho_v$, we say that an edge $e$ incident to $v$ is \emph{attached} to the corresponding marker vertex (say, $x$) of $\Gamma_v$. 
When $e$ is incident to $v$ and to a leaf $\ell$, we make a small abuse of notation by saying that $\ell$ is attached to $x$.

\medskip

Let $\tau$ be a graph-decorated tree and $\ell$, $\ell'$ be leaves of $\tau$.
We consider the (unique) path $p$ from $\ell$ to $\ell'$ in $\tau$. 
For any node $v$ on this path, we denote $e_{in}(v)$ (resp. $e_{out}(v)$) the edge of $p$ entering (resp. leaving) $v$.
Then $\ell'$ is said to be accessible from $\ell$ (or equivalently $\ell$ accessible from $\ell'$)
if, for every node $v$ on $p$,
the pair $\{\rho_v(e_{in}(v)),\rho_v(e_{out}(v))\}$ is an edge of the decoration $\Gamma_v$.
With this notion in hand, we can associate to $\tau$ a graph $\Gr(\tau)$, whose vertex set is
the leaf set of $\tau$, and where $\{\ell,\ell'\}$ is an edge in $\Gr(\tau)$
if and only if $\ell$ is accessible from $\ell'$ in $\tau$.
This construction is illustrated on~\cref{fig:Correspondence_TreeGraph}.
\medskip

\begin{figure}[htbp]
\includegraphics[width=10cm]{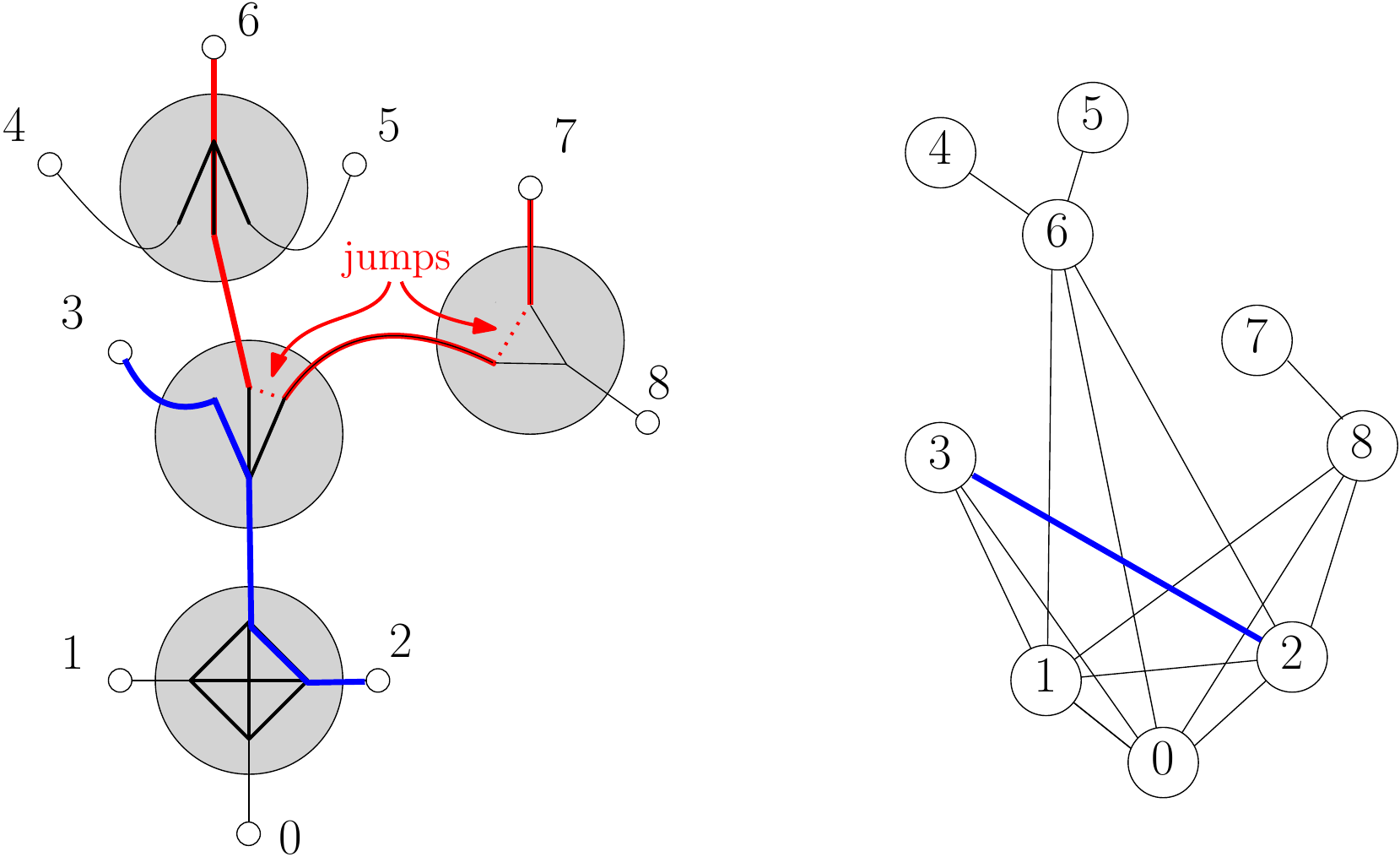}
\caption{Left: A clique-star tree $\tau$ with $n=9$ leaves drawn with its $4$ decorations. Right: The corresponding graph $\Gr(\tau)$.
To illustrate the construction of $\Gr(\tau)$, we have highlighted two pairs of leaves and the paths
between them. Following the blue path, we say that 3 is accessible from 2 in $\tau$;
accordingly $\{2,3\}$ is an edge in $\Gr(\tau)$.
On the opposite, $6$ is not accessible from $7$ in $\tau$;
accordingly $\{6,7\}$ is not an edge in $\Gr(\tau)$. (Jumps are defined in \cref{ssec:jumps}.)}
\label{fig:Correspondence_TreeGraph}
\end{figure}

In the sequel, we only consider graph-decorated trees $\tau$ 
where all decorations $\Gamma_v$ are either cliques or stars --
following \cite{ChauveFusyLumbroso}, we speak of {\em clique-star trees}. 
It is known (see~\cite[Section 3.1]{SplitTrees}) that 
the graphs which can be obtained as $\Gr(\tau)$ where $\tau$ is a clique-star tree,
are precisely the distance-hereditary graphs (DH graphs).
By convention the graph with a single vertex and the connected graph with two vertices
are DH graphs.

We note that a DH graph $G$ can possibly be obtained
as $\Gr(\tau)$ for several clique-star trees $\tau$.
Uniqueness can nevertheless be ensured adding extra conditions on $\tau$.
\begin{definition}\label{dfn:reduced_clique_star_tree}
A clique-star tree $\tau$ is called \emph{reduced} if it satisfies the following conditions:
\begin{enumerate}
\item every internal node $v$ has degree at least $3$;
\item no edge of $\tau$ connects two internal nodes
both decorated with cliques;
\item no edge of $\tau$ connects marker vertices $q$ and $q'$ 
where $q$ is the center of a star $\Gamma_v$ and $q'$ a leaf
of another star $\Gamma_{v'}$.
\end{enumerate}
\end{definition}
Then uniqueness follows directly from \cite[Theorem 2.9]{SplitTrees}
(which considers all graphs, not only DH graphs). Namely, the following holds. 
\begin{proposition} 
\label{prop:SplitTreeBij}
For every labeled DH graph $G$ of size at least $3$,
 there exists a unique reduced clique-star tree $\tau$
such that $G=\Gr(\tau)$.
\end{proposition}

\subsection{Distances in DH graphs through their clique-star trees}
\label{ssec:jumps}

Let $\tau$ be a clique-star tree and $G=\Gr(\tau)$ be the corresponding graph (which is a DH graph as we have seen). 
We denote by $d_G$ the graph distance in $G$. In this section,
we explain how $d_G$ can be read on the tree $\tau$. We recall that the leaves of $\tau$ are identified with the vertices of $G$.

For a path $p$ in $\tau$, the \emph{jumps} of $p$ are defined as follows.
When $p$ goes through a node $v$, it enters and exits through edges $e_{in}(v)$ and $e_{out}(v)$ (both incident to $v$).
If $\{\rho_v(e_{in}(v)),\rho_v(e_{out}(v))\}$ is not an edge in $\Gamma_v$, we say that $v$ is a jump of $p$. 
(In particular, and unless otherwised specified, the starting and ending points of $p$ are not jumps of $p$.)
Now, for two leaves $\ell$ and $\ell'$ of $\tau$, 
letting $p$ be the unique path from $\ell$ to $\ell'$ in $\tau$, 
the number of jumps of $p$ is denoted by $\jump(\tau,\ell,\ell')$.

\begin{lemma}\label{lem:dG}
Let $\tau$ be a clique-star tree with corresponding 
DH graph $G=\Gr(\tau)$, 
and let $\ell$, $\ell'$ be leaves of $\tau$.
Then we have $d_G(\ell,\ell')=\jump(\tau,\ell,\ell')+1$.
\end{lemma}

\begin{example}
Consider the clique-star tree $\tau$ of \cref{fig:Correspondence_TreeGraph},
and its leaves 6 and 7. The path from 6 to 7 (in red on the picture) has exactly two jumps.
Accordingly, the distance between vertices 6 and 7 in the associated DH graph
(also drawn on \cref{fig:Correspondence_TreeGraph}) is 3.
\end{example}
\begin{remark}
According to Lemma~\ref{lem:dG}, $\ell$ is accessible from $\ell'$ in $\tau$ (\emph{i.e.} $\jump(\tau,\ell,\ell')=0$)
if and only if $\{\ell,\ell'\}$ is an edge of $G$ (\emph{i.e.}  $d_G(\ell,\ell')=1$).
In other words, the lemma superseeds and generalizes the definition of the edge set of $\Gr(\tau)$.
\end{remark}

\begin{proof}
We proceed by induction. If $\tau$ has a single internal node,
then $G$ is isomorphic to the decoration of that node (hence, either a clique or a star),
and the statement holds trivially.

Let $\tau$ have $k>1$ internal nodes and assume that the statement holds for all
clique-star trees with fewer internal nodes.

Consider a node $v$ of $\tau$, all of whose neighbors but one are leaves (such a node always exists).
Denote by $d \geq 3$ the degree of $v$, by $\ell_1$, \dots, $\ell_{d-1}$ the leaves adjacent to $v$, and by $u$ the internal node of $\tau$ adjacent to $v$.
We also denote by $\Gamma_v$ the decoration of $v$, 
and by $x$ the marker vertex of $\Gamma_v$ corresponding to the edge $(v,u)$.
We let $\tau^\star$ be the clique-star tree obtained by replacing $v$ and $\ell_1$, \dots, $\ell_{d-1}$
by a single leaf $\ell^\star$ (adjacent to $u$), and denote by $G^\star=\Gr(\tau^\star)$ the associated graph. 
All these notations are sumarized on \cref{fig:NotationsPreuveDistances}
 for the reader's convenience.
 \begin{figure}[htbp]
 \[\includegraphics[height=35mm]{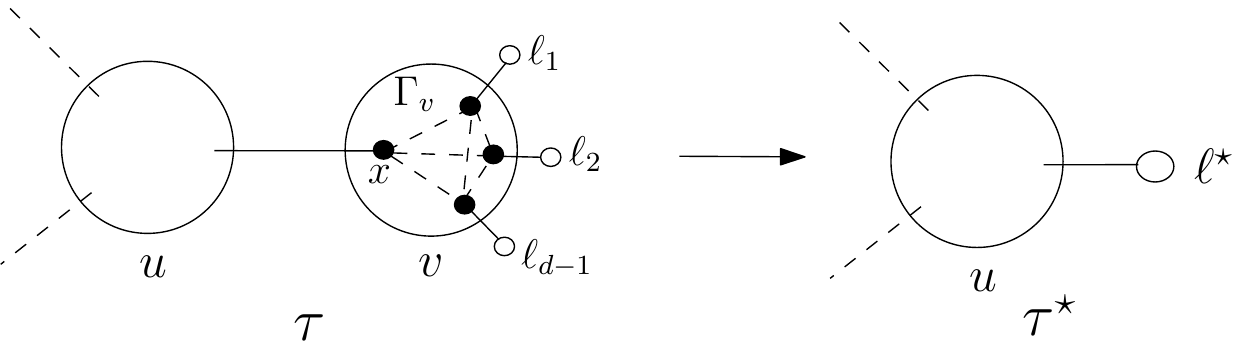}\]
 \caption{Illustration of the proof of \cref{lem:dG}.}
 \label{fig:NotationsPreuveDistances}
 \end{figure}

As we shall see, $G$ can be obtained by performing some local modifications on $G^\star$, 
which depend on $\Gamma_v$ and $x$. 
First note that leaves of $\tau$ and $\tau^\star$ different from  $\ell^\star$, $\ell_1$, \dots, $\ell_{d-1}$ are the same
and are therefore vertices in both $G$ and $G^\star$; we will call them {\em old} vertices, refering to $\ell_1$, \dots, $\ell_{d-1}$ as \emph{new}. 
By construction, adjacency relations between old vertices are identical in $G$ and $G^\star$.
So, knowing $G^\star$, to know $G$ entirely, we just have to describe the adjacency relations among new vertices, and between the new vertices and the old ones. 
To this end, we distinguish several cases. 

\begin{itemize}
 \item If $\Gamma_v$ is a clique, then 
the definition of the construction $\Gr$ implies that 
$G$ is obtained from $G^\star$ by replacing $\ell^\star$ with $d-1$ vertices $\ell_1, \dots, \ell_{d-1}$, 
which form a clique of size $d-1$, 
and such that the old neighbors of each $\ell_i$ are the neighbors of $\ell^\star$ in $G^\star$. 
 \item If $\Gamma_v$ is a star with $x$ the center of the star, 
 then similarly $G$ is obtained from $G^\star$ by replacing $\ell^\star$ with $d-1$ vertices $\ell_1, \dots, \ell_{d-1}$, 
which form an independent set of size $d-1$, 
and such that the old neighbors of each $\ell_i$ are the neighbors of $\ell^\star$ in $G^\star$. 
\item Finally, assume that $\Gamma_v$ is a star and $x$ is not the center of the star. 
Let $\ell_j$ be the leaf of $\tau$ attached to the center of $\Gamma_v$. 
Here, $G$ is obtained from $G^\star$ by 
keeping the vertex $\ell^\star$ (with its adjacent edges) but renaming it $\ell_j$, 
and adding $d-2$ vertices $\ell_1, \dots, \ell_{j-1}, \ell_{j+1}, \dots, \ell_{d-1}$, 
which form an independent set of size $d-2$, 
and all connected only to $\ell_j$.
\end{itemize}
In particular, $G$ always contains at least one vertex with exactly the same old neighbors as $\ell^\star$ in $G^\star$; call such vertices \emph{copies} of $\ell^\star$. Moreover, new vertices of $G$ which are not copies of $\ell^\star$ are pendant vertices incident to a copy of $\ell^\star$. 

With this remark, it becomes clear that distances between old vertices are the same in $G$ and $G^\star$. 
Moreover, the path between any two old leaves $\ell$ and $\ell'$ in $\tau$ also matches the path between $\ell$ and $\ell'$ in $\tau^\star$, so that we have $d_G(\ell,\ell')=d_{G^\star}(\ell,\ell')=\jump(\tau^\star,\ell,\ell')+1=\jump(\tau,\ell,\ell')+1$ as claimed. 
When $\ell$ and $\ell'$ are both new vertices, their distance $d_G(\ell,\ell')$ is either $1$ or $2$, depending on whether the corresponding marker vertices in $\Gamma_v$ are connected or not. Thus, in this case also, we have $d_G(\ell,\ell')=\jump(\tau,\ell,\ell')+1$. 
The interesting case is when $\ell$ is a new vertex and $\ell'$ an old vertex. 
Again, we proceed by case analysis.
Denote by $p$ the path from $\ell$ to $\ell'$ in $\tau$ and by $p^\star$ the path from $\ell^\star$ to $\ell'$ in $\tau^\star$. 
The path $p$ is obtained from $p^\star$ by replacing the first edge $(\ell^\star,u)$ by the two edges $(\ell, v), (v,u)$. (Recall that $u$ is the only nonleaf node of $\tau$ adjacent to $v$, corresponding to the marker vertex $x$ of $\Gamma_v$.)
\begin{itemize}
 \item Assume first that $\ell$ is a copy of $\ell^\star$. 
 Note that this happens when $\Gamma_v$ is a clique, 
 or when $\Gamma_v$ is a star with $\ell$ attached to the center of $\Gamma_v$, or when $\Gamma_v$ is a star with $x$ the center of the star. 
 Since $\ell$ is a copy of $\ell^\star$, of course $d_G(\ell,\ell')=d_{G^\star}(\ell^\star,\ell')$. 
 On the other hand, in all cases, the marker vertices of $\Gamma_v$ attached to $\ell$ and $u$ are adjacent. Therefore, we have $\jump(\tau^\star,\ell^\star,\ell')=\jump(\tau,\ell,\ell')$, and it follows that $d_G(\ell,\ell')=\jump(\tau,\ell,\ell')+1$. 
 \item The last case to consider is when $\Gamma_v$ is a star with $x$ an extremity of the star, and $\ell$ attached to another extremity of the star. 
 In this case, $p$ has one more jump than $p^\star$, since the marker vertices to which $x$ and $\ell$ are attached are not adjacent in $\Gamma_v$. 
 On the other hand, the only neighbor of $\ell$ in $G$ is the leaf of $\tau$ attached to the center of $\Gamma_v$, previously denoted $\ell_j$. 
 Since $\ell_j$ is a copy of $\ell^\star$, we have 
 $d_G(\ell,\ell') = 1 + d_G(\ell_j,\ell') = 1+d_{G^\star}(\ell^\star,\ell')$, which gives $d_G(\ell,\ell')=\jump(\tau,\ell,\ell')+1$ as desired. \qedhere
\end{itemize}
\end{proof}

\subsection{Clique-star trees as a labeled combinatorial class}
In \cref{sec: split}, we have seen that DH graphs are in bijection with reduced clique-star trees.
We recall that the latter are nonplane unrooted trees.
To use the symbolic method and tools of analytic combinatorics, it is more convenient to deal with rooted trees.
Starting from a DH graph with vertex set $\{0,1,\dots,n\}$,
we consider the reduced clique-star tree associated with it by \cref{prop:SplitTreeBij}
and see the leaf with label $0$ as the root.
\begin{definition}
\label{def:DH_tree}
 A \emph{distance-hereditary tree} (DH-tree for short) of size $n \ge 2$ is a reduced clique-star tree with $n+1$ leaves labeled from $0$ to $n$, where the leaf $0$ is seen as the root, therefore called the \emph{root-leaf}. 
\end{definition}
 
 By construction, DH-trees of size $n$ are in bijection with DH graphs with vertex set $\{0,1,\dots,n\}$.
 Most of the time, we forget the root-leaf and think at the tree as rooted in the internal node to which the root-leaf is attached;
 this node is referred to as root-node below.
 The root-leaf is represented by the symbol $\perp$ in pictures.

Having broken the symmetry when selecting a root, 
a node $v$ decorated with a star can be of two types.
\begin{itemize}
\item Either the path from $v$ to the root\footnote{Root-node or root-leaf, equivalently, unless $v$ is the root-node; in this latter case, the type of $v$ is defined in the same way considering the path (of length $1$) from $v$ to the root-leaf.} exits $v$ through an edge attached to {\em an extremity} of the star $\Gamma_v$.
In this case, we say that $v$ is of type $\SSX$. Note that one of the children of $v$ is attached to the center
of the star. We see this child as distinguished.
\item Or the path from $v$ to the root exits $v$ through the edge attached to {\em the center} of the star.
In this case, we say that $v$ is of type $\SSC$. Note that all children of $v$ are attached to extremities of the star
so that there is no distinguished child in this case.
\end{itemize}
A node decorated with a clique is of type $\KK$.

With this in mind, and recalling the conditions of \cref{{dfn:reduced_clique_star_tree}}, 
one can describe DH-trees directly as follows.
A DH-tree is a nonplane rooted tree $T$ such that
\begin{enumerate}
\label{def:DH_tree_Concrete}
\item $T$ has $n$ leaves labeled $1,\dots,n$;
\item internal nodes of $T$ (including the root)
  carry decorations, called \emph{types}, taken from the set $\{\KK, \SSC, \SSX\}$;
\item every node of type $\KK$ has at least $2$ children, none of which can be of type $\KK$;
\item every node of type $\SSC$ has at least $2$ children, none of which can be of type $\SSC$;
\item every node of type $\SSX$ has at least $2$ children, one of which is distinguished;
the distinguished child cannot be of type $\SSX$, while other are forbidden to be of type $\SSC$.
\label{def:DH_tree_Concrete2}
\end{enumerate}
\cref{fig:Example_DHtree} shows an example of DH-tree. 

\begin{figure}[htbp]
\includegraphics[width=10cm]{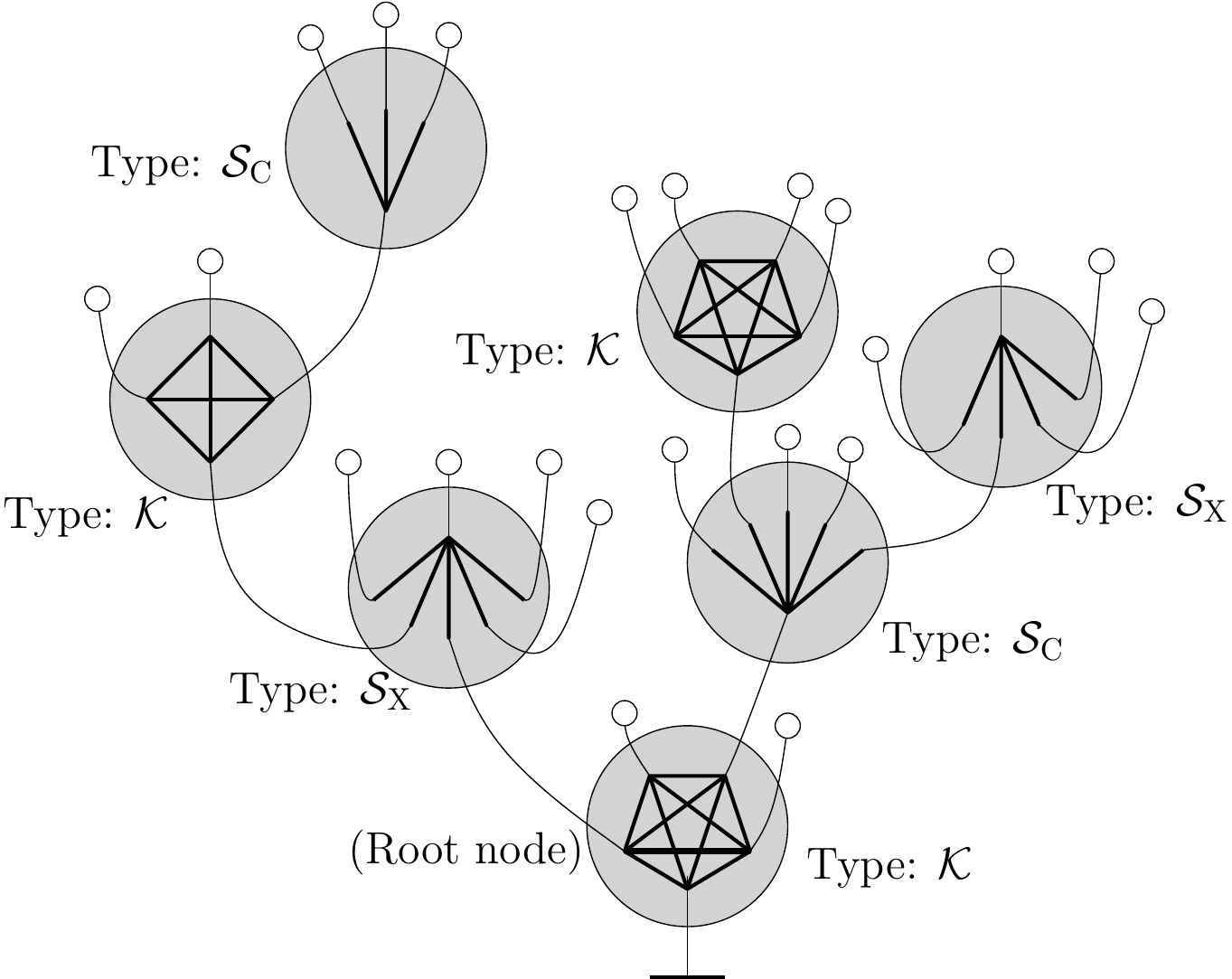}
\caption{A DH-tree of size $n=22$, omitting the labels of the leaves for readability. }
\label{fig:Example_DHtree}
\end{figure}

We will now translate this description into the framework of labeled combinatorial classes 
(see \cite{Violet} for an introduction). We recall that $+$ is used for the disjoint union of combinatorial classes;
$\mathcal A \times \mathcal B$ is the set of pairs $(a,b)$ where $a$ is in $\mathcal A$ and $b$ in $\mathcal B$
(with the convention that the label sets of $a$ and $b$ are disjoint; we refer to \cite{Violet}
for details on how to deal with labelings in combinatorial classes).
Also, if $\mathcal C$ is a combinatorial class with no element of size $0$,
then $\Set(\mathcal C)$ is the class of (unordered) sets of elements of $\mathcal C$.
An index on $\Set$ indicates restrictions on the number of elements in the set.

We say that a DH-tree is of type $t$ if its root-node is of type $t$. We let  
$\mDK$ (resp. $\mDSC$, $\mDSX$)
be the (labeled) combinatorial class of DH-trees of type $\mathcal{K}$ 
(resp. $\mathcal{S}_{\mathrm{C}}$, $\mathcal{S}_{\mathrm{X}}$). 
As usual, we use the symbol $\mathcal Z$ to represent the trivial tree reduced to one vertex (which is a leaf).

\begin{proposition}[Chauve-Fusy-Lumbroso\footnote{The equation given for $\mDSX$ in 
\cite[Theorem 3]{ChauveFusyLumbroso} is different from the one given here.
The one given here can however be found in the proof of \cite[Theorem 3]{ChauveFusyLumbroso}.}
\cite{ChauveFusyLumbroso}]
  
The combinatorial classes $\mDK,\mDSC,\mDSX$ have the following specification:
\begin{equation}
\label{eq:SpecifDH}
\begin{cases}
  \mDK =& \Set_{\geq 2}(\mathcal Z +\mDSC+\mDSX);\\
\mDSC=& \Set_{\geq 2}(\mathcal Z+ \mDK +\mDSX);\\
\mDSX=& (\mathcal Z+ \mDK +\mDSC)\times \Set_{\geq 1}(\mathcal Z + \mDK+\mDSX).
\end{cases}
\end{equation}
The class $\mD$ of all DH-trees is simply the disjoint union of the three classes above, 
\emph{i.e.}
\[
\mD=  \mDK+ \mDSC + \mDSX. 
\]
\end{proposition}

\subsection{Singularity analysis of the specification}
\label{ssec:singularity_analysis_NonMarked}
We associate to each combinatorial class of DH-trees
 a generating function $D=D(z)$, $\DK=\DK(z)$,
 $\DSC=\DSC(z)$ and $\DSX=\DSX(z)$:
$$
D=\sum_{T \in \mD} \frac{z^{|T|}}{|T|!}\ ,\quad
\DK=\sum_{T \in \mDK} \frac{z^{|T|}}{|T|!}\ ,\quad
\DSC=\sum_{T \in \mDSC} \frac{z^{|T|}}{|T|!}\ ,\quad
\DSX=\sum_{T \in \mDSX} \frac{z^{|T|}}{|T|!}\ .
$$
By loose estimates on the number of DH-trees,
 it is easy to see that each of the above series has a positive radius of convergence.
A key step in the proof of our main theorem will be given by the singularity analysis of the above series. A similar analysis is provided in \cite{ChauveFusyLumbroso} in the \emph{unlabeled} case, we here give all the details of the labeled case.

We first note that using Eqs. \eqref{eq:SpecifDH} and an immediate induction on $i \ge 0$, we have $[z^i]\DK=[z^i]\DSC$
for all $i\ge 0$, \emph{i.e.}  $\DK=\DSC$ as formal power series.
We will therefore drop $\DSC$ and use only $\DK$.
Eqs. \eqref{eq:SpecifDH} yield:
\begin{equation}
\label{eq:SpecifDH_reduced}
\begin{cases}
  \DK  =& \exp_{\geq 2}(z+\DK+\DSX);\\
\DSX=& (z+ 2\DK)\exp_{\geq 1}(z+ \DK +\DSX)\ ,
\end{cases}
\end{equation}
where $\exp_{\geq r}(y)=\sum_{\ell \geq r}y^\ell/\ell !$.

The system \eqref{eq:SpecifDH_reduced} satisfies the assumptions of  the Drmota--Lalley--Woods Theorem
(see \cite[Theorem A.6]{Nous3}\footnote{More classical references
for variants of this theorem are \cite[Section VII.6]{Violet} and \cite[Section 2.2.5]{DrmotaRandomTrees}, but the first one assumes that we have a polynomial system,
while the second one has a different well-posedness condition, which is not satisfied here
 (and uses extra parameters which are not needed here).}).
It follows that the series $\DK,\DSX$ have
the same radius of convergence $\rho$ and both have a square-root singularity at $\rho$. 
Moreover they are {\em $\Delta$-analytic},
meaning that they are defined and analytic on some set of the form
\[\{z \in \mathbb C, |z| < R_1 \text{ and }|\Arg(z-\rho)|>\theta\},\]
for some $R_1>\rho$ and $\theta>0$, where $\Arg$ is the principal determination of the logarithm. The notion of $\Delta$-analyticity is standard in analytic combinatorics,
see \cite[Chapter VI]{Violet}.
Let us introduce an auxiliary series
\begin{equation}\label{eq:defF}
F(z):=\exp_{\ge 1} (z+\DK+\DSX).
\end{equation}

\begin{lemma}
We have
\begin{equation}
\label{eq:KS_with_F}
\begin{cases}
  \DK=&\frac12 \left(\frac{F}{1+F} -z\right);\\
\DSX=&\frac{F^2}{1+F}.
\end{cases}
\end{equation}
\end{lemma}
\begin{proof}
Using $F$, we can rewrite the system \eqref{eq:SpecifDH_reduced} as
\[
\begin{cases}
  \DK  =& F-(z+\DK+\DSX) ;\\
\DSX=& (z+ 2\DK)\, F.
\end{cases}
\]
We solve this linear system for $\DK$ and $\DSX$, seeing $F$ as a parameter.
This gives the formulas of the lemma.
\end{proof}

\begin{proposition}
\label{prop:expansion_F}
The series $F$ is $\Delta$-analytic at $\rho$ and admits the following singular expansion around $\rho$:
\begin{equation}
F(z)=F(\rho)- \gamma_F \sqrt{1-z/\rho} +O\big(1-z/\rho\big),\label{eq:expansion_F}
\end{equation}
where
\begin{itemize}
\item $F(\rho)=\frac{\sqrt3-1}2$ is the unique positive root of $\ 2 F(\rho)^2+2 F(\rho)-1=0$; 
\item $\gamma_F=\frac{\sqrt2 (1+\sqrt3)^2}{4\sqrt{3+2\sqrt{3}}} \sqrt{\rho}$. 
\end{itemize}
\end{proposition}
The expression for $\gamma_F$ is computed in the companion Maple worksheet. This also holds for other constants $\gamma_H,\gamma_{H,2c},\gamma_{E,3\ell}$ arising later.

Throughout the paper, when a series $S$ has a square-root singularity,
we denote by $\gamma_S$ the coefficient of the square-root term in the singular expansion
of $S$ near its radius of convergence, with the same sign convention as above.
Also, for a variable $x$ and a (multivariate) function $A(x,\dots)$, we denote by $A_x$ the partial derivative of $A$ w.r.t. $x$. 
\begin{proof}
By \cref{eq:defF} the series $F$ is $\Delta$-analytic at $\rho$ and has a square-root singularity at $\rho$, therefore the expansion of $F$ around $\rho$ is given by \cref{eq:expansion_F} 
for some $F(\rho),\gamma_F$ which are to be determined. 
In addition, since $F$ is a series in $z$ with nonnegative coefficients, the transfer theorem ensures that $\gamma_F>0$. 

Thanks to Eqs. \eqref{eq:KS_with_F} one can eliminate $\DK$ and $\DSX$ in Eq. \eqref{eq:defF}. We obtain 
 that $F$ is the solution of the equation $F=G(z,F)$,
where 
\begin{equation}
G(z,w)=\exp_{\ge 1} \Bigg[z+\frac12 \left(\frac{w}{1+w} -z\right)+\frac{w^2}{1+w}\Bigg].\label{eq:EquationG_F}
\end{equation}

Plugging \cref{eq:expansion_F} into $F=G(z,F)$ and comparing the expansions
of both sides show that necessarily 
\begin{equation}\label{eq:charSys}
 F(\rho)=G(\rho,F(\rho)) \quad \text{and} \quad G_w(\rho,F(\rho))=1.
\end{equation}
(These equations are usually referred to as the characteristic system \cite[Section VII.4]{Violet}.)

Observing that $$G_w(z,w)= (1-\tfrac1{2(1+w)^2}) (1+G(z,w)),$$
the characteristic system yields the following equation
\begin{equation}\label{eq:Frho}
 2 F(\rho)^2+2 F(\rho)-1=0,
\end{equation}
whose only positive solution is $\frac{\sqrt3-1}2$.

Using that $F(\rho)=\frac{\sqrt3-1}2$, we can solve for $\rho$ the first of Eqs. \eqref{eq:charSys}, giving an explicit expression for $\rho$ and the numerical estimate $\rho \approx 0.1597$ (see Maple worksheet). 
Furthermore, using the Singular Implicit Functions Lemma \cite[Lemma VII.3]{Violet}, the constant $\gamma_F$ is given by
\[
\gamma_F=\sqrt{\frac{2 \rho G_z(\rho,F(\rho))}{G_{ww}(\rho,F(\rho))}}.\]
We note that $G_z(z,w)=\frac12(1+G(z,w))$, so that $2 G_z(\rho,F(\rho))=1+F(\rho)$.
Thus we have
\[
\gamma_F=\sqrt{\frac{1+F(\rho)}{G_{ww}(\rho,F(\rho))}}\, \sqrt{\rho} = \frac{(1+\sqrt3)^2}{2\sqrt{6+4\sqrt{3}}} \sqrt{\rho},
\]
where the last equality is justified in the companion Maple worksheet.
\end{proof}

\begin{remark}
Since $F$ is the solution of the implicit equation $F=G(z,F)$, it is tempting to use
the smooth implicit-function schema \cite[Theorem VII.3]{Violet} to find its dominant singularity
and asymptotic expansion.
We can however not proceed like this since the expansion of $G$ contains negative coefficients,
contradicting \cite[Hypothesis $(\bm{I_2})$ p.~468]{Violet}. 
This explains the indirect path used here. 
In short, the system \eqref{eq:SpecifDH_reduced} has the advantage of having nonnegative coefficients:
it is used to prove without effort that all series have square-root singularities.
On the other hand, $F$ is defined by a single equation,
giving simpler computations to determine explicitly the coefficients in its singular expansion.
\end{remark}

In the sequel we also need the asymptotic expansion of $\DK$, $\DSX$ and its derivative.
Using $\DSX=\frac{F^2}{1+F}$, \cref{prop:expansion_F} and singular differentiation (\cite[Theorem VI.8]{Violet}) 
we get
that $\DSX$ and $\DSX'$ are $\Delta$-analytic and that
\begin{align}
\DSX(z)&=\DSX(\rho) - \gamma_X \sqrt{1-z/\rho} +O\big(1-z/\rho\big),\label{eq:DevSX}\\
\DSX'(z)&=\frac{ \gamma_X }{2\rho}(1-z/\rho)^{-1/2} +O\big(1\big),\label{eq:DevSXprime}
\end{align}
where:
\begin{itemize}
\item $\DSX(\rho)=\frac{F(\rho)^2}{1+F(\rho)}=\frac{2-\sqrt{3}}{1+\sqrt3}$;
\item $\gamma_X = \frac{\partial \DSX}{\partial F}(F(\rho)) \, \gamma_F 
=\left( 1-\frac{1}{(1+F(\rho))^2}\right) \, \gamma_F = \frac{\sqrt6}{2 \sqrt{3+2\sqrt{3}}}\, \sqrt{\rho}$.
\end{itemize}
Similarly we obtain 
that $\DK$ is $\Delta$-analytic and that
\begin{align}
\DK(z)&=\frac{F(\rho)}{2(1+F(\rho))} -\frac{\rho}{2} - \gamma_K \sqrt{1-z/\rho} +O\big(1-z/\rho\big),\label{eq:DevSC}
\end{align}
where $\gamma_K= \frac{\partial \DK}{\partial F}(F(\rho)) \, \gamma_F 
= \frac{1}{2(1+F(\rho))^2} \, \gamma_F = \frac{1}{ \sqrt{6+4\sqrt 3}} \sqrt\rho$.

Summing, we have for $D$ the following expansion:
\begin{equation}
D(z)=D(\rho) - \gamma_D  \sqrt{1-z/\rho} +O\big(1-z/\rho\big),\label{eq:DevD}
\end{equation}
where $\gamma_D=\frac{2+\sqrt 3}{\sqrt{6+4\sqrt 3}} \sqrt\rho$.

\section{DH-trees with a marked leaf}\label{Sec:TreesWithMarkedLeaves}

In this section, we introduce and analyze combinatorial classes of DH-trees
with a marked leaf and certain conditions. 
This is a first step in the proof of \cref{th:GrosTheoreme} for unconstrained DH graphs ($f=d$).
Indeed,
the classes studied here are building blocks in the decomposition of trees with several
marked leaves, which we will consider in the next section in order to study distance matrices
of uniform random DH graphs.

\subsection{A combinatorial system of equations for DH-trees with a marked leaf}
\label{ssec:combinatorial_dec_one_marked_leaf}

\begin{definition}
Let $T$ be a DH-tree and $v$ a vertex of $T$ different from its root (note that $v$ may be a leaf). 
Let $p$ be the parent of $v$ in $T$. 
Informally, the \emph{cotype} of $v$ is the type that $p$ would have if the root were in $v$. 
More precisely, 
\begin{itemize}
 \item if $p$ is of type $\KK$, then $v$ is of cotype $\KK$; 
 \item if $p$ is of type $\SSC$, then $v$ is of cotype $\SSX$; 
 \item if $p$ is of type $\SSX$ and $v$ is the distinguished child of $p$, then $v$ is of cotype $\SSC$;
 \item if $p$ is of type $\SSX$ and $v$ is not the distinguished child of $p$, then $v$ is of cotype $\SSX$. 
\end{itemize}
\end{definition}
The reader is invited to look at the example of \cref{fig:Example_cotype}.

\begin{figure}[htbp]
\includegraphics[width=5cm]{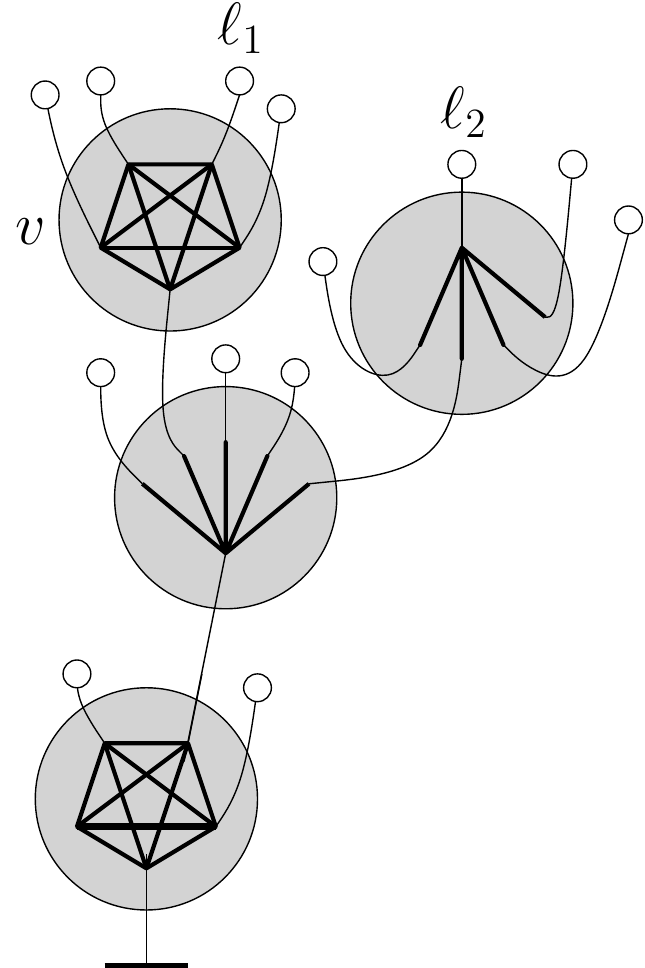}
\caption{In this DH-tree, the leaf $\ell_1$ has cotype $\KK$, the leaf $\ell_2$ has cotype $\SSC$ 
and the node $v$ has cotype $\SSX$ (its parent has type $\SSC$
 but if we would re-root the tree in $v$,
it would have type $\SSX$).}
\label{fig:Example_cotype}
\end{figure}
\begin{definition}
Let $a,b \in \{\KK,\SSC,\SSX\}$. 
We define $\mDab$ as the set of DH-trees with one marked leaf of cotype $b$, whose root-node is of type $a$. 
\end{definition}
We further set, for $a,b \in \{\KK,\SSC,\SSX\}$,
\begin{align*}
\mDap&=\mDaK + \mDaC +  \mDaX;\\
\mDpb&=\mDKb + \mDCb + \mDXb.
\end{align*}
In other words, a bullet as index (resp. exponent) denotes an unconstrained type of the root-node (resp. cotype of the marked leaf).
\medskip

We now introduce the following statistics.
Let $(T,\ell)$ be a DH-tree with one marked leaf. 
We denote $\jump(T,\ell)$ the number of jumps on the path from the marked leaf $\ell$ to the root-leaf in $T$
(in particular, the root-node might be a jump in this path, see~\cref{fig:Schema_ProofDXC}).

We consider (exponential) bivariate generating series of families of DH-trees
 with one marked leaf with respect to the size
(variable $z$) and to the number of jumps (variable $u$).
Namely, for $a,b \in  \{\bullet,\KK,\SSC,\SSX\}$,
\[\Dab (z,u) = \sum_{(T,\ell) \in \mDab} \frac{z^{|T|}}{|T|!} u^{\jump(T,\ell)}.\]
We take the convention that for a DH-tree with a marked leaf $(T,\ell)$, 
its size $|T|$ is the number of unmarked leaves of $T$ 
(in other words, the marked leaf is not counted).

\begin{proposition}
\label{prop:system_markedleaf}
The bivariate series $\Dab(z,u)$ for $a,b \in  \{\KK,\SSC,\SSX\}$ are solutions of the following systems of equations:
\begin{align}
\label{Syst:DaK}
& \begin{cases}
\DKK=(1+\DCK+\DXK) \, \exp_{\ge 1}(\DSC+\DSX+z);\smallskip\\
\DXK=(\DCK+\DKK) \, \exp_{\ge 1}(\DSX+\DK+z) \\
\qquad \qquad \qquad+ u \cdot (\DKK+\DXK) \, (\DSC+\DK+z)  \, \exp(\DSX+\DK+z);\smallskip\\
\DCK=(\DKK+\DXK) \, \exp_{\ge 1}(\DSX+\DK+z);
\end{cases}\smallskip
  \\
  \label{Syst:DaX}
&  \begin{cases}
\DKX=(\DXX+\DCX) \, \exp_{\ge 1}(\DSC+\DSX+z); \smallskip\\
\DXX=(\DCX+\DKX) \, \exp_{\ge 1}(\DSX+\DK+z)\\
\qquad \qquad \qquad+u \cdot (1+\DKX+\DXX) \, (\DSC+\DK+z) \,\exp(\DSX+\DK+z);\smallskip\\
\DCX=(1+\DKX+\DXX)\, \exp_{\ge 1}(\DSX+\DK+z);
\end{cases} \smallskip\\
  \label{Syst:DaC}
& \begin{cases}
\DKC=(\DCC+\DXC) \, \exp_{\ge 1}(\DSC+\DSX+z);\smallskip\\
\DXC=(1+\DCC+\DKC) \, \exp_{\ge 1}(\DSX+\DK+z)\\
\qquad \qquad \qquad+u \cdot (\DKC+\DXC)\, (\DSC+\DK+z) \,\exp(\DSX+\DK+z);\smallskip\\
\DCC=(\DKC+\DXC) \, \exp_{\ge 1}(\DSX+\DK+z).
\end{cases}
\end{align}
\end{proposition}

\begin{proof}
We prove in details the case of $\DXC$ (second equation in the system \eqref{Syst:DaC}). 
The eight other equations are proved in a similar way.

Hence we consider a DH-tree $T$ of type $\SSX$ with a marked leaf of cotype $\SSC$. We can decompose $T$ as a root-node $v$, to which several subtrees are attached. (Notations are summarized in \cref{fig:Schema_ProofDXC}.) 
The subtrees attached to $v$ are:
\begin{itemize}
\item The subtree $T'$ containing the marked leaf. In order to keep track of variable $u$ we need to consider two cases. 
\begin{itemize}
\item First, $T'$ may be attached to the center of $\Gamma_v$ (Case A of \cref{fig:Schema_ProofDXC}). 
In this case, there is no jump in $v$. 
Also, $T'$ (if not reduced to a leaf) is of type $\KK$ or $\SSC$. 
Note that $T'$ may also be reduced to a leaf (hence, the marked leaf), since a leaf attached to the center of $\Gamma_v$ has indeed cotype $\SSC$. 
\item Otherwise, $T'$ is attached to an extremity of $\Gamma_v$ (Case B of \cref{fig:Schema_ProofDXC}).
In this case, there is a jump in $v$.
Here, $T'$ can be of type $\KK$ or $\SSX$, and $T'$ cannot be reduced to a leaf since a leaf attached to an extremity of $\Gamma_v$ would have cotype $\SSX$.
\end{itemize}
\item Attached to every (other) extremity of $v$ one has a tree of type $\KK$ or $\SSX$ or a leaf. 
\item Attached to the center of $v$ (if this is not where $T'$ is attached, \emph{i.e.} in Case B in \cref{fig:Schema_ProofDXC}) there is a tree of type $\KK$ or $\SSC$ or a leaf.
\end{itemize}

\begin{figure}[htbp]
\includegraphics[width=10cm]{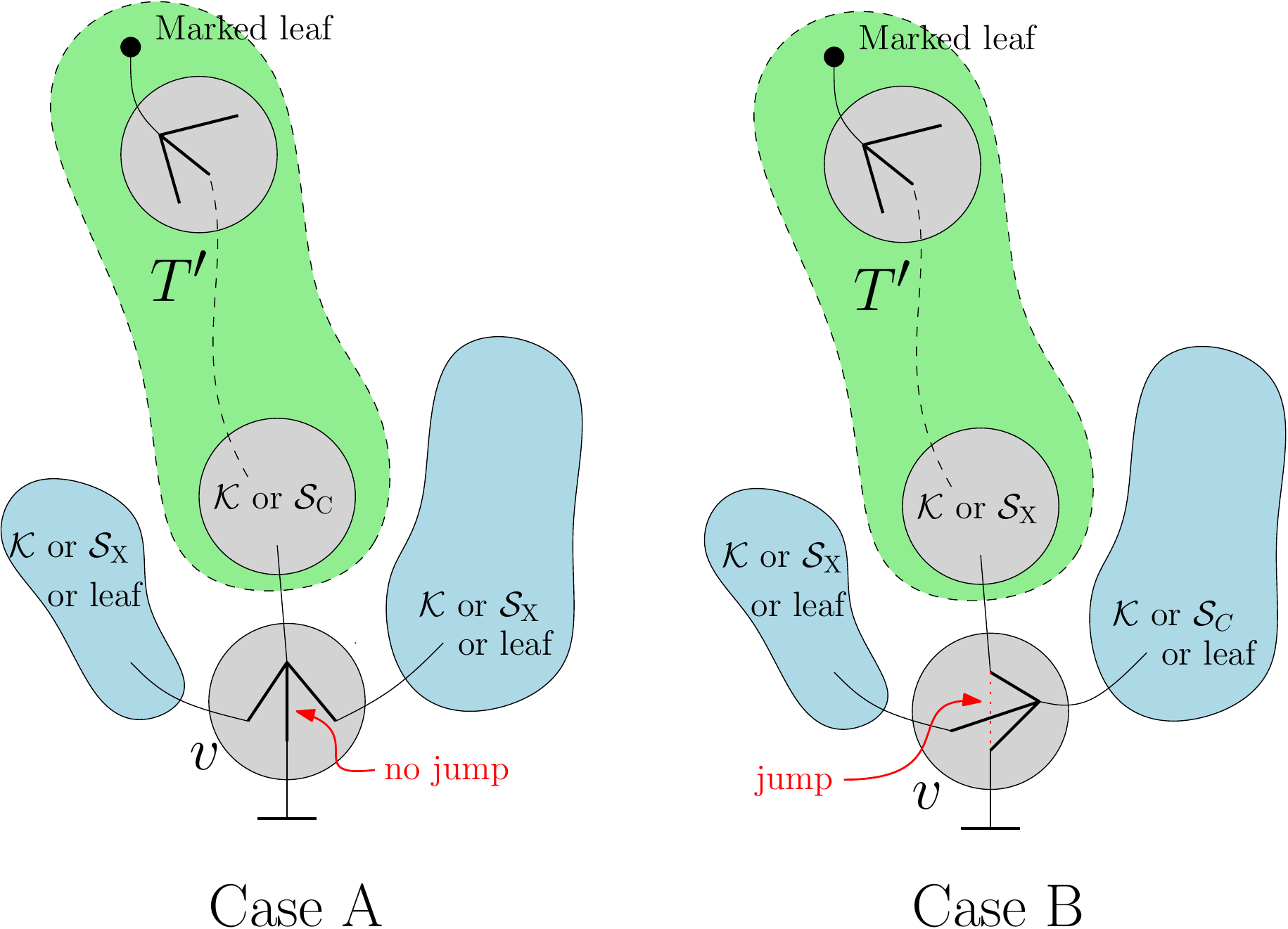}
\caption{Decomposition of a tree in $\mDXC$.}
\label{fig:Schema_ProofDXC}
\end{figure}

We now translate this decomposition on generating functions. 
According to the case analysis above, $T'$ is counted by:
\begin{itemize}
\item in Case A: $1+\DCC+\DKC$ (since a single marked leaf, counted by $1$, is allowed for $T'$);
\item in Case B: $\DKC+\DXC$.  
\end{itemize}
The remaining trees attached to $v$ are counted by:
\begin{itemize}
\item in Case A: $\exp_{\ge 1}(\DSX+\DK+z)$ (because they form a nonempty unordered sequence of trees in $\mDSX + \mDK + \mathcal{Z}$)
\item in Case B: $(\DSC+\DK+z) \,\exp(\DSX+\DK+z)$ (with a distinguished tree attached to the center which is either in $\mDSC$ or in $\mDK$ or a leaf, and other trees which form an unordered sequence of trees in $\mDSX + \mDK +\mathcal{Z}$).
\end{itemize}
Finally, a factor $u$ appears in Case B to take into account the jump in $v$.
Hence
\begin{multline*}
\DXC=(1+\DCC+\DKC) \, \exp_{\ge 1}(\DSX+\DK+z)\\
+u \cdot (\DKC+\DXC)\, (\DSC+\DK+z) \,\exp(\DSX+\DK+z). \qedhere
\end{multline*}
\end{proof}

\subsection{Resolution of the system}
\label{ssec:sol_Dab}
Recall (see \cref{ssec:singularity_analysis_NonMarked}) that 
\[\DSC=\DK, \quad\exp_{\ge 1}(\DK+\DSX+z)=F\ \text{ and }\DSC+\DK+z=2\DK+z=\tfrac{F}{1+F}.\]
It implies $ (\DSC+\DK+z) \,\exp(\DSX+\DK+z)=F$, allowing to simplify the systems \labelcref{Syst:DaK,Syst:DaX,Syst:DaC} as follows: 
\begin{align}
\label{Syst:DaK2}
& \begin{cases}
\DKK=(1+\DCK+\DXK) \, F;\smallskip\\
\DXK=(\DCK+\DKK) \, F 
+ u \, F\cdot (\DKK+\DXK) ;\smallskip\\
\DCK=(\DKK+\DXK) \, F;
\end{cases}\smallskip
  \\
  \label{Syst:DaX2}
&  \begin{cases}
\DKX=(\DXX+\DCX) \, F; \smallskip\\
\DXX=(\DCX+\DKX) \, F+u \, F\cdot (1+\DKX+\DXX) ;\smallskip\\
\DCX=(1+\DKX+\DXX)\, F;
\end{cases} \smallskip\\
  \label{Syst:DaC2}
& \begin{cases}
\DKC=(\DCC+\DXC) \, F;\smallskip\\
\DXC=(1+\DCC+\DKC) \, F+u \, F \cdot (\DKC+\DXC);\smallskip\\
\DCC=(\DKC+\DXC) \, F.
\end{cases}
\end{align}
Solving the system\footnote{see Maple worksheet.} gives the following formulas (put under a suitable form
for the subsequent asymptotic analysis): 
\begin{align}
\DKK&=\frac{F}{F+1}+\frac{F^2}{(1+F)(1-2F)-Fu}; \label{eq:sol_DKK}\\
\DXX&=\frac{-1}{F+1}+\frac{(1-F)^2}{(1+F)(1-2F)-Fu}; \label{eq:sol_DXX}\\
\DCC&=\frac{F^2}{(1+F)(1-2F)-Fu}; \label{eq:sol_DCC}\\
\DKX=\DXK&=\frac{-F}{F+1}+\frac{F\, (1-F)}{(1+F)(1-2F)-Fu}; \label{eq:sol_DKX}\\
\DKC=\DCK&=\frac{F^2}{(1+F)(1-2F)-Fu}; \label{eq:sol_DKC}\\
\DXC=\DCX&=\frac{F(1-F)}{(1+F)(1-2F)-Fu}.\label{eq:sol_DXC}
\end{align}
\begin{remark}
Symmetries $D_a^b=D_b^a$ in above equations can easily be explained combinatorially. Indeed, we can see a DH-tree with root-leaf $r$ and a marked leaf $\ell$ as a DH-tree rooted in $\ell$ where $r$ is a marked leaf; doing so, the type of the (old) root becomes the cotype of $r$ and the cotype of $\ell$ becomes the type of the (new) root.
\end{remark}

Recalling that $F$ depends only on $z$ (not on $u$), in each case, the series can be written under the form
\begin{equation}
\label{eq:formeDab}
\Dab=Q_a^b(z)+\frac{M_a^b(z)}{1-u\, H_a^b(z)},
\end{equation}
where $Q_a^b$, $M_a^b$ and $H_a^b$ are rational functions in $F$.
For example, looking at \cref{eq:sol_DKK}, we have 
\[Q_\KK^\KK=\frac{F}{F+1},\ 
M_\KK^\KK=\frac{F^2}{(1+F)(1-2F)}\text{ and }H_\KK^\KK=\frac{F}{(1+F)(1-2F)}.\]
Similar formulas are easily written for other $a,b \in \{\KK,\SSC,\SSX\}$,
looking at \cref{eq:sol_DXX,eq:sol_DCC,eq:sol_DKX,eq:sol_DKC,eq:sol_DXC}.

Interestingly (and we shall later use these remarks), $H_a^b=H=\frac{F}{(1+F)(1-2F)}$ is the same for all $a,b \in \{\KK,\SSC,\SSX\}$ and  $M_a^b$ factorizes as $M_a^b=\Lambda_a \Lambda_b$, where
\[ \Lambda_\KK=\Lambda_{\SSC}= \frac{F}{\sqrt{(1+F)(1-2F)}};\quad  \Lambda_{\SSX}= \frac{1-F}{\sqrt{(1+F)(1-2F)}}.\]
Recall from \cref{eq:Frho}, that $1-2 F(\rho)=  2 F(\rho)^2$  so that $(1+F(z))(1-2F(z)) \notin (-\infty,0)$ for $z$ close to $\rho$. Hence, the definitions of $\Lambda_\KK, \Lambda_{\SSC}$ and $\Lambda_{\SSX}$ make sense near $\rho$; we shall only use them in this domain. 

Moreover all these formulas immediately extend to the case where $a$ or $b$ or both is/are equal
to $\bullet$ (unconstrained type of the root-node or cotype of the marked leaf),
with the natural convention that
\begin{align*}
H_\bullet^b&= H_a^\bullet =H ;\\
M_\bullet^b &=M^b_\KK+M^b_{\SSC}+M^b_{\SSX},\\
\Lambda_\bullet &=\Lambda_\KK+\Lambda_{\SSC}+ \Lambda_{\SSX} ;
\end{align*}
and conventions similar to the second line for $M_a^\bullet$, $Q_a^\bullet$ and $Q_\bullet^b$.

Since $F$ has nonnegative coefficients and $2F(\rho)=\sqrt{3}-1<1$, the denominators of $Q_a^b(z)$, $M_a^b(z)$, $\Lambda_a^b(z)$ and $H(z)$
 are positive for $z$ in $[0,\rho]$ and thus,
 the series $Q_a^b$, $M_a^b$, $\Lambda_a^b(z)$ and $H$ all have radius of convergence
 $\rho$ and a square-root singularity in $\rho$, inherited from that of $F$.
Later (in the proof of Proposition \ref{prop:theoreme_local_limite_rTn})
 the function $H$ will play a particular role in the asymptotic analysis so let us now compute its  expansion at $\rho$:  
 \[H(z)=H(\rho) -\gamma_H \sqrt{1-z/\rho} +O(1-z/\rho),\]
where
\begin{align}
H(\rho)&=\frac{F(\rho)}{(1+F(\rho))(1-2F(\rho))}=1,\label{Hrho=1}\\
\gamma_H&=\frac{\partial H}{\partial F}(F(\rho)) \, \gamma_F =
\left( \frac1{3(1+F(\rho))^2} +\frac2{3(1-2F(\rho))^2}\right) \, \gamma_F  \label{eq:gamma_H}\\
& =\frac{ 2 \cdot (3-\sqrt3)}{\sqrt{6+4\sqrt{3}}\, (2-\sqrt3)^2} \sqrt{\rho}\notag 
\end{align}
whose numerical estimate is $\gamma_H \approx 3.9258$ (see Maple worksheet).

\section{$k$-point distances and induced subtrees}\label{Sec:TreesWithInducesSubtrees}

The goal of this section is to obtain the joint convergence in distribution of distances between marked leaves in a uniform DH-tree (see \cref{Coro:ConvergenceEnriched} below). 
This allows us to complete the proof of \cref{th:GrosTheoreme} in the case of unconstrained DH graphs
($f=d$).

\subsection{Marked leaves and induced subtrees}\label{ssec:51}
In this section, we consider DH-trees with $k$ marked leaves $(\ell_1,\dots,\ell_k)$ with the following convention.
\begin{definition}
\label{def:marked_DH_tree}
 A DH-tree of size $n$ with $k$ marked leaves is a nonplane rooted tree $T$ such that
\begin{itemize}
\item $T$ has $n$ nonmarked leaves labeled $1,\dots,n$;
\item additionally, $T$ has $k$ ordered leaves carrying marks $(\ell_1,\dots,\ell_k)$;
\item items ii) to v) p.~\labelcpageref{def:DH_tree_Concrete,def:DH_tree_Concrete2} are satisfied.
\end{itemize}
\end{definition}
Equivalently, it is a DH-tree of size $n+k$ where leaves with labels $n+1,\dots,n+k$
are seen as marked and get marks $\ell_1$, \dots, $\ell_k$, respectively.
These marked leaves are not counted in the size.
With this convention, the exponential generating series of DH-trees with $k$ marked leaves
is $D^{(k)}(z)$ (the $k$-th derivative of $D$).
\cref{prop:SplitTreeBij}
 is immediately rephrased as follows.
\begin{proposition}\label{prop:SplitTreeBij_rephrased}
Labeled DH graphs of size $n+k+1$ are in bijection with DH-tree
of size $n$ and $k$ marked leaves (when $n+k+1 \geq 3$).
\end{proposition}
To simplify notation, we write $\bm \ell=(\ell_1,\dots,\ell_k)$. 
\bigskip

We recall the definition of induced subtree.
\begin{definition}
Let $T$ be a DH-tree with $k$ marked leaves $\bm \ell$.
We call {\em essential vertices} of $T$ (w.r.t the marked leaves $\bm \ell$)
its root-leaf, its $k$ marked leaves $\bm \ell$ and 
their first common ancestors.
Then,
the subtree of $T$ induced by $\bm \ell$ is obtained as follows:
\begin{itemize}
\item its vertices are the essential vertices of $T$;
\item its genealogy (ancestor/descendant relation) is inherited from that of $T$.
\end{itemize}
\end{definition}

\cref{fig:InducedSubtree} illustrates this definition. We remark that the subtree of $T$ induced
by $k$ marked leaves $\bm \ell$ is naturally rooted at the vertex corresponding to the root-leaf of $T$.
This vertex is always of degree 1, and will be called root-leaf of the induced subtree.
\begin{figure}[htbp]
\includegraphics[width=10cm]{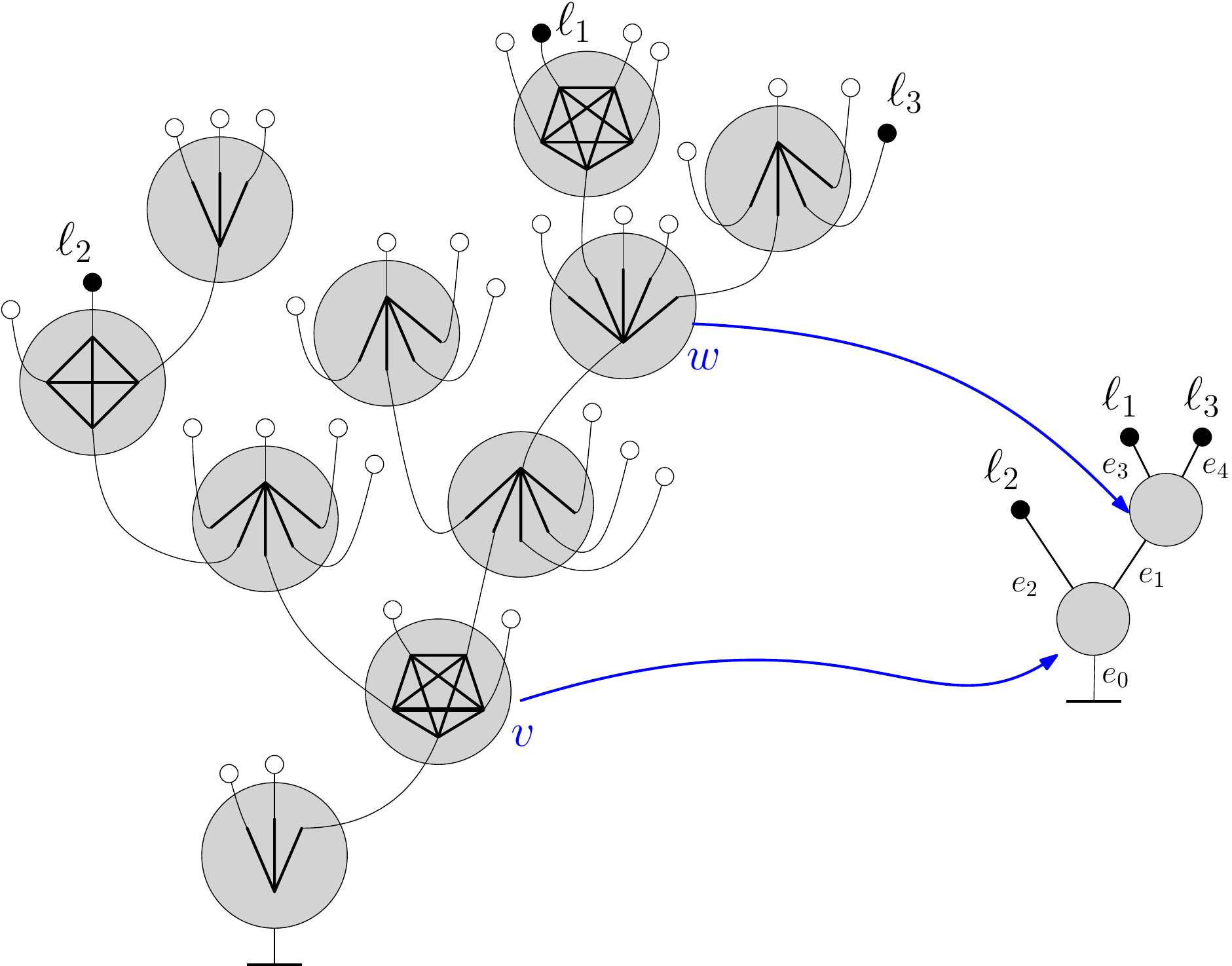}
\caption{Left: A DH-tree $T$ of size $n=28$ with $k=3$ marked leaves. The nodes $v$ and $w$ are the first common ancestors of $\ell_1,\ell_2$ and $\ell_3$.  Right: The subtree $t_0$ of $T$ induced by $(\ell_1,\ell_2,\ell_3)$. We have highlighted the correspondence between first common ancestors in $T$ and internal vertices of the induced subtree $t_0$.}
\label{fig:InducedSubtree}
\end{figure}

We now enrich the notion of induced subtrees to record the number of jumps along some paths of $T$.
Consider a DH-tree $T$ with $k$ marked leaves $\bm\ell$.
Let $t_0$ be the associated induced subtree. Each edge $e$ in $t_0$ corresponds 
to a path between two consecutive essential vertices of $T$. 
We define $\jump_e(T;\bm \ell)$ as the number of jumps of the path corresponding to $e$,
with the convention that essential vertices are not counted as jumps (but note that the root-node of $T$ can be a jump).
We call {\em enriched induced subtree} of $(T,\bm \ell)$ the induced subtree $t_0$,
with the quantities $\jump_e(T;\bm \ell)$ attached to its edges.
It will be convenient to fix for each tree $t$ with $k$ leaves
an enumeration $(e_0,e_1,\dots)$ of its edges 
such that $e_0$ is the edge adjacent to the root-leaf
(for instance a breath-first traversal of the tree with an arbitrary planar embedding). 
Then 
the enriched induced subtree of $(T,\bm \ell)$ can be written as a tuple 
$(t_0,a_0,\dots,a_m)$, where $t_0$ is the induced subtree of $(T,\bm \ell)$ 
and $a_i=\jump_{e_i}(T;\bm \ell)$. 
In the following we denote
$
r(T,\bm \ell):=(t_0,a_0,\dots,a_m).
$

\subsection{Combinatorial decomposition}
Recall that a $k$-proper tree is an (unrooted) nonplane tree with $k+1$ leaves
where each internal node has degree $3$ (with one leaf considered as the root-leaf and the other leaves denoted $\{\ell_1,\dots,\ell_k\}$). 
It is easily observed that a $k$-proper tree has $k+1$ leaves, $k-1$ internal vertices and $2k-1$ edges. 
It is also a standard fact (see, \emph{e.g.}, \cite{AldousCRTIII}) that the cardinality of the set of $k$-proper trees is exactly $(2k-3)!!$ ,
where we recall that $(2k-3)!!$ is the product of all odd positive integers less than or equal to $2k-3$.
Indeed, a $k$-proper tree can be obtained in a unique way from a $(k-1)$-proper tree by selecting
one of its $2k-3$ edges and grafting in the middle a new edge with a leaf $\ell_k$ at its extremity.

Let us fix a $k$-proper tree $t_0$.
We consider the following class of marked DH-trees.
\begin{definition}
\label{def:Dt}
We let $\mDt$ be the labeled combinatorial class of DH-trees $T$
with $k$ marked leaves $\bm \ell$ such that:
\begin{enumerate}
\item the subtree of $T$ induced by $\bm \ell$ is $t_0$;
\item no two essential vertices of $T$ are neighbors of each other.
\end{enumerate}
\end{definition}
Item ii) is a technical condition to have a nicer combinatorial decomposition
in \cref{eq:decompo_Dt0} below.

Recall that we have fixed an enumeration $(e_0,e_1,\dots,e_{2k-2})$ of the $(2k-1)$ edges of our $k$-proper tree $t_0$, 
in which $e_0$ is the edge attached to the root-leaf of $t_0$.
Consider the following multivariate generating series for $\mDt$: 
\[\Dt(z,u_0,\dots,u_{2k-2})=  \sum_{(T;\bm \ell) \in \mDt} \frac{z^{|T|}}{|T|!} u_0^{\jump_{e_0}(T,\bm \ell)}
\cdots u_{2k-2}^{\jump_{e_{2k-2}}(T,\bm \ell)}.\]

In order to compute the series $\Dt$ we introduce the following new classes of DH-trees. 
For $a,b,c \in \{\bullet, \KK,\SSX,\SSC\}$, let $\J_{a}^{b\,c}$ be the set of DH-trees $T$ with two (ordered) marked leaves such that
\begin{itemize}
\item the two marked leaves are children of the root-node;
\item if $T_1$ is a DH-tree of type $b$,
one can glue $T_1$ on the first 
marked leaf of $T$ (merging the marked leaf and the root-node of $T_1$)
without violating the adjacency restrictions defining DH-trees (conditions iii) to v) p.~\pageref{def:DH_tree_Concrete});
\item the same condition holds with gluing a DH-tree of type $c$ on the second marked leaf;
\item additionally, if $T_0$ is a DH-tree  with a marked leaf of cotype $a$,
one can glue $T$ on the marked leaf of $T_0$
without violating the adjacency restrictions defining DH-trees.
\end{itemize}
\begin{lemma}
\label{lem:J}
The generating function of $\J_{a}^{b\,c}$ is 
\begin{align}
\label{eq:J}
J_{a}^{b\,c}(z)& =  (\Ind_A+\Ind_B+\Ind_C) \exp(\DSX+\DK+z) +\Ind _{\KK \notin \{a,b,c\}}
\exp(\DSX+\DSC+z)\\ & \qquad \qquad + \Ind_{\SSC \notin \{a,b,c\}} (\DSC + \DK+z) \exp(\DK+\DSX+z) \notag\\
&=(\Ind_A+\Ind_B+\Ind_C+\Ind _{\KK \notin \{a,b,c\}}) (1+F) +\Ind_{\SSC \notin \{a,b,c\}} F, \notag
\end{align}
where
$$A=\{(a,b,c) \mid a \neq \SSX, b \neq \SSC, c\neq \SSC \},$$
$$B=\{(a,b,c) \mid b \neq \SSX, a \neq  \SSC,  c\neq \SSC \},$$
$$C=\{(a,b,c) \mid c \neq \SSX, a \neq  \SSC, b \neq \SSC \}.$$
\end{lemma}
\begin{proof}
We consider different cases depending on the type of the root-node.
The trees of $\J_{a}^{b\,c}$ having a root-node of type $\KK$ are counted by  $\Ind_{\KK \notin \{a,b,c\}}
\exp(\DSX+\DSC+z)$.
The ones having a root-node of type $\SSC$ are counted by $\Ind_A  \exp(\DSX+\DK+z)$.
Finally the ones having a root-node of type $\SSX$ are counted by $(\Ind_B+\Ind_C) \exp(\DSX+\DK+z)+ \Ind_{\SSC \notin \{a,b,c\}} (\DSC + \DK+z) \exp(\DK+\DSX+z)$ since the center of the star may be connected to the first marked leaf, to the second marked leaf or to neither of them. 

To conclude the proof, we use that $\DK=\DSC$, and  \cref{eq:defF,eq:KS_with_F}.
\end{proof}

For $0\leq i \leq 2k-2$ let $v_i$ (resp. $w_i$) be the vertex incident to $e_i$ in $t_0$ closest to (resp. farthest from) the root-leaf of $t_0$.  In particular, some $v_i$'s are equal to each other, $v_0$ is the root-leaf and some $w_i$ are leaves (see \cref{fig:PreuveDecompo_Dt0}, right).

If $w_i$ is not a leaf, let $s_i$ (resp. $g_i$) be the  smallest (resp. greatest) index of the edges from $w_i$ to its (two) children.

\begin{proposition}
\label{prop:decompo_Dt0}
We have 
\begin{equation}
\label{eq:decompo_Dt0}
\Dt(z,u_0,\dots,u_{2k-2})=  \sum_{\mathbf{(tp,ct)}\in \E}\prod_{i=0}^{2k-2} D_{tp_i}^{ct_i}(z,u_i) \prod_{i \atop
 w_i\text{ is an internal node}}  J_{ct_i}^{tp_{s_i} tp_{g_i}}(z)
\end{equation}
where
\begin{align*}\E=\{\mathbf{(tp,ct)}=(tp_i, ct_i)_{0\leq i \leq 2k-2} \mid tp_i,ct_i  \in\{\bullet, \KK, \SSX, \SSC \} &\text{ and } tp_i=\bullet \text{ iff } i=0 \\
  & \text{ and }  ct_i=\bullet \text{ iff }  w_i \text{ is a leaf }\}.
  \end{align*}
\end{proposition}
\begin{proof}
   We shall build a size-preserving bijection from $\mDt$ to the disjoint union
   $$\biguplus_{\mathbf{(tp,ct)}\in \E}\prod_{i=0}^{2k-2} \mD_{tp_i}^{ct_i} \prod  \J_{ct_i}^{tp_{s_i} tp_{g_i}}$$

   Let $(T,\bm \ell) \in \mDt$. Then $T$ is a DH-tree. 
   Let $\bar{v_i}$ (resp. $\bar{w_i}$) be the essential vertex of $T$ corresponding to $v_i$ (resp. $w_i$). 
   We set $tp_0=\bullet$, and $ct_i=\bullet$ when $w_i$ is a leaf of $t_0$.
   Otherwise, we denote by $ct_i$ the cotype of $\bar{w_i}$,
   and by $tp_i$ the type of the child of $\bar{v_i}$ which is the root of the subtree containing $\bar{w_i}$ (this type is well-defined thanks to item ii) of \cref{def:Dt}).
   We then have $(tp_i, ct_i)_{0\leq i \leq 2k-2}\in \E$.
   For example with $(T,\bm \ell)$ given in \cref{fig:PreuveDecompo_Dt0},
   \[tp_0=\bullet, \ tp_1 = tp_2 =tp_4= \SSX, \ tp_3 = \KK, \ ct_0 = \SSX, \ ct_1=\SSC,
   \ ct_2=ct_3=ct_4=\bullet.\] 
  
  \begin{figure}[htbp]
  \includegraphics[width=10cm]{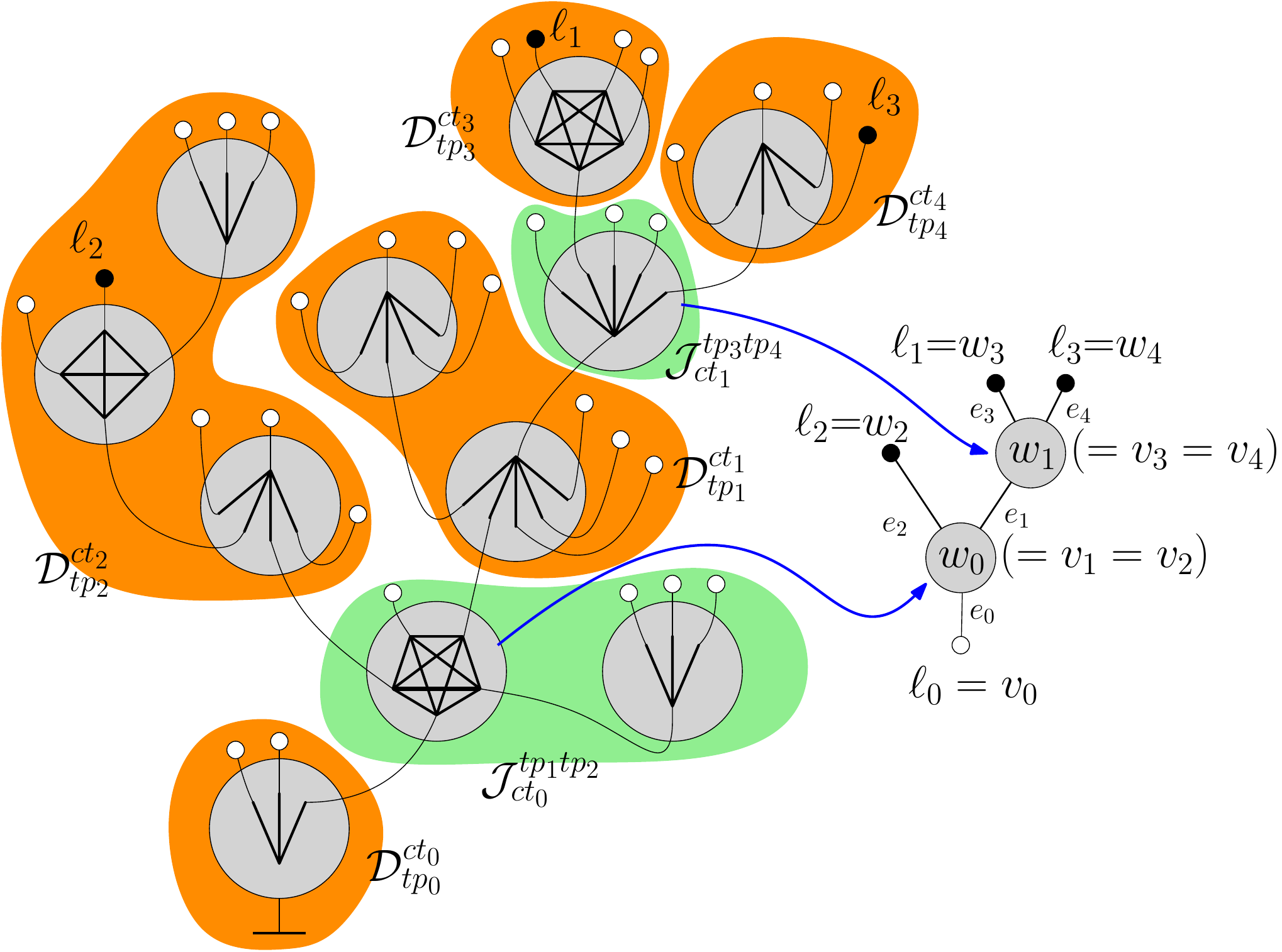}
  \caption{Decomposition of a DH-tree in $\mDt$. The groups of nodes indicated by orange and green areas correspond to the pieces defined in the proof of \cref{prop:decompo_Dt0}.}
  \label{fig:PreuveDecompo_Dt0}
  \end{figure}  

  We decompose $(T,\bm \ell)$ as follows. For each $i$ such that $w_i$ is an internal node of $t_0$, we cut the parent edge from $\bar{w_i}$,
  as well as the two edges incident to $\bar{w_i}$ which are the start of a path going to a marked leaf (since $t_0$ is a $k$-proper tree, there are always exactly two such edges). 
  This operation turns $(T,\bm \ell)$ into a disjoint union of trees, which we call pieces.
  Each edge that is cut is replaced by 
  a marked leaf (in the piece closer to the root of $T$) 
  and a root-leaf (in the piece further away from the root of $T$). 
 Then the piece containing $\bar{w_i}$ belongs to $\J_{ct_i}^{tp_{s_i} tp_{g_i}}$ (for every internal node $w_i$).
  Moreover, the pieces containing none of the $\bar{w_i}$ are in bijection with the edges of $t_0$, and the piece corresponding to $e_i$ belongs to $\mD_{tp_i}^{ct_i}$.

  By decomposing $(T,\bm \ell)$ we have indeed obtained an element of $\biguplus_{\mathbf{(tp,ct)}\in \E}\prod_{i=0}^{2k-2} \mD_{tp_i}^{ct_i} \prod  \J_{ct_i}^{tp_{s_i} tp_{g_i}}$.
  
  Conversely, let $(tp_i, ct_i)_{0\leq i \leq 2k-2}\in \E$ and take a tuple
  $$(T_{e_i})_{0\leq i\leq 2k-2}\times (T_{w_i})_{w_i\text{ internal node}} \in \prod_{i=0}^{2k-2} \mD_{tp_i}^{ct_i} \prod  \J_{ct_i}^{tp_{s_i} tp_{g_i}}.$$
  From these trees, we build a tree uniquely as follows.
  For every internal node $w_i$ of $t_0$, we glue the root-leaf of $T_{w_i}$ to the marked leaf of $T_{e_j}$, where $e_j$ is the edge from $w_i$ to its parent (when gluing, the two edges from the root-leaf and from the marked leaf become one edge, and the leaves disappear). Moreover, we glue the first marked leaf of $T_{w_i}$ to the root-leaf of $T_{e_{s_i}}$ and we glue the second marked leaf of $T_{w_i}$ to the root-leaf of $T_{e_{g_i}}$ (recall that $s_i$ (resp. $g_i$) is the smallest (resp. greatest) index of the edges from $w_i$ to its children).
  
  Since $t_0$ is a $k$-proper tree, once these gluings are done, we obtain only one tree $T$, 
  with one root-leaf (the one of $T_{e_0}$) and 
  $k$ marked leaves (those of $T_{e_i}$ where $e_i$ is incident to a leaf of $t_0$) which are in one-to-one correspondence with the leaves of $t_0$. 
  By construction (recalling also the definition of $\J_{a}^{b\,c}$), $T$ is a DH-tree, whose $k$ marked leaves induce $t_0$, and which satisfies item ii) of \cref{def:Dt} (since elements of $\mDab$ are DH-trees thus have one or more internal node(s)).
  All together, we have $T \in \mDt$.
  
  Finally, we have a size-preserving bijection, since the size of $T$ is the sum of the sizes of the $T_{e_i}$ and of the $T_{w_i}$. Indeed, for $\mDt$, $\mDab$ and $\J_{a}^{b\,c}$, the root-leaf and marked leaves are not counted in the size, and the leaves which have disappeared when gluing are all marked leaves or root-leaves.  
\end{proof}

\subsection{Asymptotic analysis}
Recall the notation $r(T,\bm \ell)$ from the end of \cref{ssec:51}, denoting the enriched induced subtree of $(T,\bm \ell)$.
\begin{proposition}
\label{prop:theoreme_local_limite_rTn}
Let $(\bm T_n, \bm \ell)$ be a uniform random DH-tree of size $n$ with $k$ marked leaves (not counted in the size).
Fix a $k$-proper tree $t_0$ and real numbers $x_0$, \dots, $x_{2k-2}>0$.
We set $a_i=\lfloor x_i \sqrt n \rfloor$. Then
\begin{equation} \label{eq:ConvLocale}
\mathbb P\big[ r(\bm T_n, \bm \ell) =(t_0,a_0,\dots,a_{2k-2}) \big] 
\sim \frac{\gamma_H^{2k}}{2^{k} \sqrt{n}^{2k-1}} s \exp\left(\frac{-\gamma^2_H s^2}{4}\right),
 \end{equation}
where $s=\sum_i x_i$.
Moreover, this estimate is uniform for $(x_0, \dots, x_{2k-2})$ in any compact subset
of $(0;+\infty)^{2k-1}$.
\end{proposition}

\begin{proof}
We first note that, for $n$ large enough,
\[ r(\bm T_n, \bm \ell) =(t_0,a_0,\dots,a_{2k-2}) \]
implies that $(\bm T_n, \bm \ell)$ is in $\Dt$.
 Indeed, item i) of \cref{def:Dt} comes from the definition of $t_0$; item ii) follows from the fact that  for every $i>0$,
we have $a_i>0$ (for $n$ large enough): so, there must be some jumps between each pair of essential vertices, and thus
they cannot be neighbors.

Therefore, writing $p(n):=\mathbb P\big[ r(\bm T_n, \bm \ell) =(t_0,a_0,\dots,a_{2k-2}) \big]$,
we have, for $n$ large enough, 
\begin{equation}
\label{eq:p}
p(n)= \frac{[z^n u_0^{a_0} \dots u_{2k-2}^{a_{2k-2}}] \Dt(z,u_0,\dots,u_{2k-2})} { [z^n] D^{(k)}(z) }.\end{equation}
We first analyze the denominator.
From \cref{eq:DevD} and singular differentiation,
we have
\[D^{(k)}(z)= \frac{1\cdot 3 \cdots (2k-3)}{2^k}  \gamma_D \rho^{-k} (1-z/\rho)^{1/2-k} + \mathcal O\big((1-z/\rho)^{-k} \big).\]
Applying the transfer theorem then yields
\begin{align}
 [z^n] D^{(k)}(z)
&\sim  \frac{1\cdot 3 \cdots (2 k-3)}{2^k}  \gamma_D \rho^{-k-n} \frac{n^{k-3/2}}{\Gamma(k-1/2)}\nonumber\\
&\sim  \frac{\gamma_D}{2 \sqrt \pi} \rho^{-k-n} n^{k-3/2}. \label{eq:coef_Dk}
\end{align}
Consider now the numerator of \cref{eq:p}.
We start from \cref{eq:decompo_Dt0} and use that from \cref{eq:formeDab} all $D_a^b$ are of the form 
\[\Dab=Q_a^b(z)+\frac{M_a^b(z)}{1-u\, H(z)}.\]
When expanding the product $\prod_{i=0}^{2k-2} D_{tp_i}^{ct_i}(z,u_i)$ in \cref{eq:p}, we can forget the terms
$Q_{tp_i}^{ct_i}(z)$ without changing the coefficient of 
$z^n u_0^{a_0} \dots u_{2k-2}^{a_{2k-2}}$
(indeed, since $x_i>0$, we have $a_i>0$ and $D_{tp_i}^{ct_i}(z,u_i)$ is the only factor containing $u_i$).
Also, clearly, $[u_i^{a_i}]\frac{1}{1-u_i\, H(z)}=H(z)^{a_i}$.
We therefore get
\begin{equation}
\label{eq:coeff_Dt}
[z^n u_0^{a_0} \dots u_{2k-2}^{a_{2k-2}}] \Dt(z,u_0,\dots,u_{2k-2})
= 
 \sum_{\mathbf{(tp,ct)}\in \E}
 [z^n] \, \widetilde M_\mathbf{(tp,ct)}(z) H(z)^{a_0+\dots+a_{2k-2}},
 \end{equation}
 where 
 \begin{equation}
 \label{eq:Mtilde}
 \widetilde M_\mathbf{(tp,ct)}(z)=\prod_{i=0}^{2k-2} M_{tp_i}^{ct_i}(z) \prod_{i \atop
 w_i\text{ is an internal node}}  J_{ct_i}^{tp_{s_i} tp_{g_i}}(z).
 \end{equation}
We apply\footnote{Of course, when $(x_0, \dots, x_{2k-2})$ spans a compact subset of $(0,+\infty)^{2k-1}$, then $s$ spans a compact subset of $(0,+\infty)$.} the Semi-large powers Theorem (see \cref{Th:SemiLarge} p.\pageref{Th:SemiLarge}) with $s=\sum_{i=0}^{2k-2} x_i$.
Using that $H(\rho)=1$ (see \cref{Hrho=1}),
we have 
\[[z^n] \, \widetilde M_\mathbf{(tp,ct)}(z) H(z)^{a_0+\dots+a_{2k-2}}\sim
\frac{s\gamma_H}2 \exp\left(\frac{-s^2\, \gamma_H^2}4 \right) \frac{1}{n\rho^{n}\sqrt{\pi}}
 \widetilde M_\mathbf{(tp,ct)}(\rho) ,
 \]
 where we recall that $\gamma_H$ is given by \cref{eq:gamma_H}.
Therefore we have
{\small  \begin{equation}
 \label{eq:coef_Dt_final}
 [z^n u_0^{a_0} \dots u_{2k-2}^{a_{2k-2}}] \Dt(z,u_0,\dots,u_{2k-2})
\sim \frac{s\gamma_H}2 \exp\left(\frac{-s^2\, \gamma_H^2}4 \right) \frac{1}{n\rho^{n}\sqrt{\pi}}
   \sum_{\mathbf{(tp,ct)}\in \E} \widetilde M_\mathbf{(tp,ct)}(\rho).
 \end{equation}}
To make notation lighter, we set $\kappa:=\sum_{\mathbf{(tp,ct)}\in \E} \widetilde M_\mathbf{(tp,ct)}(\rho)$,
which we will evaluate later. From \cref{eq:p,eq:coef_Dk,eq:coef_Dt_final},
we have
\[p(n) \sim
\frac{\frac{s\gamma_H}2 \exp\left(\frac{-s^2\, \gamma_H^2}4 \right) \frac{1}{n\rho^{n}\sqrt{\pi}}\kappa}
{ \frac{\gamma_D}{2 \sqrt \pi} \rho^{-k-n} n^{k-3/2}}=n^{-k+1/2} \rho^{k} \frac{ \gamma_H \kappa}{\gamma_D} \, s\, \exp\left(\frac{-s^2\, \gamma_H^2}4 \right).  \]
To conclude the proof of \cref{prop:theoreme_local_limite_rTn}, it remains to check that
\begin{equation}
\label{eq:relation_constant}
 \rho^{k} \frac{ \gamma_H \kappa}{\gamma_D} =\frac{\gamma_H^{2k}}{2^{k} }.
\end{equation}
To this end we simplify the quantity $\kappa=\sum_{\mathbf{(tp,ct)}\in \E} \widetilde M_\mathbf{(tp,ct)}(\rho)$. Since $M_a^b=\Lambda_a \Lambda_b$, we have
\[\widetilde M_\mathbf{(tp,ct)}(\rho)
= \prod_{i=0}^{2k-2} \Lambda_{tp_i}(\rho) \, \Lambda_{ct_i}(\rho) \ \prod_{i \atop
 w_i\text{ is an internal node}}  J_{ct_i}^{tp_{s_i}\, tp_{g_i}}(\rho).
\]
The first product runs over edges of $t_0$. We can rearrange its terms according to vertices.
Namely, we get a term $\Lambda_\bullet(\rho)$ for the root-leaf of $t_0$
and one for each leaf of $t_0$ (the type of the root-leaf and the cotypes
of the leaves are $\bullet$; see the definition of $\E$ in \cref{prop:decompo_Dt0}).
Additionally, for each $i$ such that $w_i$ is an internal vertex, we get a factor 
$\Lambda_{ct_i}(\rho)$ from the parent edge $e_i$ of $w_i$,
and two factors $\Lambda_{tp_{s_i}}(\rho)$ and $\Lambda_{tp_{g_i}}(\rho)$ from the children edges $e_{s_i}$ and $e_{g_i}$ of $w_i$.
The above display therefore rewrites as
\[\widetilde M_\mathbf{(tp,ct)}(\rho) 
=\Lambda_\bullet(\rho)^{k+1} \prod_{i \atop
 w_i\text{ is an internal node}}  \Lambda_{ct_i}(\rho) \Lambda_{tp_{s_i}}(\rho)
 \Lambda_{tp_{g_i}}(\rho)
 J_{ct_i}^{tp_{s_i}\, tp_{g_i}}(\rho).\]
 We now want to sum this quantity over $\mathbf{(tp,ct)}$ in $\E$.
 Note that choosing an element of $\E$ consists in choosing $ct_i$, $tp_{s_i}$
 and $tp_{g_i}$ for each internal vertex $w_i$.
The sum $\kappa=\sum_{\mathbf{(tp,ct)}\in \E} \widetilde M_\mathbf{(tp,ct)}(\rho)$
therefore factorizes over internal vertices of $t_0$ (there are $k-1$ of them)
 and we get
\[\kappa= \Lambda_\bullet(\rho)^{k+1} \left( 
\sum_{ct,tp_s,tp_g \in \{\KK,\SSX,\SSC\}^3}\Lambda_{ct}(\rho) 
\Lambda_{tp_{s}}(\rho) \Lambda_{tp_{g}}(\rho)
 J_{ct}^{tp_{s}\, tp_{g}}(\rho)\right)^{k-1}.   \]
We can write $\kappa=\mu \nu^k$, with 
\begin{align}
\mu&= \Lambda_\bullet(\rho) \left( 
\sum_{ct,tp_s,tp_g \in \{\KK,\SSX,\SSC\}^3}\Lambda_{ct}(\rho) 
\Lambda_{tp_{s}}(\rho) \Lambda_{tp_{g}}(\rho)
 J_{ct}^{tp_{s}\, tp_{g}}(\rho)\right)^{-1}; \label{eq:mu}\\
\nu&= \Lambda_\bullet(\rho) \left( 
\sum_{ct,tp_s,tp_g \in \{\KK,\SSX,\SSC\}^3}\Lambda_{ct}(\rho) 
\Lambda_{tp_{s}}(\rho) \Lambda_{tp_{g}}(\rho)
 J_{ct}^{tp_{s}\, tp_{g}}(\rho)\right).\label{eq:nu}
\end{align}
Then \cref{eq:relation_constant} holds for any $k \ge 1$ if 
\[\gamma_H^2=2 \rho\, \nu \quad \text{ and } \quad \gamma_H\mu=\gamma_D,\]
which we verify using Maple, from the definitions of the $\Lambda_\alpha$ and \cref{lem:J} for the $J_a^{b,c}$ (observing that $J_a^{b,c}=J_a^{c,b}=J_b^{a,c}$ for all $a,b,c$).
\end{proof}

\cref{prop:theoreme_local_limite_rTn} is a kind of local limit theorem for $r(\bm T_n, \bm \ell)$.
It is rather standard that such statements imply convergence in distribution statements.
We now state the convergence in  distribution of $r(\bm T_n, \bm \ell)$ (after normalization),
which we prove for completeness.

\begin{corollary}\label{Coro:ConvergenceEnriched}
Recall that $(\bm T_n, \bm \ell)$ denotes a uniform random DH-tree of size $n$ with $k$ marked leaves (not counted in the size). We set
$$
(\bm t_0^n,A_0,\dots,A_{2k-2})=r(\bm T_n, \bm \ell).
$$
Then
$$
\left(\bm t_0^n,\frac{A_0}{\sqrt{n}},\dots,\frac{A_{2k-2}}{\sqrt{n}}\right) \stackrel{(d)}{\to}
\left( \bm t_0,X_0,\dots X_{2k-2}\right)
$$
where 
\begin{itemize}
\item  $\bm{t}_0$ is a uniform $k$-proper tree;
\item $(X_0,\dots,X_{2k-2})$ have joint density 
\begin{equation}\label{eq:density}
(2k-3)!!\, \frac{\gamma_H^{2k}}{2^k} s\exp(-\gamma_H^2 s^2/4)\, dx_0\dots dx_{2k-2},\qquad \text{with }s:= x_0 +\dots +x_{2k-2}
\end{equation}
and are independent from $\bm{t}_0$.
\end{itemize}
\end{corollary}

\begin{proof}
Fix a tree $t_0$ and positive real numbers 
$b_0<c_0$, $b_1<c_1$, \dots, $b_{2k-2}<c_{2k-2}$. We consider the probability
\begin{multline}\label{eq:ConvergenceLoi}
\mathbb{P}\left(\left(\bm t_0^n,A_0,\dots,A_{2k-2} \right)\in \{t_0\}\times[b_0\sqrt{n},c_0\sqrt{n}]\times \dots \times[b_{2k-2}\sqrt{n},c_{2k-2}\sqrt{n}] \right)\\
= \sum_{a_0=\lceil b_0\sqrt{n}\rceil}^{\lfloor c_0\sqrt{n} \rfloor} \cdots
\sum_{a_{2k-2}=\lceil b_{2k-2}\sqrt{n}\rceil}^{\lfloor c_{2k-2}\sqrt{n} \rfloor}
\mathbb P\big[ r(\bm T_n, \bm \ell) =(t_0,a_0,\dots,a_{2k-2}) \big] 
\end{multline}
Summands on the right-hand side are asymptotically given by \cref{eq:ConvLocale}.
Since this formula is uniform in every compact subset of $(0,+\infty)^{2k-2}$, 
we can substitute each summand by its equivalent and get
\begin{multline}
\mathbb{P}\left(\left(\bm t_0^n,A_0,\dots,A_{2k-2} \right)\in \{t_0\}\times[b_0\sqrt{n},c_0\sqrt{n}]\times \dots \times[b_{2k-2}\sqrt{n},c_{2k-2}\sqrt{n}] \right) \\
\sim  \sum_{a_0=\lceil b_0\sqrt{n}\rceil}^{\lfloor c_0\sqrt{n} \rfloor} \cdots
\sum_{a_{2k-2}=\lceil b_{2k-2}\sqrt{n}\rceil}^{\lfloor c_{2k-2}\sqrt{n} \rfloor}
\frac{1}{\sqrt{n}^{2k-1}} \frac{\gamma_H^{2k}}{2^{k}} s \exp\left(\frac{-\gamma^2_H s^2}4 \right),
\label{eq:P_reducedtree_in_interval}
 \end{multline}
 with $s=\sum_{i=0}^{2k-2} x_i$.
This sum is well approximated by the corresponding Riemann sum and  converges to
\begin{equation}
\label{eq:RiemannSum}
\frac{\gamma_H^{2k}}{2^{k}} \int_{[b_0,c_0]\times \dots \times[b_{2k-2},c_{2k-2}]} s\exp\left(\frac{-\gamma^2_H s^2}4\right)dx_0 \dots dx_{2k-2}.
\end{equation}

Therefore we have, for any $k$-proper tree $t_0$,
\[\liminf_{n \to +\infty} \mathbb{P}\left( (\bm T_n, \bm \ell) \text{ induces }t_0 \right)
\ge \frac{\gamma_H^{2k}}{2^{k}} \int_{[b_0,c_0]\times \dots \times[b_{2k-2},c_{2k-2}]} 
s \exp\left(\frac{-\gamma^2_H s^2}4 \right) dx_0 \dots dx_{2k-2}.\]
Since this holds for any 
$b_0<c_0$, $b_1<c_1$, \dots, $b_{2k-2}<c_{2k-2}$,
we have
\[\liminf_{n \to +\infty} \mathbb{P}\left( (\bm T_n, \bm \ell) \text{ induces }t_0 \right)
\ge \frac{\gamma_H^{2k}}{2^{k}} \int_{\mathbb R_+^{2k-1}} s \exp\left(\frac{-\gamma^2_H s^2}4 \right)dx_0 \dots dx_{2k-2}.\]
Call $I$ the right-hand-side.
Performing the change of variables $s_0:=x_0$, $s_1:=x_0+x_1$, \dots, $s_{2k-2}:=x_0+\dots+x_{2k-2}=s$
we get (note that the Jacobian matrix of this change of variable has determinant 1):
\[ I= \frac{\gamma_H^{2k}}{2^{k}} \int_{\mathbb R_+} s \exp\left(\frac{-\gamma^2_H s^2}4 \right) \left( \int_{s_0 \le \dots \le s_{2k-1} \le s} ds_0 \dots ds_{2k-1} \right) ds.\]
The inner integral is equal to $\frac{s^{2k-2}}{(2k-2)!}$.
Thus we get
\begin{multline*}
I=  \frac{\gamma_H^{2k}}{2^{k} (2k-2)!} \int_{\mathbb R_+} s^{2k-1} \exp\left(\frac{-\gamma^2_H s^2}4 \right) ds 
\\ = 
 \frac{1}{(2k-2)!} \int_{\mathbb R_+}  y^{2k-1} \exp(-y^2/2) dy = \frac{2^{k-1}(k-1)!}{(2k-2)!}=\frac{1}{(2k-3)!!},
 \end{multline*}
where the second inequality is obtained by setting $y=\gamma_H s/\sqrt{2}$ and the third
by repeating integration by part.
Summing up, for any $k$-proper tree $t_0$, we have
\[\liminf_{n \to +\infty} \mathbb{P}\left( (\bm T_n, \bm \ell) \text{ induces }t_0 \right)
\ge \frac{1}{(2k-3)!!}.\]
Since there are $(2k-3)!!$ $k$-proper trees $t_0$,
 the infimum limit needs to be an actual limit and the inequality is an equality.
Therefore, we have proved that $\bm t_0^n$ converges in distribution to a uniform $k$-proper tree. 

Now, \cref{eq:P_reducedtree_in_interval,eq:RiemannSum} 
imply that
\begin{multline*}
 \mathbb{P}\left[\left(A_0,\dots,A_{2k-2} \right)\in [b_0\sqrt{n},c_0\sqrt{n}]\times \dots \times[b_{2k-2}\sqrt{n},c_{2k-2}\sqrt{n}] \ \big| \ (\bm T_n, \bm \ell) \text{ induces }t_0  \right]\\
=\frac{\mathbb{P}\left[\left(\bm t_0^n,A_0,\dots,A_{2k-2} \right)\in \{t_0\}\times[b_0\sqrt{n},c_0\sqrt{n}]\times \dots \times[b_{2k-2}\sqrt{n},c_{2k-2}\sqrt{n}] \right]}{\mathbb{P}\left( (\bm T_n, \bm \ell) \text{ induces }t_0 \right)}\\
\sim (2k-3)!! \frac{\gamma_H^{2k}}{2^{k}} \int_{[b_0,c_0]\times \dots \times[b_{2k-2},c_{2k-2}]} s\exp\left(\frac{-\gamma^2_H s^2}4\right)dx_0 \dots dx_{2k-2}.
\end{multline*}
Consequently, for any $t_0$, conditioning on \enquote{$(\bm T_n, \bm \ell) \text{ induces }t_0$}
 the vector $$\left(\frac{A_0}{\sqrt{n}},\dots,\frac{A_{2k-2}}{\sqrt{n}}\right)$$ converges in distribution
to a vector $(X_0,\dots, X_{2k-2})$ with density given by \eqref{eq:density}, whose expression does not depend on $t_0$.
This ends the proof of the corollary.
\end{proof}

\subsection{Gromov--Prohorov convergence of DH graphs}

Let $\Gn$ be the uniform DH graph of size $n$.
We want to deduce from \cref{Coro:ConvergenceEnriched} the convergence in distribution of the marginals of the distance matrix of $\Gn$.
To do this recall that \cref{lem:dG} allows us to estimate distances in $\Gn$ in terms of jumps in the associated DH-tree.

We first reformulate \cref{lem:dG} with the vocabulary of induced subtrees.
For $k\geq 2$ let $G$ be a  DH graph  of size $n+k+1$ (whose vertex set is therefore  $\{1, \ldots, n+k+1\}$).  Let $v_0$ be the vertex of $G$ with label $n+k+1$ and, for $1\leq i\leq k$, let  $v_i$ be the vertex of $G$ with label $n+i$.
Denote by  $(T, \bm \ell)$ the DH-tree associated to $G$ in the following way: 
the tree $T$, whose root-leaf $\ell_0$ corresponds to $v_0$, 
has size $n$ and $k$ marked leaves $\ell_1,\dots,\ell_k$ respectively corresponding to vertices $v_1,\dots,v_k$.
We denote by $(t_0,\alpha_0,\dots,\alpha_{2k-2})$ the enriched induced subtree $r(T,\bm \ell)$ (defined at the end of \cref{ssec:51}). 

\begin{lemma}\label{lem:DistancesGn} 
For $0\leq i,j\leq k$, 
let $\mathcal{P}_{i,j}^{t_0}$ be the path joining leaves $\ell_i$ and $\ell_j$ in $t_0$. Then, for some $\zeta(G,k)$ such that $1\leq \zeta(G,k)\leq k$, 
\[d_G(v_i,v_j)=\sum_{r:\ e_r\in\mathcal{P}_{i,j}^{t_0}}\alpha_r+ \zeta(G,k),\]
where $e_0,\dots,e_{2k-2}$ is the enumeration of edges of $t_0$.
\end{lemma}
\begin{proof}
\cref{lem:dG} states that
\[d_G(v_i,v_j)=\sum_{r:\ e_r\in\mathcal{P}_{i,j}^{t_0}}\alpha_r+ \#\{\text{ jumps in essential vertices}\}\,+1,\]
hence the new formulation.
\end{proof}

\begin{proposition}\label{Prop:ConvergenceMatrix}
Let $(\Gm)_{m\geq 3}$ be a sequence of uniform random labeled DH graphs of size $m$. Let $k\geq 1$ and $V_0,V_1,\dots,V_k$ be uniform i.i.d.~vertices in $\Gm$. Then we have the joint convergence in distribution:
\begin{equation}\label{eq:ConvergenceMatrix}
\bigg(\frac{1}{\sqrt{m}}\frac{\sqrt{2}}{\gamma_H}d_{\Gm}(V_i,V_j) \bigg)_{0\leq i,j\leq k} 
\xrightarrow[m\to+\infty]{(d)}
 \bigg( d_\infty(v_i,v_j) \bigg)_{0\leq i,j\leq k}
\end{equation}
where the right-hand side denotes the marginals of distances in the Brownian CRT defined by \cref{eq:MarginCRT}.
\end{proposition}

\begin{proof}
We fix $k\geq 1$. 
We first observe that with probability $1- \mathcal{O}(k^2/m)$ we have that $k$ i.i.d. uniform vertices in $\Gm$ are distinct. Therefore we can prove 
\eqref{eq:ConvergenceMatrix} where $(V_0,V_1,\dots,V_k)$ is a uniform $k+1$-tuple of distinct vertices.

Since the distribution of $\Gm$ is invariant by relabeling of vertices, we have that
$$
\bigg(d_{\Gm}(V_i,V_j) \bigg)_{0\leq i,j\leq k} 
\stackrel{(d)}{=}
\bigg(d_{{\bm H}_{m-k-1}}(W_i,W_j) \bigg)_{0\leq i,j\leq k} 
$$
where ${\bm H}_{m-k-1}$ is a uniform DH graph of size $m-k-1$ with
$k+1$ marked vertices $W_0,\dots,W_k$ not counted in the size.

Using \cref{lem:DistancesGn} with $G={\bm H}_{m-k-1}$ yields 
$$
\bigg(d_{\Gm}(V_i,V_j) \bigg)_{0\leq i,j\leq k} 
\stackrel{(d)}{=}
\left(\sum_{r:\ e_r\in\mathcal{P}_{i,j}^{\bm t_0^m}}A_r +\mathcal{O}(1) \right)_{0\leq i,j\leq k} 
$$
We finally use the convergence obtained in \cref{Coro:ConvergenceEnriched} (put  $n=m-k-1$) and the criterion of
 \cref{lem:distances_CRT} .
\end{proof}

From \cref{Th:CritereGP}, \cref{Prop:ConvergenceMatrix} implies the convergence of uniform DH graphs of size $n$ towards the Brownian CRT w.r.t. the Gromov--Prohorov topology. Thus this concludes the proof of \cref{th:GrosTheoreme} in the case $f=d$.

\section{The case of $2$-connected DH graphs}\label{Sec:2connectedDH}
The goal of this section is to prove the convergence of a uniform
random $2$-connected DH graph to the Brownian CRT, i.e.~the
case $f=2c$ in \cref{th:GrosTheoreme}.
We start by giving a characterization of $2$-connected DH graphs
through the associated (reduced) clique-star tree.
The proof of the case $f=2c$ in \cref{th:GrosTheoreme} then follows essentially
the same steps as that of the case $f=d$ (unconstrained DH graphs).
We shall indicate all necessary modifications.

\subsection{Combinatorial characterization}
Recall that a vertex $v$ in a connected graph $G$ is called a {\em cut-vertex}
if removing $v$ (and edges incident to $v$) disconnects $G$.
A connected graph $G$ without cut-vertices is said to be {\em $2$-connected}.
Cut-vertices in DH graphs, and hence $2$-connected DH graphs, are easily characterized through
the associated reduced clique-star tree.

\begin{lemma}
  Let $G$ be a DH graph and let $\tau$ be a clique-star tree such that $G=\Gr(\tau)$.
  A vertex $\ell$ in G is a cut-vertex if and only if the associated leaf in $\tau$ is connected
  to the center of a star. Consequently, a distance-hereditary graph $G=\Gr(\tau)$ is $2$-connected
  if and only if no leaf of $\tau$ is connected to
  the center of a star.
  \label{lem:characterization_2connected}
\end{lemma}

\begin{proof}
  We abusively call also $\ell$ the leaf of $\tau$
  corresponding to the vertex $\ell$ of $G$,
  and $v$ the unique vertex of $\tau$ adjacent to $\ell$.
  We also denote by $\Gamma_v$ the decoration of $v$, 
  and by $x$ the marker vertex of $\Gamma_v$ corresponding to the edge $(v,\ell)$.
  Finally we denote $G\backslash \ell$ the graph 
  obtained by removing $\ell$ (and its incident edges) from $G$.
  
  By construction $G \backslash \ell=\Gr(\tau \backslash \ell)$,
  where $\tau \backslash \ell$ is the decorated tree
  obtained from $\tau$ by erasing the leaf $\ell$ and replacing in $v$ the decoration
  $\Gamma_v$ by $\Gamma_v \backslash x$ 
  (note that $\tau \backslash \ell$ might not be a clique-star tree).
  By \cite[Lemma 2.3]{SplitTrees}, $G \backslash \ell$ is connected if and only if
  all decorations of $\tau \backslash \ell$ are connected.
  The only potentially non-connected decoration is $\Gamma_v \backslash x$
  and it is disconnected precisely when $\Gamma_v$ is a star, and $x$ its center.
  This proves the characterization of cut-vertices given in the lemma.
  The characterization of $2$-connected graphs follows immediately.
\end{proof}

By abuse of terminology, we say that a clique-star tree, or a DH-tree, is $2$-connected
if the associated DH graph is $2$-connected , {\it i.e.}~if it does not contain a leaf (including the root-leaf
in the case of DH-trees) linked to the center of a star. 
Specializing the bijection between DH-graphs and DH-trees to $2$-connected objects, the above lemma allows to easily adapt the system of equations \eqref{eq:SpecifDH} to this setting:
\begin{equation}
\label{eq:SpecifDH2c}
\begin{cases}
  \mDKdc =& \Set_{\geq 2}(\mathcal Z +\mDSCdc+\mDSXdc);\\
\mDSCdc=& \Set_{\geq 2}(\mathcal Z+ \mDKdc +\mDSXdc);\\
\mDSXdc=& (\mDKdc +\mDSCdc)\times \Set_{\geq 1}(\mathcal Z + \mDKdc+\mDSXdc);\\
\mDdc=& \mDKdc + \mDSXdc.
\end{cases}
\end{equation}
Here $\mDdc$ is the class of all $2$-connected DH-trees,
while $\mDKdc$ and $\mDSXdc$ are the subclasses of $\mDdc$,
consisting of trees with root of type $\KK$ or $\SX$, respectively.
The class $\mDSCdc$ is the class of DH-trees with a root of type $\SC$,
such that {\em no other leaf than the root-leaf} is connected to the center of a star.
DH-trees in $\mDSCdc$ are not $2$-connected DH-trees since one of their leaves, 
namely the root-leaf, is connected to a center of a star.
This explains why $\mDSCdc$ does not appear in the equation defining $\mDdc$ above.
We nevertheless need to introduce this auxiliary class to write a full system of equations.

\subsection{Singularity analysis of the system}
\label{ssec:asymptotics_mainseries_2c}
As usual, for a class $\mDdc_\alpha$, we denote by $\overline{D}_{\alpha}$ its exponential generating function. 
From \cref{eq:SpecifDH2c}, we immediately check that, as in the unconstrained case,
we have $\DKdc=\DSCdc$. Also, from the Drmota-Lalley-Woods theorem,
all series $\Ddc$, $\DKdc$ and $\DSXdc$ have the same radius of convergence $\rho_{2c}$
and square root-singularities.

Again, it is useful to introduce the series
\[F_{2c}=\exp_{\ge 1}(z+\DKdc+\DSXdc).\]
The system \ref{eq:SpecifDH2c} is then rewritten as 
\begin{equation}                        
\label{eq:SpecifDH2cSimplified}                   
\begin{cases}                           
  \DKdc =& F_{2c} - z -\DKdc-\DSXdc;\\
\DSXdc=& 2\, \DKdc\, F_{2c},
\end{cases}                             
\end{equation}
which is easily solved as
\begin{equation}
\label{eq:KS_with_F_2c}
\begin{cases}
  \DKdc=&\frac{F_{2c}-z}{2+2F_{2c}};\medskip \\
\DSXdc=&\frac{(F_{2c})^2-zF_{2c}}{1+F_{2c}}.
\end{cases}
\end{equation}
This implies that $F_{2c}$ is solution of an equation of the type $F_{2c}=G(z,F_{2c})$,
with
\begin{equation}
G(z,w)=\exp_{\ge 1} \Bigg(z+ \frac{w-z}{2+2w}+\frac{w^2-zw}{1+w}\Bigg).
\end{equation}
Arguing as in \cref{prop:expansion_F}, we find, after some elementary computations
(the last equality being computed in the companion Maple worksheet), that:
\begin{itemize}
  \item $\rho_{2c}=2\,(F_{2c}(\rho_{2c}))^2+2\,F_{2c}(\rho_{2c})-1$;
  \item $F_{2c}(\rho_{2c})$ is the unique positive solution of the equation 
    $s=\exp_{\ge 1}(2s-\frac12)$;
  \item $F_{2c}(z)=F_{2c}(\rho_{2c})- \gamma_{F,2c} \sqrt{1-z/\rho_{2c}} +O\big(1-z/\rho_{2c}\big)$
    with $\gamma_{F,2c}>0$ and 
    \[(\gamma_{F,2c})^2=\frac{\rho_{2c}(1+F(\rho_{2c}))}{1+2F(\rho_{2c})}=\frac{\Big( 2\,(F_{2c}(\rho_{2c}))^2+2\,F_{2c}(\rho_{2c})-1 \Big)\,\Big( 1+F_{2c}(\rho_{2c}) \Big)}{1+2F_{2c}(\rho_{2c})}.\]
\end{itemize}

\subsection{$2$-connected DH-trees with a marked leaf}
As in the case of general DH graphs, 
the next step is to analyze families of $2$-connected DH-trees with a marked leaf.
For $a,b$ in $\{\KK,\SX,\SC\}$, 
we denote $\mDdcab$ the class of DH-tree with a root of type $a$, a marked leaf
of cotype $b$, and such that no leaf is connected to the center of a star,
except possibly the root-leaf or the marked leaf (when $a$ and/or $b$ is equal to $\SC$).
Moreover, we let $\Ddcab(z,u)$ be the corresponding bivariate (exponential) 
generating series, where the exponent of the variable $z$ is the size (number of nonmarked nonroot leaves) of the tree and the exponent of $u$ is
the number of jumps on the path from the root-leaf to the marked leaf.

These nine series satisfy the following system of equations,
whose proof is similar to that of \cref{prop:system_markedleaf}:
\begin{align}
\label{Syst:DaKdc}
& \begin{cases}
\DdcKK=(1+\DdcCK+\DdcXK) \, \exp_{\ge 1}(\DdcSC+\DdcSX+z);\smallskip\\
\DdcXK=(\DdcCK+\DdcKK) \, \exp_{\ge 1}(\DdcSX+\DdcK+z) \\
\qquad \qquad \qquad+ u \cdot (\DdcKK+\DdcXK) \, (\DdcSC+\DdcK)  \, \exp(\DdcSX+\DdcK+z);\smallskip\\
\DdcCK=(\DdcKK+\DdcXK) \, \exp_{\ge 1}(\DdcSX+\DdcK+z);
\end{cases}\smallskip
  \\
  \label{Syst:DaXdc}
&  \begin{cases}
\DdcKX=(\DdcXX+\DdcCX) \, \exp_{\ge 1}(\DdcSC+\DdcSX+z); \smallskip\\
\DdcXX=(\DdcCX+\DdcKX) \, \exp_{\ge 1}(\DdcSX+\DdcK+z)\\
\qquad \qquad \qquad+u \cdot (1+\DdcKX+\DdcXX) \, (\DdcSC+\DdcK) \,\exp(\DdcSX+\DdcK+z);\smallskip\\
\DdcCX=(1+\DdcKX+\DdcXX)\, \exp_{\ge 1}(\DdcSX+\DdcK+z);
\end{cases} \smallskip\\
  \label{Syst:DaCdc}
& \begin{cases}
\DdcKC=(\DdcCC+\DdcXC) \, \exp_{\ge 1}(\DdcSC+\DdcSX+z);\smallskip\\
\DdcXC=(1+\DdcCC+\DdcKC) \, \exp_{\ge 1}(\DdcSX+\DdcK+z)\\
\qquad \qquad \qquad+u \cdot (\DdcKC+\DdcXC)\, (\DdcSC+\DdcK) \,\exp(\DdcSX+\DdcK+z);\smallskip\\
\DdcCC=(\DdcKC+\DdcXC) \, \exp_{\ge 1}(\DdcSX+\DdcK+z).
\end{cases}
\end{align}
Recall that we have $ \exp_{\ge 1}(\DdcSC+\DdcSX+z)= \exp_{\ge 1}(\DdcSX+\DdcK+z)=F_{2c}$. Furthermore, using that $\DSCdc+\DKdc =2\DKdc= \frac{F_{2c}-z}{1+F_{2c}}$,
we have \[(\DdcSC+\DdcK) \,\exp(\DdcSX+\DdcK+z) = F_{2c}-z.\] 
After these simplifications, the system is similar to that of \cref{Syst:DaK2,Syst:DaX2,Syst:DaC2}, except that $F$ is replaced by $F_{2c}$
and $uF$ by $u\, (F_{2c}-z)$.
This system is solved as follows (either directly or by substituting $F$ with $F_{2c}$
and $uF$ with $u\, (F_{2c}-z)$ in \cref{eq:sol_DKK,eq:sol_DKX,eq:sol_DKC,eq:sol_DXX,eq:sol_DXC,eq:sol_DCC}):
\begin{align}
\DdcKK&=\frac{F_{2c}}{F_{2c}+1}-\frac{(F_{2c})^2}{(1+F_{2c})(1-2F_{2c})-u(F_{2c}-z)}; \label{eq:sol_DKKdc}\\
\DdcXX&=\frac{-1}{F_{2c}+1}+\frac{(1-F_{2c})^2}{(1+F_{2c})(1-2F_{2c})-u(F_{2c}-z)}; \label{eq:sol_DXXdc}\\
\DdcCC&=\frac{(F_{2c})^2}{(1+F_{2c})(1-2F_{2c})-u(F_{2c}-z)}; \label{eq:sol_DCCdc}\\
\DdcKX=\DdcXK&=\frac{-F_{2c}}{F_{2c}+1}+\frac{F_{2c}\, (1-F_{2c})}{(1+F_{2c})(1-2F_{2c})-u(F_{2c}-z)}; \label{eq:sol_DKXdc}\\
\DdcKC=\DdcCK&=\frac{(F_{2c})^2}{(1+F_{2c})(1-2F_{2c})-u(F_{2c}-z)}; \label{eq:sol_DKCdc}\\
\DdcXC=\DdcCX&=\frac{F_{2c}(1-F_{2c})}{(1+F_{2c})(1-2F_{2c})-u(F_{2c}-z)}.\label{eq:sol_DXCdc}
\end{align}
Recalling that $F_{2c}$ depends only on $z$ (not on $u$),
we note that in each case, the series can be written under the form
\begin{equation}
\label{eq:formeDab2}
\Ddcab=\overline{Q}_a^{b}(z)+\frac{\overline{M}_a^{b}(z)}{1-u\, \overline{H}_a^{b}(z)},
\end{equation}
where $\overline{Q}_a^{b}$, $\overline{M}_a^{b}$ and $\overline{H}_a^{b}$ are rational functions in $F_{2c}$ (and $z$ in the case of $\overline{H}_a^{b}$).
For example, looking at \cref{eq:sol_DKKdc}, we have 
\[\overline{Q}_\KK^{\KK}=\frac{F_{2c}}{F_{2c}+1},\ 
\overline{M}_\KK^{\KK}=\frac{(F_{2c})^2}{(1+F_{2c})(1-2F_{2c})}\text{ and }\overline{H}_\KK^{\KK}=\frac{F_{2c}-z}{(1+F_{2c})(1-2F_{2c})}.\]
Similar formulas are easily written for other $a,b \in \{\KK,\SSC,\SSX\}$,
looking at \cref{eq:sol_DXXdc,eq:sol_DCCdc,eq:sol_DKXdc,eq:sol_DKCdc,eq:sol_DXCdc}.
As in the case of unconstrained DH graphs,
the auxiliary series $\overline{H}_a^{b}=\overline{H}$ do not depend on $a$ and $b$.
At $z=\rho_{2c}$, $\overline{H}$ admits an expansion of the form 
$\overline{H}(z)=\overline{H}(\rho_{2c})- \gamma_{H,2c} \sqrt{1-z/\rho_{2c}} +O\big(1-z/\rho_{2c}\big)$ 
with 
\begin{equation}\label{eq:gamma_H_2c}
\gamma_{H,2c} = \frac{2\sqrt{(1+2F_{2c}(\rho_{2c})) \cdot (2F_{2c}(\rho_{2c})^3 +4F_{2c}(\rho_{2c})^2+F_{2c}(\rho_{2c})-1)}}{1-F_{2c}(\rho_{2c}) -2 F_{2c}(\rho_{2c})^2}, 
\end{equation}
whose numerical estimate is $\gamma_{H,2c} \approx 7.5022$ (see Maple worksheet).

Furthermore we have
\[\overline{H}(\rho_{2c})=\frac{F_{2c}(\rho_{2c})-\rho_{2c}}{(1+F_{2c}(\rho_{2c}))(1-2F_{2c}(\rho_{2c}))}=1,\]
thanks to the relation $\rho_{2c}=2F_{2c}(\rho_{2c})^2+2F_{2c}(\rho_{2c})-1$
given at the end of \cref{ssec:asymptotics_mainseries_2c}.

\subsection{$2$-connected DH-trees with marked leaves inducing a given subtree}
Using the same terminology as in \cref{Sec:TreesWithInducesSubtrees},
we define $\overline{\mathcal{D}}_{t_0}$ to be the class of $2$-connected DH-trees with $k$ marked leaves inducing a given $k$-proper tree $t_0$.
Furthermore, we let $\overline{D}_{t_0}(z,u_0,...,u_{2k-2})$ be the multivariate (exponential) generating
series of $\overline{\mathcal{D}}_{t_0}$, where the exponent of $z$ is the size of the tree and the exponent of $u_i$ the number of jumps in the path corresponding to $e_i$ 
(in the fixed enumeration $(e_0,e_1,\dots, e_{2k-2})$ of the edges of $t_0$).

To write a combinatorial decomposition for $\overline{D}_{t_0}(z,u_0,...,u_{2k-2})$,
we need to introduce a subclass $\overline{\J}_{a}^{b\,c}$ of $\J_{a}^{b\,c}$,
where the nonmarked (and nonroot) leaves are not allowed to be attached to the center of a star. The generating function of this auxiliary class is given by
\begin{align}
\label{eq:J2c}
\overline{J}_{a}^{b\,c}(z)& =  (\Ind_A+\Ind_B+\Ind_C) \exp(\DdcSX+\DdcK+z) +\Ind _{\KK \notin \{a,b,c\}}
\exp(\DdcSX+\DdcSC+z)\\ & \qquad \qquad + \Ind_{\SSC \notin \{a,b,c\}} (\DdcSC + \DdcK) \exp(\DdcK+\DdcSX+z) \notag\\
&=(\Ind_A+\Ind_B+\Ind_C+\Ind _{\KK \notin \{a,b,c\}}) (1+F_{2c}) +\Ind_{\SSC \notin \{a,b,c\}} (F_{2c}-z), \notag
\end{align}
where $A,B,C$ are given in \cref{lem:J}.

At this stage, there is a small difference with the case of unconstrained DH-trees.
In a $2$-connected DH-tree, the root cannot have type $\SC$ and no leaves (in particular the marked ones) can have cotype $\SC$.
Therefore in the combinatorial decomposition of \cref{fig:PreuveDecompo_Dt0},
the piece corresponding to $e_0$ has a root type different from $\SC$,
and pieces corresponding to leaf-edges of $t_0$ have a marked leaf with a cotype different
from $\SC$ as well.
This is easily captured in equations by defining 
\begin{align*}
\mDdcap&=\mDdcaK  +  \mDdcaX;\\
\mDdcpb&=\mDdcKb + \mDdcXb.
\end{align*}
In the unconstrained case, each of these equations had an extra term corresponding to the type (or cotype) $\SC$. With these definitions, \cref{prop:decompo_Dt0} is still valid when  replacing
each series by its $2$-connected counterpart.
The asymptotic analysis in the $2$-connected case is then identical
to that of the unconstrained case, up to the verification of the identities
\[ \gamma_{H,2c}^2=2 \rho_{2c}\, \nu_{2c}\text{ and }\gamma_{H,2c}\mu_{2c}=\gamma_{D,2c},\]
where $\nu_{2c}$ and $\mu_{2c}$ are defined via the obvious analogs of
\cref{eq:mu,eq:nu}.
Verifying these identities is done in the companion Maple worksheet.
We therefore have the following analog of \cref{prop:theoreme_local_limite_rTn}.
\begin{proposition}
\label{prop:theoreme_local_limite_rTn_2c}
Let $(\bm T_n, \bm \ell)$ be a uniform random $2$-connected DH-tree of size $n$ with $k$ marked leaves (not counted in the size).
Fix a $k$-proper tree $t_0$ and real numbers $x_0$, \dots, $x_{2k-2}>0$.
We set $a_i=\lfloor x_i \sqrt n \rfloor$. Then
\begin{equation} \label{eq:ConvLocale_2c}
\mathbb P\big[ r(\bm T_n, \bm \ell) =(t_0,a_0,\dots,a_{2k-2}) \big] 
\sim \frac{\gamma_{H,2c}^{2k}}{2^{k} \sqrt{n}^{2k-1}} s \exp\left(\frac{-\gamma^2_{H,2c} s^2}{4}\right),
 \end{equation}
where $s=\sum_i x_i$.
Moreover, this estimate is uniform for $x_0$, \dots, $x_{2k-2}$ in any compact subset
of $(0;+\infty)^{2k-1}$.
\end{proposition}
From here, the convergence to the Brownian CRT in Gromov--Prohorov topology, {\it i.e.}~the second case of \cref{th:GrosTheoreme}, follows
using the same arguments as in the case of unconstrained DH graphs.

\section{The case of $3$-leaf power graphs}\label{Sec:3leaf}
The goal of this section is to prove the convergence of a uniform
random $3$-leaf power graph to the Brownian CRT, \emph{i.e.}~the
case $f=3\ell$ in \cref{th:GrosTheoreme}.
We start by recalling a characterization of $3$-leaf power graphs
through their associated (reduced) clique-star tree, given in \cite{SplitTrees}.
The proof of convergence then follows essentially
the same steps as in the two other cases.
There is however one notable difference.
As we shall see, in this model,
first common ancestors of marked leaves are of type and cotype $\mathcal S_X$ with probability
tending to 1; therefore,
we only need to consider two types of trees with one marked leaf,
simplifying significantly the analysis.

\subsection{Definition and combinatorial analysis of $3$-leaf power graphs}
 This section follows closely \cite[Section 2]{ChauveFusyLumbroso}.
\begin{definition}
  \label{def:3Leaf}
Let $T$ be a tree and $L$ its set of leaves.
The $k$-leaf power graph $G$ of $T$ has by definition vertex set $L$,
and $\ell$ and $\ell'$ are connected in $G$ if they are at distance at most $k$ in $T$.
And  a graph is a $k$-leaf power graph if it is the $k$-leaf power graph of some tree.
\end{definition}
We are interested in the case $k=3$. It is known,
see \emph{e.g.}, \cite[Section 2]{ChauveFusyLumbroso} 
that $3$-leaf power graphs form
a subclass of distance-hereditary graphs, 
and that they can be characterized on the clique-star trees as follows (see \cite[Section 3.3]{SplitTrees}).
\begin{proposition}
A distance hereditary graph $G$ is a $3$-leaf power graph
if and only if its reduced clique star-tree $\tau$ satisfies the following properties:
\begin{itemize}
\item the set of star nodes forms a connected subtree of $\tau$;
\item no edge connects two centers of star nodes.
\end{itemize}
\end{proposition}

In the sequel, we call {\em $3$-leaf power trees} 
the DH-trees corresponding to (rooted) {\em $3$-leaf power graphs}.
Let  $\mathcal E$ be the combinatorial class
of $3$-leaf power trees.
To get a combinatorial decomposition of this class,
it is convenient to introduce the following subclasses:
\begin{itemize}
\item $\mESX$, $\mESC$ and $\mEK$ are the subclasses of $\mathcal E$,
where the root-node is required to have type $\SSX$, $\SSC$ and $\KK$ respectively;
\item $\mL$ is the class containing the tree restricted to a single leaf,
and trees consisting of a single internal node, of type $\KK$, with at least two pending leaves.
\end{itemize}
Recall also that $\mZ$ is a class with only one element, which is of size 1,
representing a leaf.
The following set of equations,
characterizing all these classes, is obtained easily:
\begin{equation}
\label{eq:Specif3Leaf}
\begin{cases}
\mL = \mZ + \Set_{\geq 2}(\mZ);\\
\mESX= \mL \times \Set_{\geq 1}(\mL+\mESX);\\
\mESC= \Set_{\geq 2}(\mL+\mESX);\\
\mEK =\mL + (\mESX + \mESC) \, \Set_{\ge 1}(\mZ); \\
\mathcal E= \mESX + \mESC + \mEK.
\end{cases}
\end{equation}
In terms of generating series (with the usual convention that $Y(z)$ is the exponential generating function of a class $\mathcal{Y}$), the first equation implies $L = e^z-1$.
The second equation yields:
\begin{equation}
\label{eq:Specif3Leaf_reduced}
  \ESX=  L \cdot \big( \exp(L+\ESX)-1 \big) = (e^z-1) \cdot \big(\exp(\ESX+e^z -1)-1 \big).
\end{equation}
We note that other equations of the system \cref{eq:Specif3Leaf} are nonrecursive
and simply express $\ESC$, $\EK$ and $E$ in terms of $\ESX$.
It is thus not surprising that most of the asymptotic analysis reduces to that of $\ESX$.
We first prove the following result.
\begin{proposition}
The series $\ESX$ is $\Delta$-analytic at $\rhotl=\log(1+e^{-1})$ and admits the following singular expansion around $\rhotl$:
\begin{align}
\ESX&=(1-e^{-1}) - \gamma_{E,3\ell}\sqrt{1-z/\rhotl}+ \mathcal{O}(1-z/\rhotl),\label{eq:ESX}
\end{align}
where
\begin{equation}
\gamma_{E,3\ell}=\sqrt{2(1+e)\log(1+e^{-1})}, \label{eq:gamma_E3l} 
\end{equation}
whose numerical estimate is $\gamma_{E,3\ell} \approx 1.5263$ (see Maple worksheet).
\end{proposition}

\begin{proof}
As in the proof of \cref{prop:expansion_F} we use the smooth implicit-function schema. We write
$\ESX=G(z,\ESX)$ where
$$
G(z,w)=(e^z-1)\left(\exp(w+e^z -1)-1 \right) 
$$
which is analytic in $z,w$ on the whole complex plane
 and which has nonnegative coefficients. 
 The characteristic system $\{G(r,s)=s; G_w(r,s)=1\}$
 is easily solved and the unique solution is
$$
r= \log(1+e^{-1}),\qquad s=1-e^{-1}.
$$
We apply  \cite[Theorem VII.3]{Violet} and obtain \cref{eq:ESX}. 
\end{proof}
With system \ref{eq:Specif3Leaf} and the singular expansion of $\ESX$ given above,
we find that of the other series. In particular,
\begin{equation}
\label{eq:DevE}
E=(e-e^{-1}-e^{-2}) - (e+1) \gamma_{E,3\ell} \sqrt{1-z/\rhotl} + \mathcal{O}(1-z/\rhotl),
\end{equation}
which will be useful later. 

\subsection{$3$-leaf power trees with a marked leaf}
We now consider families of $3$-leaf power trees with a marked leaf.
It turns out that the only classes relevant for the asymptotic analysis
are the classes $\mEXX$ and $\mEXp$ defined as follows:
 we let $\mEXX$ (resp.~$\mEpX$ or $\mEXp$) be the subclass of 
$3$-leaf power trees with a marked leaf such that the root has type $\SX$
and the marked leaf has cotype $\SX$
(resp.~with no constraints on the type of the root or on the cotype of the marked leaf).
As above, we consider the associated exponential bivariate generating series
 $\EXX$, $\EpX$ and $\EXp$, where the exponent of $z$ is the size of the tree
  (number of nonmarked nonroot leaves) and that of $u$ is the number of jumps on the 
  path from the root-leaf to the marked leaf.
  By symmetry we have $\EpX=\EXp$.

We also let $\mLp$ be the class of objects in $\mL$ with a marked leaf.
Its generating series is $\Lp=e^z$ (there is no jumps in such objects).
An easy combinatorial analysis yields the following equations:
\begin{align*}
\EXX&= u \cdot (1+\EXX) \cdot L \cdot \Set(L+\ESX);\\
\EXp&=\Lp \cdot \Set_{\geq 1}(L+\ESX)+ u \cdot (\Lp+\EXp)\cdot L \cdot \Set(L+\ESX).
\end{align*}
This is a 2x2 linear system of equations in the unknown series $\EXX$ and $\EXp$.
The solutions can be put under a form similar to \cref{eq:formeDab}:
\begin{align*}
\EXX(u) &= -1 +\frac1{1-u\, P}\\
\EXp(u) &= -e^z + \frac{e^z\, \exp(e^z-1+\ESX)}{1-u\, P},
\end{align*}
where 
$$P=P(z)=(e^z-1)\, \exp(e^z-1+\ESX).$$
Using \cref{eq:ESX}, we immediately see that $P$ has radius of convergence $\rho_{3\ell}$,
is $\Delta$-analytic and admits the following singular expansion for $z$ near $\rho_{3\ell}$:
\begin{equation}
\label{eq:DevP}
 P=1 - \gamma_{E,3\ell}\sqrt{1-z/\rhotl}+ \mathcal{O}(1-z/\rhotl).
 \end{equation}

\subsection{3-leaf power trees with $k$ marked leaves}
Let us fix a $k$-proper tree $t_0$.
We consider the following class of marked 3-leaf power trees.
\begin{definition}
\label{def:Et}
We let $\mEt$ be the labeled combinatorial class of 3-leaf power trees $T$
with $k$ marked leaves $\bm \ell$ such that:
\begin{enumerate}
\item the subtree of $T$ induced by $\bm \ell$ is $t_0$;
\item no two essential vertices of $T$ are neighbors of each other;
\item every internal essential vertex of $T$ has type and cotype $\SX$ and its children which are roots of subtrees containing marked leaves are also of type $\SX$.   
\end{enumerate}
\end{definition}
As above, we fix an enumeration $(e_0,\dots,e_{2k-2})$ of the edges of $t_0$ such that the edge adjacent to the root-leaf is labeled with $e_0$; here, we additionally require that the edges
incidents to leaves of $t_0$ get labels $e_1$, \dots, $e_k$.
Recall that we defined $\jump_e(T;\bm \ell)$ as the number of jumps on the path corresponding to $e$,
with the convention that essential vertices are not counted as jumps (but the root-node of $T$ can be a jump).
We consider the following multivariate generating series for $\mEt$: 
\[\Et(u_0,\dots,u_{2k-2})=  \sum_{(T;\bm \ell) \in \mEt} \frac{z^{|T|}}{|T|!} u_0^{\jump_{e_0}(T,\bm \ell)}
\cdots u_{2k-2}^{\jump_{e_{2k-2}}(T,\bm \ell)}.\]

Moreover, let $\I_{a}^{b\,c}$ be the set of 3-leaf power trees $T$ with two marked leaves such that
\begin{itemize}
\item the two marked leaves are children of the root-node;
\item if $T_1$ is a 3-leaf power of type $b$,
one can glue $T_1$ on the first 
marked leaf of $T$ (merging the marked leaf and the root-node of $T_1$) such that the tree obtained is a 3-leaf power tree; 
\item the same condition holds with gluing a 3-leaf power tree of type $c$ on the second marked leaf;
\item additionally, if $T_0$ is a 3-leaf power tree  with a marked leaf of cotype $a$,
one can glue $T$ on the marked leaf of $T_0$
obtaining a 3-leaf power tree.
\end{itemize}

\begin{lemma}
The generating function of $\I_{\SX}^{\SX\,\SX}$ is 
\begin{equation}
\label{eq:I}
I_{\SX}^{\SX\,\SX}(z) = L \exp(\ESX+L). 
\end{equation}
\end{lemma}

\begin{proof}
Because of the allowed adjacencies between nodes of various types in 3-leaf power trees, a 3-leaf power tree in $\I_{\SX}^{\SX\,\SX}$ necessarily has a root-node $r$ of type $\SX$. The factor $L$ then accounts for the tree pending under the center of the star labeling $r$, and $\exp(\ESX+L)$ accounts for the trees pending under its extremities which do not correspond to the marked leaves.
\end{proof}

\begin{proposition}
\label{prop:decompo_Et0}
We have 
\begin{equation}
\label{eq:decompo_Et0}
\Et(u_0,\dots,u_{2k-2})=  \left( \prod_{i=0}^k \EXp(u_i) \right)  
\cdot \left( \prod_{i=k+1}^{2k-2} \EXX(u_i) \right)  
\cdot \left(\I_{\SX}^{\SX\,\SX} \right)^{k-1}. 
\end{equation}
\end{proposition}
\begin{proof}[Sketch of proof]
We use the same decomposition as in the proof of \cref{prop:decompo_Dt0}.
The main difference is that, because of item iii) in \cref{def:Et} above,
all types $tp_i$ and cotypes $ct_i$ which do not correspond to the marked leaves 
or the root-leaf must be equal to $\SX$.
Consequently the tuple $(tp_i, ct_i)_{0\leq i \leq 2k-2}$ can only take one possible value
and we get
\[ \Et(u_0,\dots,u_{2k-2})=\EpX(u_0) \left( \prod_{i=1}^k \EXp(u_i) \right)  
\cdot \left( \prod_{i=k+1}^{2k-2} \EXX(u_i) \right)  
\cdot \left(\I_{\SX}^{\SX\,\SX} \right)^{k-1}.\]
We conclude using the symmetry $\EXp=\EpX$.
\end{proof}

\begin{proposition}
\label{prop:theoreme_local_limite_rTn_3L}
Let $(\bm T_n, \bm \ell)$ be a uniform random 3-leaf power tree of size $n$ with $k$ marked leaves (not counted in the size).
Fix a $k$-proper tree $t_0$ and real numbers $x_0, \dots, x_{2k-2}>0$.
We set $a_i=\lfloor x_i \sqrt n \rfloor$. Then

\begin{equation} \label{eq:ConvLocale_3L}
\mathbb P\Big[ r(\bm T_n, \bm \ell) =(t_0,a_0,\dots,a_{2k-2}) \wedge (\bm T_n, \bm \ell) \in \mathcal E_{t_0} \Big] 
\sim \frac{\gamma_{E,3\ell}^{2k}}{2^{k} \sqrt{n}^{2k-1}} s \exp\left(\frac{-\gamma_{E,3\ell}^2 s^2}{4}\right),
 \end{equation}
where $s=\sum_i x_i$.
Moreover, this estimate is uniform for $x_0$, \dots, $x_{2k-2}$ in any compact subset
of $(0;+\infty)^{2k-1}$.
\end{proposition}
\begin{proof}
Writing $p(n):=\mathbb P\big[ r(\bm T_n, \bm \ell) =(t_0,a_0,\dots,a_{2k-2}) \wedge (\bm T_n, \bm \ell) \in \mathcal E_{t_0} \big]$,
we have
\begin{equation}
\label{eq:pE}
p(n)= \frac{[z^n u_0^{a_0} \dots u_{2k-2}^{a_{2k-2}}] \Et(z,u_0,\dots,u_{2k-2})} { [z^n] E^{(k)}(z) }.\end{equation}
We first analyze the denominator.
From \cref{eq:DevE}, routine computations yield:
\begin{align}
 [z^n] E^{(k)}(z)=  \frac{(1+e) \, \gamma_{E,3\ell}}{2 \sqrt \pi} \rhotl^{-k-n} n^{k-3/2}. \label{eq:coef_Ek}
\end{align}

With the same reasoning as in \cref{prop:theoreme_local_limite_rTn},
we have
\begin{equation}
\label{eq:coeff_Et}
[z^n u_0^{a_0} \dots u_{2k-2}^{a^{2k-2}}] \Et(z,u_0,\dots,u_{2k-2})
= 
 [z^n] \, N(z) P(z)^{a_0+\dots+a_{2k-2}},
 \end{equation}
 where 
 \begin{equation}
N(z)=\Big(e^z \, \exp\big(e^z-1+\ESX\big)\Big)^{k+1} \cdot 
\Big( (e^z-1) \cdot \exp\big(e^z-1+\ESX\big)\Big)^{k-1}.
 \end{equation}
Applying the Semi-large powers Theorem (\cref{Th:SemiLarge})
and using \cref{eq:DevP},
we have 
\[[z^n] \, N(z) P(z)^{a_0+\dots+a_{2k-2}}\sim
\frac{s\gamma_{E,3\ell}}2 \exp\left(\frac{-s^2\, \gamma_{E,3\ell}^2}4 \right) \frac{1}{n\rhotl^{n}\sqrt{\pi}}
 N(\rhotl).\]
And since $N(\rhotl)=(e+1)^{k+1}$, it follows that
\[p(n) \sim \frac
{\frac{s\gamma_{E,3\ell}}2 \exp\left(\frac{-s^2\, \gamma_{E,3\ell}^2}4 \right) \frac{1}{n\rhotl^{n}\sqrt{\pi}}
\, (e+1)^{k+1}}
{\frac{(1+e) \gamma_{E,3\ell}}{2 \sqrt \pi} \rhotl^{-k-n} n^{k-3/2}}
= (e+1)^k \rhotl^k \, n^{-k+1/2} \, s \exp\left(\frac{-s^2\, \gamma_{E,3\ell}^2}4 \right).
  \]
  This concludes the proof since $\gamma_{E,3\ell}^2=2(1+e)\rhotl$.
  \end{proof}

From here, the convergence to the Brownian CRT in Gromov--Prohorov topology, {\it i.e.}~the third case of \cref{th:GrosTheoreme}, follows
using the same arguments as in the case of unconstrained DH graphs.

\appendix
\section{The Semi-large powers Theorem}\label{Sec:AppendixSemiLarge}

In order to prove \cref{prop:theoreme_local_limite_rTn,prop:theoreme_local_limite_rTn_2c,prop:theoreme_local_limite_rTn_3L}, we need to estimate quantities of the form 
$
[z^n] M(z) H(z)^{p}
$
where $p$ is of order $\sqrt{n}$.
The following statement is essentially the Semi-large powers Theorem (\cite[Theorem IX.16]{Violet} for $\lambda=1/2$, see also \cite{SemiLargePowers} for the original reference)
 which deals with the case $M(z)= 1$. As we will see, there are no particular difficulties in generalizing the proof.

\begin{theorem}
\label{Th:SemiLarge}
Let $\rho >0$ and let $M,H$ be $\Delta$-analytic functions at $\rho$.
Assume that
\begin{enumerate}
\item $H$ has a square-root singularity at $\rho$: for some $\sigma,h>0$
$$
H(z) = \sigma - h\sqrt{1 - z/\rho} + \mathcal{O}(1-z/\rho).
$$
\item $M$ converges at $\rho$.
\end{enumerate}
Let us fix a compact subset $K$ of $(0,+\infty)$ and a constant $C>0$.
Take $(a_n)$ a sequence of real numbers such that 
 \begin{equation}\label{eq:a_n}
|a_n- x\sqrt{n}|\leq C,
\end{equation}
for some $x$ in $K$.
Then we have
$$
[z^n] M(z)H(z)^{a_n}
\stackrel{n\to +\infty}{\sim}
{\sf Ray}\left(\frac{xh}{\sigma} \right) \frac{1}{n\rho^{n}\sqrt{\pi}}
 M(\rho) \sigma^{a_n},
$$
where ${\sf Ray}(x)=\tfrac{x}{2}e^{-x^2/4}$ is the Rayleigh density.

The error term in the above convergence is uniform for all $x\in K$,
and all sequences $(a_n)$ satisfying \eqref{eq:a_n}, but depends on $M$, $H$, $K$ and $C$.
\end{theorem}
\begin{proof}[Proof of \cref{Th:SemiLarge}]

We mimic the proof of \cite[Theorem IX.16]{Violet}.
By assumption there exists $R_1 > \rho$ and $\theta >0$ such that $M,H$ are analytic on 
\[\{ z : |z| <R_1 \text{ et }|\Arg(z-\rho)| > \theta \}.\]
 Fix $R$ in $(\rho,R_1)$ and write
\[
[z^n]  M(z) H(z)^{a_n}= \frac{1}{2i\pi} \int_{\gamma}  M(z) H(z)^{a_n} \frac{dz}{z^{n+1}},
\]
where $\gamma=\gamma_0\cup \overline{\gamma_0} \cup \gamma_1\cup\gamma_2$ is a closed counter-clockwise contour surrounding $0$ consisting of the following pieces (see \cref{Fig:HankelContour}):
\begin{itemize}
\item $\gamma_0$ is a line segment starting at $\rho+i/n$, with a slope $\theta$
and stopping when it reaches the circle $\{z : |z| = R\}$;
\item $\overline{\gamma_0}$ is its complex conjugate (in reverse direction);
\item $\gamma_1$ is a semi-circle centered at $\rho$ of radius $1/n$ from $\rho - i\rho/n$ to $\rho +i\rho/n$;
\item $\gamma_2$ is an arc of circle of radius $R$ closing $\gamma$.
\end{itemize}
The modulus of the integral on $\gamma_2$ is easily bounded by $\O(B^{\sqrt n} R^{-n})$ for some constant $B$.
On the remainder of the contour, we set $z=\rho(1-t/n)$, \emph{i.e.}~$t=n(1-z/\rho)$ and  get
$$
[z^n]  M(z) H(z)^{a_n}=\frac{1}{2i\pi} \int_{g}  \frac{-\rho dt}{n \rho^{n+1} (1-t/n)^{n+1}} M(\rho(1-\tfrac{t}{n})) H(\rho (1-\tfrac{t}{n}) )^{a_n},
$$
where $g$ is the image of $\gamma_0 \cup \overline{\gamma_0}\cup\gamma_1$ by the change of variable.
\begin{figure}[htbp]
\begin{center}
\includegraphics[width=12cm]{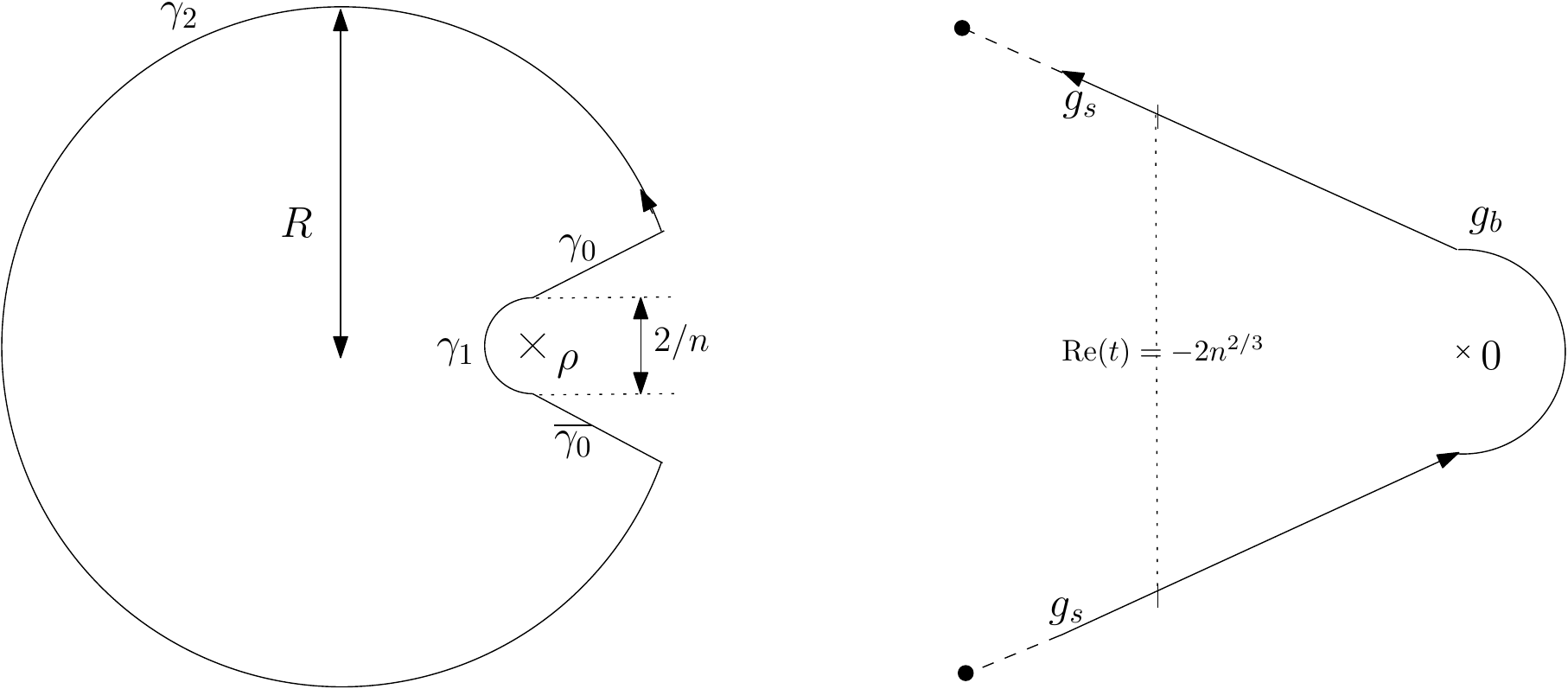}
\end{center}
\caption{Left: The contour $\gamma$. Right: The contour $g$, which is the image of $\gamma_0 \cup \overline{\gamma_0} \cup \gamma_1$ in the variable $t$,
where $t=n(1-z/\rho)$.}
\label{Fig:HankelContour}
\end{figure}

We write $g=g_s\cup g_b$ (for small and big), where $g_s=\{t\in g ;\ \mathrm{Re}(t)<-n^{3/5}\}$ (this set is not connected) and 
$g_b=\{t\in g ;\ \mathrm{Re}(t)\geq-n^{3/5}\}$; see again \cref{Fig:HankelContour}.

On $g_s$, we have 
\[ |1-\tfrac{t}n|^{n+1} \ge (1+n^{-2/5})^{n+1} = \Theta(e^{n^{3/5}}). \]

Since $H$ and $M$ are bounded on the integration path, for $a_n$ as in \eqref{eq:a_n} , we have
\[\int_{g_s} M(\rho(1-\tfrac{t}{n})) H(\rho (1-\tfrac{t}{n}) )^{a_n} \frac{-\rho dt}{n \rho^{n+1} (1-t/n)^{n+1}}
= \mathcal O\big(\rho^{-n} e^{B'\sqrt n-n^{3/5}}\big), \]
for some constant $B'$.
For  $t$ in $g_b$, a simple computation gives
$$M(\rho(1-\tfrac{t}{n})) H(\rho (1-\tfrac{t}{n}) )^{a_n}
 \sim M(\rho)\sigma^{a_n} \exp\left( -\tfrac{xh\sqrt{t}}\sigma \right).
$$
Besides $(1-\tfrac{t}{n})^{-n-1}=\exp(t+\mathcal{O}(\tfrac{t^2}{n}))=\exp(t +o(t))$. We obtain
$$
\frac{1}{2i\pi} \int_{g_b} M(\rho(1-\tfrac{t}{n})) H(\rho (1-\tfrac{t}{n}) )^{a_n} \frac{-\rho dt}{n \rho^{n+1} (1-t/n)^{n+1}}
=
-M(\rho)\sigma^{a_n}\frac{\rho^{-n}}{2i\pi n} \int_{g_b} e^{A(n,t) } dt,
$$
where
$$
e^{A(n,t)}=e^t e^{\tfrac{-h x \sqrt{t}}\sigma} e^{\mathcal{O}(\tfrac{t^2}{n})}.
$$
The big-$\mathcal{O}$ term above is uniform for $x$ in any compact subinterval of $(0,+\infty)$.
Expanding  $e^{-hx\sqrt{t}/\sigma} $,  putting $u=-t$, and using Hankel's formula for the Gamma function \cite[Eq.(13) p. 745]{Violet} yields (writing $MH^{a}$ for $M(\rho(1-\tfrac{t}{n})) H(\rho (1-\tfrac{t}{n}) )^{a_n}$)
\begin{align*}
\frac{1}{2i\pi} \int_{g_b} MH^{a}\frac{dz}{z^{n+1}}
&\sim
-M(\rho)\sigma^{a_n}\frac{\rho^{-n}}{2i\pi n} \sum_{\ell\geq 0} \frac{(-xh/\sigma)^\ell }{ \ell!}   \int_{-g_b}  e^{-u}(-u)^{\ell/2} (-du)\\
&\sim
-M(\rho)\sigma^{a_n}\frac{\rho^{-n}}{n} 
\sum_{\ell\geq 0} \frac{(-xh/\sigma)^\ell }{ \ell!}  \times \left(\frac{1}{\Gamma(-\ell/2)}\right).
\end{align*}
Now we use the complement formula $\Gamma(s)\Gamma(-s)=-\tfrac{\pi}{s\sin(\pi s)}$ to rewrite the RHS:
$$
\frac{1}{ \Gamma(-\ell/2)}
=\frac{-\sin(\pi \ell/2) \Gamma(\ell/2+1)}{\pi}.
$$
The sine factor vanishes for even $\ell$ so we are left with
\begin{align*}
\frac{1}{2i\pi} \int_{g_b} MH^{a}\frac{dz}{z^{n+1}}
&\sim
-M(\rho)\sigma^{a_n}\frac{\rho^{-n}}{n} 
\sum_{m\geq 0} \frac{(-xh/\sigma)^{2m+1} }{\pi (2m+1)!} (-1)^{m}\Gamma(m+3/2)\\
&\sim
M(\rho)\sigma^{a_n}\frac{\rho^{-n}}{n} 
\sum_{m\geq 0} \frac{(xh/\sigma)^{2m+1}}{ \sqrt{\pi}2^{2m+1}m!}  (-1)^{m}\\
&\sim
M(\rho)\sigma^{a_n}\frac{\rho^{-n}}{n} 
\frac{1}{ \sqrt{\pi}} {\sf Ray}\Big(\frac{xh}{\sigma}\Big).
\end{align*}
We observe that the integral on $g_s$ does not contribute to the asymptotics and we get the announced result.\qedhere
\end{proof}

\section{DH graphs form a subcritical block-stable class}\label{Sec:AppendixDH_are_stable}

We start by recalling the definition of blocks and block-stable graph classes;
we refer to \cite{SubcriticalClasses} for details.
Recall first that the notion of cut-vertices and of $2$-connected graphs have been
defined at the beginning of \cref{Sec:2connectedDH}.
A {\em block} in a graph $G$ is a maximal induced subgraph $B$ of $G$ without cut-vertices (of itself).
A class $\mathcal C$ of graphs is called {\em block-stable} if the following holds:
a graph $G$ is in $\mathcal C$ if and only if all its blocks are in $\mathcal C$.

We now argue that DH graphs form a block-stable class of graphs.
It is known, see \emph{e.g.} \cite[Theorem 10.1]{BLS99},
that DH graphs are the graphs avoiding as induced subgraphs
the following graphs:
the house, the holes, the gem and the domino.
 Since all these graphs are $2$-connected,
 an induced copy of any of them in a graph $G$
 is necessarily included in a single block of $G$.
 Therefore the avoidance of these induced subgraphs 
 can be checked for each block separately, and
the class of DH graphs is indeed block-stable.

Since the class is block-stable, the generating series $D(z)$
of rooted DH graphs or, equivalently of DH trees,
satisfies the equation
\begin{equation}
\label{eq:block_decompo}
D(z)=z \exp(B'(D(z))),
\end{equation}
where $B$ is the generating series of (unrooted) blocks (this is eq.(14) in \cite{SubcriticalClasses}).
It remains to check that the class of DH graphs is subcritical,
\emph{i.e.}~that $D(\rho)$ is smaller than the radius of convergence $\rho_B$ of $B$. (Recall that $D$ and $\rho$ were defined in \cref{ssec:singularity_analysis_NonMarked}.)

For this, we recall that in general, blocks
are either $2$-connected graphs, or restricted to a single vertex or to two vertices 
with a single edge.
Hence the series $B'$ of rooted blocks of DH graphs coincide,
up to the coefficients of $1$ and $z$, with the generating series $\Ddc$
of rooted $2$-connected DH graphs.
In particular, $\rho_B=\rho_{2c}$ and our analysis in \cref{ssec:asymptotics_mainseries_2c} shows that $B''(\rho_B)=+\infty$.
Consequently, there exists $\tau <\rho_B$ such that $\tau B''(\tau)=1$.
This implies that $D(z)$ belongs to the smooth inverse-function schema
in the sense of \cite[Definition VII.3, p.453]{Violet}. From \cite[Theorem VII.2]{Violet},
we have that $D(\rho)=\tau$. Therefore $D(\rho)<\rho_B$ as wanted,
and the class of DH graphs is indeed subcritical.

\section{Gromov--Hausdorff--Prohorov convergence in \cite{SubcriticalClasses}}\label{Sec:AppendixGH-GHP}

The main result of \cite{SubcriticalClasses} is the convergence for the Gromov--Hausdorff topology of a uniform random
graph in a subcritical block-stable class to the Brownian CRT.
We argue here that without further effort, the authors could have proven convergence
for the stronger Gromov--Hausdorff--Prohorov topology.
We use here notation from \cite{SubcriticalClasses}.
The proof compares a uniform random graph $C^\bullet_n$ in the class
and its block decomposition tree $T_n$. 
It uses the fact that the identity map from $T_n$ to $C^\bullet_n$ 
does not modify much distances.
Obviously this identity map brings the uniform distribution on vertices of $T_n$
to that on vertices of $C^\bullet_n$.
Therefore, using \cite[Prop.6, p.763]{miermont2009tessellations},
 we see that $C^\bullet_n$ and $T_n$ are close for the 
GHP topology.
Besides, since $T_n$ has the distribution of a conditioned Galton--Watson tree,
it is known that $T_n$ converges to the Brownian CRT for the 
GHP topology (for the GH topology, a classical reference is \cite{LeGall};
for the GHP topology, a much stronger result is given in \cite{HeWinkel}).
We conclude that $C^\bullet_n$ also 
converges to the Brownian CRT for the GHP topology, as claimed.

\section*{Acknowledgements}
MB has been partially supported by the Swiss National Science Foundation, under grant number 200021-172536.

The authors are grateful to Anita Winter for explanations and bibliographic pointers
on the relation between convergence of distance matrices, convergence for the Gromov--Prohorov distance
and convergence for the Gromov--Hausdorff--Prohorov distance.
We are also grateful to Éric Fusy for suggesting the argument given in \cref{Sec:AppendixDH_are_stable}
to justify that the class of DH graphs is subcritical.

Last but not least, this project started as a collaboration with Mickaël Maazoun
and the authors are indebted to him for his input at the beginning.

\end{document}